\documentclass[11pt]{amsart}
\usepackage{amsthm,amsmath,amssymb,xypic,color,mathtools,xspace}
\usepackage{enumerate}
\usepackage{tikz}
\usepackage{wrapfig}
\usetikzlibrary{calc,matrix,arrows,shapes,decorations.pathmorphing,decorations.markings,decorations.pathreplacing,patterns}

\numberwithin{equation}{section}

\newcommand{\be}{\beta}

\newcommand{\twd}{twisted differential\xspace}
\newcommand{\twds}{twisted differentials\xspace}

\newcommand{\ord}{\operatorname{ord}\nolimits}
\newcommand{\CC}{{\mathbb{C}}}

\newcommand{\PP}{{\mathbb{P}}}

\newcommand{\RR}{{\mathbb{R}}}
\newcommand{\ZZ}{{\mathbb{Z}}}

\newcommand{\VV}{{\mathbb{V}}}
\newcommand{\calO}{{\mathcal O}}

\newcommand{\calA}{{\mathcal A}}
\newcommand{\calB}{{\mathcal B}}

\newcommand{\calM}{{\mathcal M}}

\newcommand{\bM}{{\overline\calM}}

\newcommand{\calE}{{\mathcal E}}

\newcommand{\calX}{{\mathcal X}}
\newcommand{\calY}{{\mathcal Y}}
\newcommand{\calZ}{{\mathcal Z}}
\newcommand{\calD}{{\mathcal D}}

\newcommand{\bfu}{{\boldsymbol{u}}}

\newcommand{\bfzero}{{\boldsymbol{0}}}
\newcommand{\op}{\operatorname}

\newcommand{\SL}{\op{SL}}

\newcommand{\GL}{\op{GL}}

\newcommand{\proj}{{\mathbb P}}

\newcommand\Res{\operatorname{Res}}
\newcommand{\Turn}{\operatorname{Turn}}

\newcommand{\barmoduli}[1][g]{{\overline{\mathcal M}}_{#1}}
\newcommand{\moduli}[1][g]{{\mathcal M}_{#1}}

\newcommand{\omoduli}[1][g]{{\Omega\mathcal M}_{#1}}
\newcommand{\modulin}[1][g,n]{{\mathcal M}_{#1}}
\newcommand{\omodulin}[1][g,n]{{\Omega\mathcal M}_{#1}}

\newcommand{\pomoduli}[1][g]{{\proj\Omega\mathcal M}_{#1}}
\newcommand{\pobarmoduli}[1][g]{{\proj\Omega\overline{\mathcal M}}_{#1}}

\newcommand{\obarmoduli}[1][g]{{\Omega\overline{\mathcal M}}_{#1}}

\newcommand{\DC}{\Delta^{\circ}}

\newcommand{\omoduliinc}[2][g,n]{{\Omega\mathcal M}_{#1}^{{\rm inc}}(#2)}
\newcommand{\obarmoduliinc}[2][g,n]{{\Omega\overline{\mathcal M}}_{#1}^{{\rm inc}}(#2)}
\newcommand{\pobarmoduliinc}[2][g,n]{{\proj\Omega\overline{\mathcal M}}_{#1}^{{\rm inc}}(#2)}
\newcommand{\omoduliincp}[2][g,\lbrace n \rbrace]{{\Omega\mathcal M}_{#1}^{{\rm inc}}(#2)}
\newcommand{\obarmoduliincp}[2][g,\lbrace n \rbrace]{{\Omega\overline{\mathcal M}}_{#1}^{{\rm inc}}(#2)}

\newcommand{\famcurv}{\mathcal{X}} 
\newcommand{\famomega}{\mathcal{W}} 
\newcommand{\seczero}{\mathcal{Z}}

\newcommand{\dualsheave}[1][X]{\omega_{#1}}

\newcommand{\divisor}[1]{{\rm div }\left( #1 \right)}

\newcommand{\banach}[2]{\mathcal{O}(#1)_{#2}}

\newcommand{\permGroup}[1][n]{\mathfrak{S}_{#1}}

\newcommand{\rom}[1]{\textup{\uppercase\expandafter{\romannumeral#1}}}

\theoremstyle{plain}
\newtheorem{thm}{Theorem}[section]
\newtheorem{lm}[thm]{Lemma}
\newtheorem{prop}[thm]{Proposition}
\newtheorem{cor}[thm]{Corollary}

\theoremstyle{definition}
\newtheorem{df}[thm]{Definition}

\newtheorem{rem}[thm]{Remark}
\newtheorem{exa}[thm]{Example}

\def\be{\begin{equation}}   \def\ee{\end{equation}}     \def\bes{\begin{equation*}}    \def\ees{\end{equation*}}
\def\ba{\be\begin{aligned}} \def\ea{\end{aligned}\ee}   \def\bas{\bes\begin{aligned}}  \def\eas{\end{aligned}\ees}
\def\={\;=\;}  \def\+{\,+\,}

\begin{document}
\title[Compactification of strata]{Compactification of strata of abelian differentials}
\author[Bainbridge]{Matt Bainbridge}
\address{Department of Mathematics, Indiana University, Bloomington, IN 47405, USA}
\email{mabainbr@indiana.edu}
\thanks{Research of the first author is supported in part by the Simons Foundation grant \#359821}
\author[Chen]{Dawei Chen}
\address{Department of Mathematics, Boston College, Chestnut Hill, MA 02467, USA}
\email{dawei.chen@bc.edu}
\thanks{Research of the second author is supported in part by the National Science Foundation under the CAREER grant DMS-13-50396 and a Boston College Research Incentive Grant.}
\author[Gendron]{Quentin Gendron}
\address{Institut f\"ur algebraische Geometrie, Leibniz Universit\"at Hannover, Welfengarten 1,
30167 Hannover, Germany}
\email{gendron@math.uni-hannover.de}
\thanks{Research of the third author was supported in part by ERC-StG 257137.}
\author[Grushevsky]{Samuel Grushevsky}
\address{Mathematics Department, Stony Brook University,
Stony Brook, NY 11794-3651, USA}
\email{sam@math.stonybrook.edu}
\thanks{Research of the fourth author is supported in part by the National Science Foundation under the grants DMS-12-01369 and DMS-15-01265, and by a Simons Fellowship in Mathematics (Simons Foundation grant \#341858 to Samuel Grushevsky)}
\author[M\"oller]{Martin M\"oller}
\address{Institut f\"ur Mathematik, Goethe-Universit\"at Frankfurt, Robert-Mayer-Str. 6-8,
60325 Frankfurt am Main, Germany}
\email{moeller@math.uni-frankfurt.de}
\thanks{Research of the fifth author is supported in part by ERC-StG 257137.}

\begin{abstract}
We describe the closure of the strata of abelian differentials with prescribed type of zeros and poles, in the projectivized Hodge bundle over the Deligne-Mumford moduli space of stable curves with marked points. We provide an explicit characterization of pointed stable differentials in the boundary of the closure, both a complex analytic proof and a flat geometric proof for smoothing the boundary differentials, and numerous examples. The main new ingredient in our description is a global residue condition arising from a full order on the dual graph of a stable curve.
\end{abstract}

\date{\today}

\maketitle
\tableofcontents

\section{Introduction}\label{sec:intro}

\subsection{Background}\label{subsec:back}
The Hodge bundle $\omoduli[g]$ is a complex vector bundle of
rank $g$ over the moduli space $\calM_g$ of genus $g$ Riemann surfaces.
A point $(X,\omega)\in\omoduli[g]$ consists of
a Riemann surface $X$ of genus~$g$ and a (holomorphic) abelian differential~$\omega$ on $X$. The
complement of the zero section $\omoduli[g]^*$ is naturally stratified into strata
$\omoduli[g](\mu)$ where the multiplicity of all the zeros of $\omega$ is
prescribed by a partition $\mu = (m_1, \ldots, m_n)$ of $2g-2$.
By scaling the differentials, $\CC^*$ acts on $\omoduli[g]^*$ and preserves the stratification, hence
one can consider the projectivized strata $\proj\omoduli(\mu)$ in the projectivized Hodge bundle $\proj\omoduli[g] = \omoduli[g]^*/\CC^*$.

An abelian differential $\omega$ defines a flat metric with conical singularities such that the underlying
Riemann surface $X$ can be realized as a plane polygon whose edges are pairwise identified via translation. In this sense
$(X, \omega)$ is called a flat surface or a translation surface. Varying the shape of flat surfaces induces a $\GL_2^+(\RR)$-action
on the strata of abelian differentials, called Teichm\"uller dynamics. A number of questions about surface geometry boil down to
understanding the $\GL_2^+(\RR)$-orbit closures in Teichm\"uller dynamics. What are their dimensions? Do they possess
manifold structures? How can one calculate relevant dynamical invariants? From the viewpoint of algebraic geometry the orbit closures are of
an independent interest for cycle class computations, which can provide crucial information for understanding the geometry of moduli spaces.

Many of these questions can be better accessed if one can describe a geometrically meaningful compactification of the strata. In particular,
the recent breakthrough of Eskin, Mirzakhani, Mohammadi \cite{eskinmirzakhani, esmimo} and Filip \cite{filip} shows that any orbit closure
(under the standard topology) is a quasiprojective subvariety of a stratum. Thus describing the projective subvarieties that are closures of orbits in a compactified stratum
can shed further light on the classification of orbit closures.

Identify Riemann surfaces with smooth complex curves. The Deligne-Mumford compactification $\barmoduli$ of $\calM_g$ parameterizes stable genus $g$ curves that are (at worst) nodal curves with finite automorphism groups. The Hodge bundle $\omoduli[g]$ extends as a rank $g$ complex vector bundle $\obarmoduli$ over $\barmoduli$. The fiber of $\obarmoduli$ over a nodal curve $X$ parameterizes stable differentials that have (at worst) simple poles at the nodes of $X$ with opposite residues on the two branches of a node. One way of compactifying $\proj\omoduli(\mu)$ is by taking its closure in the projectivized Hodge bundle $\PP\obarmoduli$ over $\barmoduli$, and we call it the \emph{Hodge bundle compactification} of the strata.

Alternatively, one can lift a stratum $\proj\omoduli(\mu)$ to the moduli space $\calM_{g,n}$ of genus $g$ curves with $n$ marked points by adding on each curve the data of the zeroes
of differentials. Let $\barmoduli[g,n]$ be the Deligne-Mumford compactification of $\calM_{g,n}$ that parameterizes stable genus $g$ curves with $n$ marked points. Taking the closure of $\proj\omoduli(\mu)$ in $\barmoduli[g,n]$ provides another compactification, which we call the \emph{Deligne-Mumford compactification} of the strata.

By combining the two viewpoints above, in this paper we describe
a strata compactification that we call the {\em incidence variety compactification}
$\PP\obarmoduliinc{\mu}$. Let $\PP\obarmoduli[g,n]$ be the projectivized Hodge bundle over
$\barmoduli[g,n]$, which parameterizes \emph{pointed stable differentials}. Then the incidence variety compactification of $\proj\omoduli(\mu)$ is defined as the closure of the stratum in the projectivization $\proj\obarmoduli[g,n]$.

The incidence variety compactification records both the limit stable differentials and the limit positions of the zeros when abelian differentials become identically zero on some irreducible components of the nodal curve. It contains more information than the Hodge bundle compactification, because the latter loses the information
about the limit positions of the zeros on the components of nodal curves where the stable
differentials vanish identically. The incidence variety compactification also contains
more information than the Deligne-Mumford compactification, because the latter loses the information on the relative sizes of flat surfaces corresponding to the components of nodal curves where the stable differentials are not identically zero.

Our characterization of the boundary of the incidence variety compactification is in terms of a collection of (possibly meromorphic) differentials
on the components of a pointed stable curve that satisfy certain combinatorial and residue conditions given by a full order on the vertices of the dual graph of the curve. Meromorphic differentials naturally arise in the description of the boundary objects, and the incidence variety compactification works just as well for the strata of meromorphic differentials, hence we take the meromorphic case into account from the beginning. In order to deal with meromorphic differentials, we consider the closure of the corresponding strata in the Hodge bundle over $\barmoduli[g,n]$ twisted by the polar part $\tilde{\mu}$
of the differentials, which we denote by $K\barmoduli(\tilde{\mu})$ and introduce in Section~\ref{subsec:mero_ivc}.

Before we state the main result, let us first provide some motivation from several viewpoints for the reader to get a feel for the form of the answer that we get.

\subsection{Motivation via complex analytic geometry}\label{subsec:mot-analytic}

Given a pointed stable differential $(X, \omega, z_1,\ldots,z_n)\in\obarmoduli[g,n]$,
that is a stable curve $X$ with marked points $z_1, \ldots, z_n$ at the zeros of
a stable differential $\omega$, the question is
whether it is the limit of a family of abelian differentials $(X_t,\omega_t)$ contained in
a given stratum $\omoduli[g](\mu)$ such that the $z_i$ are the limits of the zeros
of $\omega_t$.
\par
Suppose $f:\calX\to\Delta$ is a family of abelian differentials over a disk $\Delta$ with parameter~$t$, whose underlying curves
degenerate to $X$ at $t=0$. If for an irreducible component $X_v$ of $X$
the limit
$$\omega_0 \,\coloneqq\, \lim_{t\to 0} \omega_t$$
is not identically zero,
then on $X_v$ the limits of zeros of $\omega_t$ are simply the zeros of~$\omega$.
Thus our goal is to extract from this family
a nonzero (possibly meromorphic) differential for every irreducible component of $X$ where $\omega_0$ is identically zero.
The analytic way to do this is to take for every $X_v$ a suitable scaling
parameter $\ell_v\in\ZZ_{\leq 0}$ such that the limit
$$\eta_v\,\coloneqq\,\lim_{t\to 0} t^{\ell_v} \omega_t|_{X_v}$$
is well-defined and not identically zero. This is done in Lemma~\ref{le:scalinglimit}. Along this circle of ideas, we prove our main result by the plumbing techniques in Section~\ref{sec:proof}.

\subsection{Motivation via algebraic geometry}\label{subsec:mot-algebraic}

Now we sketch the algebro-geometric viewpoint of the above setting.
Think of the family $\omega_t$ as a section of the vector bundle $f_* \omega_{\calX^*/\Delta^*}$ of abelian differentials on the fibers over
the punctured disc, where $\omega_{\calX^*/\Delta^*}$ is the relative dualizing line bundle. The Hodge bundle
$f_* \omega_{\calX/\Delta}$ extends
$f_* \omega_{\calX^*/\Delta^*}$ to a vector bundle over the entire disc, but so does any twisting $f_* \omega_{\calX/\Delta}
(\sum c_v X_v)$ by an arbitrary integral linear combination of the irreducible components $X_v$
of the central fiber $X$. Based on the idea of Eisenbud-Harris' limit linear series \cite{eihalimit} (for curves of compact type), we want to choose coefficients
$c_v$ in such a way that $\omega_t$ extends over $t=0$ to a section of the corresponding twisted dualizing line bundle, whose restriction
 $\eta_v$ to every irreducible component $X_v$ is not identically zero (see the discussion in~\cite{chen} for more details). While the machinery of limit linear series for stable curves of arbitrary type is not available in full generality, our Definition~\ref{def:twistedMeroDiff} of twisted differentials works for all stable curves. It is modeled on the collection
of $\eta_v$ defined above.

The next question is to determine which twisted differentials $\eta=\lbrace\eta_v\rbrace$ on a pointed stable curve $X$ arise as actual limits of abelian differentials $(X_t, \omega_t)$ that lie in a given stratum. First, $\eta$ must have suitable zeros at the limit positions of the zeros of $\omega_t$. Moreover if $X$ is reducible, the limit of canonical line bundles of $X_t$ is not unique, as it can be obtained by twisting the dualizing line bundle of $X$ by any component $X_v$ (treated as a divisor in the universal curve). Accordingly the limit $\eta$ of $\omega_t$ on $X$ can be regarded as a section of certain twisted dualizing line bundle, which is not identically zero on any component of $X$ (for otherwise we can twist off such a component). Since the dualizing line bundle of $X$ at a node is generated by differentials with simple poles, after twisting by $\sum c_v X_v$, on one branch of the node the zero or pole order gains $c_v$ and on the other branch it loses $c_v$, hence the zero and pole orders of $\eta$ on the two branches of every node must add up to the original sum of vanishing orders $-2 = (-1) + (-1)$. See~\cite[Section 4.1]{chen} for some examples and more details.
We can then partially orient the dual graph $\Gamma$ of $X$ by orienting the edge
from the zero of $\eta$ to the pole, and leaving it unoriented if the differential has a simple pole at both branches. In this way, we obtain a partial order on the vertices $v$ of $\Gamma$, with equality permitted (see also \cite{fapa}).

It turns out that a partial order is insufficient to characterize actual limits of abelian differentials in a given stratum.
For the degenerating family $(X_t, \omega_t)$, comparing the scaling parameters $\ell_v$
discussed in Section~\ref{subsec:mot-analytic} extends this partial order to a full order on the vertices of $\Gamma$, again with equality permitted.
The final ingredient of our answer is the global residue condition~(4)
in Definition~\ref{def:twistedAbType} that requires the limit twisted differentials~$\eta$ to be compatible with the full order on $\Gamma$.
Simply speaking, this global residue condition arises from applying Stokes' formula to~$\omega_t$ on each level of $\Gamma$ (i.e.,~truncating~$\Gamma$ at vertices that are equal in the full order), for $t\to 0$.

\subsection{Motivation via flat surfaces}\label{subsec:mot-flat}

Degeneration of flat surfaces has been studied in connection
with counting problems, for example in \cite{emz, EMR, ekz}. Most of
the degeneration arguments there rely on a theorem of Rafi~\cite{RafiThTh} on
the comparison of flat and hyperbolic lengths for surfaces near the boundary.
Rafi used a thick-thin decomposition of the flat surfaces by cutting along
hyperbolically short curves. For each piece of the thick-thin decomposition Rafi defined a real number, the size,
in terms of hyperbolic geometry.  His main theorem says that after
rescaling by size, hyperbolic and flat lengths on the thick
pieces are comparable, up to universal constants. Rafi's size is closely related to the scaling parameters $\ell_v$ discussed above, implied by comparing our scaling limit~\eqref{eq:scalinglimit} with the `geometric compactification
theorem' \cite[Theorem~10]{ekz}.
\par
Rafi associated to every closed geodesic a flat representative
of an annular neighborhood. Depending on the curvature of the boundary
such a neighborhood is composed of flat cylinders in the
middle and expanding annuli on both sides, any of the three
possibly not being present. For degenerating families of flat surfaces
this observation can be applied to the vanishing cycles of
the family nearby a nodal fiber. In our result we use the vanishing cycles to read off the global residue condition that constrains degeneration of flat surfaces.
\par
While some of our terminology might be translated into the language of
\cite{RafiThTh} or \cite{ekz}, in the literature on flat surfaces and Teichm\"uller dynamics,
no systematic attempt to describe the set of all possible limit objects under degeneration has been made. In Section~\ref{sec:flatproof} we provide an alternative proof of our main result by constructions of flat surfaces, where a pair of half-infinite cylinders correspond to two simple poles of a twisted differential attached together, and an expanding annulus that appears corresponds to a zero matching a higher order pole.

\subsection{Level graphs}\label{subsec:level}

We now introduce the relevant notions that will allow us to
state our result. Recall that $\Gamma$ denotes the dual graph of a nodal curve $X$, whose vertices and edges correspond to irreducible components and nodes of $X$, respectively. First, we start with the idea of comparing irreducible components of
$X$.

A {\em full order} on the graph $\Gamma$ is a relation $\succcurlyeq$ on
the set $V$ of vertices of $\Gamma$ that is reflexive, transitive,
and such that for any $v_1, v_2 \in V$ at least one of the
statements $v_1 \succcurlyeq v_2$ or $v_2 \succcurlyeq v_1$ holds. We say that $v_1$ is of {\em higher or equal level} compared with $v_2$ if and only if $v_1 \succcurlyeq v_2$. We write $v_1\asymp v_2$ if they are of the same level, that is if both $v_1 \succcurlyeq v_2$ and $v_2 \succcurlyeq v_1$ hold.
We write $v_1\succ v_2$ if $v_1\succcurlyeq v_2$ but $v_2\not\succcurlyeq v_1$, and say that~$v_1$ is of \emph{higher level} than $v_2$.
We call the set of maxima of $V$ the {\em top level} vertices.

\par
We remark that equality is permitted in our definition of a full order.  Any map $\ell:V\to\RR$ assigning real numbers to vertices of $\Gamma$ defines a full order on $\Gamma$ by setting $v_1 \succcurlyeq v_2$ if and only if $\ell(v_1) \geq \ell(v_2)$.
Conversely, every full order can be induced from such a level map, but not from a unique one. Later we will see that
the levels are related to the scaling parameters introduced in Section~\ref{subsec:mot-analytic}, hence it would be convenient to consider the maps $\ell:V\to\RR_{\le 0}$ assigning {\em non-positive} levels only, with $\ell^{-1}(0)\ne\emptyset$ being the top level.
\par
We call a graph $\Gamma$ equipped with a full order on its vertices a {\em level graph}, denoted by~$\overline{\Gamma}$.
We will use the two notions `full order' and `level graph' interchangeably, and draw the level graphs by aligning horizontally vertices of the same level, so that a level map is given by the projection to the vertical axis, and the top level vertices are actually placed at the top (see the examples in~Section~\ref{sec:examples}).

\subsection{Twisted differentials}\label{subsec:twisted}

Throughout the paper we use $\ord_{q} \eta$ to denote the zero or pole order of a differential $\eta$ at $q$, and use $\Res_{q} \eta$ to denote the residue of $\eta$ at $q$. We now introduce the key notion of twisted differentials.

\begin{df} \label{def:twistedMeroDiff}
For a tuple of integers $\mu = (m_1,\ldots,m_n)$,
a {\em twisted differential of type~$\mu$}
on a stable $n$-pointed curve $(X,z_1,\ldots,z_n)$
is a collection of (possibly meromorphic)
differentials $\eta_v$ on the irreducible components $X_v$ of $X$
such that no $\eta_v$ is identically zero and the following properties hold:
\begin{itemize}
\item[(0)] {\bf (Vanishing as prescribed)} Each differential $\eta_v$ is holomorphic and nonzero outside of the nodes and marked points of $X_v$. Moreover, if a marked point $z_i$ lies on~$X_v$, then $\ord_{z_i} \eta_v=m_i$.
\item[(1)] {\bf (Matching orders)} For any node of $X$ that identifies $q_1 \in X_{v_1}$ with $q_2 \in X_{v_2}$,
$$\ord_{q_1} \eta_{v_1}+\ord_{q_2} \eta_{v_2}\=-2. $$
\item[(2)] {\bf (Matching residues at simple poles)}  If at a node of $X$
that identifies $q_1 \in X_{v_1}$ with $q_2 \in X_{v_2}$ the condition $\ord_{q_1}\eta_{v_1}=
\ord_{q_2} \eta_{v_2}=-1$ holds, then $\Res_{q_1}\eta_{v_1}+\Res_{q_2}\eta_{v_2}=0$.
\end{itemize}
\end{df}
\par
These conditions imply that the set of zeros and poles of a twisted
differential consists of the marked points~$z_i$ and some of the nodes of the curve $X$.

\subsection{Twisted differentials compatible with a level graph}\label{subsec:grc}

We want to study under which conditions a twisted differential arises as a limit in a degenerating family of abelian differentials contained in a given stratum. As mentioned before, the conditions depend on a full order on the dual graph $\Gamma$ of $X$, and we need a little more notation.
\par
Suppose that~$\overline\Gamma$ is a level graph with the full order determined by a
level function~$\ell$. For a given level~$L$ we call the subgraph
of $\Gamma$ that consists of all vertices $v$ with $\ell(v) > L$ along with edges between them
the {\em graph above level $L$} of $\Gamma$, and denote it by
$\overline\Gamma_{>L}$. We similarly define the graph $\overline\Gamma_{\geq L}$
{\em above or at level $L$}, and the graph $\overline\Gamma_{=L}$
{\em at level $L$}.
\par
Accordingly we denote by $X_{>L}$ the subcurve of $X$ with dual graph $\overline\Gamma_{>L}$ etc.
For any node $q$ connecting two irreducible
components~$X_v$ and $X_{v'}$, if $v\succ v'$, we denote by $q^+\in X_v$ and $q^-\in X_{v'}$ the two preimages of the node on the normalization of $X$. If $v\asymp v'$, we still write $q_v^\pm$ for the preimages of the node, where the choice of which one is $q_v^+$ is arbitrary and will be specified.
In the same way
we denote by $v^+(q)$ and $v^-(q)$ the vertices at the two ends of an edge representing a node~$q$.
\par
\begin{df}\label{def:twistedAbType}
Let $(X, z_1, \ldots, z_n)$ be an $n$-pointed stable curve with a level graph $\overline\Gamma$.
A twisted differential $\eta$ of type $\mu$ on $X$ is called {\em compatible with~$\overline\Gamma$} if, in addition to conditions~(0), (1), (2) in Definition~\ref{def:twistedMeroDiff}, it also satisfies the following conditions:
\begin{itemize}
\item[(3)]{\bf (Partial order)} If a node of $X$  identifies $q_1 \in X_{v_1}$ with $q_2 \in X_{v_2}$, then $v_1\succcurlyeq  v_2$ if and only if $\ord_{q_1} \eta_{v_1}\ge -1$. Moreover,  $v_1\asymp v_2$ if and only if
$\ord_{q_1} \eta_{v_1} = -1$.
\item[(4)] {\bf (Global residue condition)} For every level $L$
and every connected component~$Y$ of $X_{>L}$ that does not
contain a marked point with a prescribed pole (i.e., there is no $z_i\in Y$ with $m_i<0$) the following condition holds. Let
$q_1,\ldots,q_b$ denote the set of all nodes where~$Y$ intersects $X_{=L}$. Then
$$ \sum_{j=1}^b\Res_{q_j^-}\eta_{v^-(q_j)}\=0,$$
where we recall that $q_j^-\in X_{=L}$ and $v^-(q_j)\in\overline\Gamma_{=L}$.
\end{itemize}
\end{df}
\par
We point out that a given twisted differential satisfying conditions~(0), (1) and (2) may not be compatible with any level graph, or may be compatible with different level graphs with the same underlying dual graph. Condition~(3) is equivalent to saying that, if $v_1$ is of higher level than
$v_2$, then $\eta_{v_1}$ is holomorphic at every node of the intersections of $X_{v_1}$ and $X_{v_2}$, and moreover if $v_1\asymp v_2$,
then $\eta_{v_1}$ and $\eta_{v_2}$ have simple poles at every node where they intersect.

\subsection{Main result}\label{subsec:main}

Recall that the incidence variety compactification of a stratum of abelian differentials $\proj\omoduli(\mu)$ is defined as the closure of the stratum in the projectivized Hodge bundle over $\barmoduli[g,n]$ that parameterizes pointed stable differentials. For a stratum of meromorphic differentials, we take the closure in the twisted Hodge bundle by the polar part of the differentials (see Section~\ref{sec:back_moduli} for details). Our main result characterizes boundary points of the incidence variety compactification for both cases.

\begin{thm} \label{thm:main}
A pointed stable differential $(X,\omega,z_1,\ldots,z_n)$ is contained in the incidence variety compactification of a
stratum $\proj\omoduli(\mu)$ if and only if the following conditions hold:
\begin{itemize}
\item[(i)] There exists a level graph $\overline\Gamma$ on $X$ such that its maxima are the irreducible components $X_v$ of $X$ on which~$\omega$ is not identically zero.
\item[(ii)] There exists a twisted differential~$\eta$ of type $\mu$ on~$X$, compatible with $\overline\Gamma$.
\item[(iii)] On every irreducible component $X_v$ where $\omega$ is not identically zero, $\eta_{v} = \omega|_{X_v}$.
\end{itemize}
\end{thm}

In Sections~\ref{subsec:mot-analytic} and~\ref{subsec:mot-algebraic}, we have briefly explained the ideas behind the necessity of these conditions. The sufficiency part is harder, that is, how can we deform pointed stable differentials satisfying the above conditions into the interior of the stratum? We provide two proofs, in Section~\ref{sec:proof} by using techniques of plumbing in complex analytic geometry, and in Section~\ref{sec:flatproof} by using constructions of flat surfaces.
\par
Recall that the incidence variety compactification combines the two approaches of compactifying the strata in the projectivized Hodge bundle $\proj\obarmoduli$ over $\barmoduli$ and in the Deligne-Mumford space $\barmoduli[g,n]$. In particular, it admits two projections $\pi_1$ to $\PP\obarmoduli$ and $\pi_2$ to $\barmoduli[g,n]$ by forgetting the marked points and forgetting the differentials, respectively. Hence our result completely determines
the strata closures in the Hodge bundle compactification and in the Deligne-Mumford compactification.
\par
\begin{cor}\label{cor:main}
A stable differential $(X, \omega)$ lies in the Hodge bundle compactification of a stratum if and only if there exists a pointed stable differential satisfying the conditions in Theorem~\ref{thm:main} that maps to $(X, \omega)$ via $\pi_1$.

A pointed stable curve $(X, z_1, \ldots, z_n)$ lies in the Deligne-Mumford  compactification of a stratum if and only if there exists a pointed stable differential satisfying the conditions in Theorem~\ref{thm:main} that maps to $(X, z_1, \ldots, z_n)$ via $\pi_2$.
\end{cor}

\begin{rem}\label{rem:finitechoice}
Given a level graph $\overline\Gamma$, there can only exist finitely many twisted differentials $\eta$ compatible with $\overline\Gamma$ and satisfying the conditions of the theorem, up to scaling~$\eta_v$ on each irreducible component by a nonzero number. Indeed, if all the pole and zero orders of $\eta$ at all nodes are given, it determines $\eta$ uniquely up to scaling on each irreducible component. For any irreducible component $X_v$ on the bottom level, the zeros and poles of~$\eta_v$ are prescribed outside of the nodes of $X_v$, and at those nodes~$\eta_v$ only has poles. As the total number of zeros and poles of $\eta_v$, counted with multiplicity, is equal to $2g_v-2$, it implies that the sum of the orders of poles of~$\eta_v$ at all nodes of $X_v$ is fixed, and hence there are finitely many choices. For each such choice, on every irreducible component intersecting $X_v$ the order of zero of $\eta$ at any node where it intersects $X_v$ is thus uniquely determined. If another component on the bottom level intersects $X_v$, then $\eta$ has a simple pole at both branches of that node. Therefore, we can prove that there are finitely many choices of $\eta$ up to scaling, by induction on the number of irreducible components of $X$.
\end{rem}

For later use we relate a twisted differential $\eta$ and a pointed stable differential $(X, \omega, z_1, \ldots, z_n)$ satisfying
Theorem~\ref{thm:main} as follows.

\begin{df}\label{def:associated-diff}
Given a twisted differential~$\eta$ compatible with a level graph on $X$, define the {\em associated pointed stable differential} $(X,\omega, z_1,\ldots, z_n)$ by taking $\omega$ to be equal to $\eta$ on all top level components and identically zero on components of lower levels, and by taking $z_1, \ldots, z_n$ to be the set of zeros and poles of $\eta$ away from the nodes of~$X$.

Conversely given a pointed stable differential $(X,\omega, z_1,\ldots, z_n)$, if it is associated to a twisted differential $\eta$ compatible with certain level graph on $X$, we say that $\eta$ is an {\em associated twisted differential} of $(X,\omega, z_1,\ldots, z_n)$.
\end{df}

Using the above definition, we can restate Theorem~\ref{thm:main} as that a pointed stable differential
$(X, \omega, z_1, \ldots, z_n)$ is contained in the incidence variety compactification of $\proj\omoduli(\mu)$ if and only if there exists a level graph and a compatible twisted differential~$\eta$ of type $\mu$ such that $(X, \omega, z_1, \ldots, z_n)$ is associated with $\eta$.

Note that multiplying $\eta$ by any nonzero numbers on components not of top level does not change the associated pointed stable differential. Hence given a pointed stable differential, the associated twisted differential~$\eta$ may not be unique. Indeed it is not unique even up to scaling on each irreducible component (see Example~\ref{exa:etanotunique}).

\subsection{History of the project and related work} \label{subsec:related}

Recently there have been several attempts via different viewpoints that aim at understanding the boundary behavior of the strata of abelian differentials. In a talk given in August 2008, Kontsevich discussed the problem of compactifying the strata, focusing on the matching order and matching residue conditions. In \cite{gendron} the third author studied the incidence variety compactification and applied the plumbing techniques to prove a special case of our main result when all the residues of a twisted differential are zero. In that case the global residue condition~(4) obviously holds. Motivated by the theory of limit linear series, in \cite{chen} the second author studied the Deligne-Mumford strata compactification and deduced the necessity of the conditions~(0), (1), and (2). He also obtained partial smoothing results in the case of curves of pseudo-compact type by combining techniques of algebraic geometry and flat geometry. In \cite{fapa} Farkas and Pandharipande studied the Deligne-Mumford strata compactification by imposing the conditions~(0), (1), (2), and (3), i.e., without the global residue condition~(4). It turns out the corresponding loci are reducible in general, containing extra components of equal dimension or one less in the boundary of $\barmoduli[g,n]$. Modulo a conjectural relation to Pixton's formula of the double ramification cycle, in \cite[Appendix]{fapa} Janda, Pandharipande, Pixton, and Zvonkine used the extra components to recursively compute the cycle classes of the strata in $\barmoduli[g,n]$. In \cite{miwr} Mirzakhani and Wright concentrated on a collapse of the Hodge bundle compactification by keeping track only of components where the stable differentials are not identically zero. They proved an identification between the tangent space of the boundary of a $\GL_2^+(\RR)$-orbit closure and the intersection of the tangent space to the orbit closure with the tangent space to the boundary of their compactification. In \cite{grkrno} Krichever, Norton, and the fourth author studied
degenerations of meromorphic differentials with all periods real, where plumbing techniques are also used and a full order on the dual graph also arises.

In Summer 2015 the authors of the current paper met in various combinations on several occasions, including in Bonn, Luminy, Salt Lake City, and Boston. After stimulating discussions, the crucial global residue condition and the proof of sufficiency emerged, finally completing the characterization of the compactification of the strata.

\subsection{Applications}\label{subsec:applications}
The main novel aspect of our determination of the closures of the strata is the global residue condition, which is used to characterize exactly those stable differentials that appear in the closure of a stratum in the Hodge bundle, and not to extraneous components. Thus any further work aimed at understanding the structure of strata compactifications must build on our description in an essential way. In particular, the global residue condition was a cornerstone in the work of Sauvaget~\cite{Sauvaget} who analyzed the boundary of the strata in the Hodge bundle in order to understand the homology classes of the strata closures, and in the work of Mullane~\cite{Mullane}, who used the global residue condition to analyze certain divisor closures in $\barmoduli[g,n]$, and discovered an infinite series of new extremal effective divisors. Furthermore, in~\cite{ChenChen} Qile Chen and the second author described algebraically the principal boundary of the strata in terms of twisted differentials, solving a problem that had been open for flat surfaces for more than a decade, while Ulirsch, Werner, and the fifth author applied in \cite{MUW}
our compactification to solve the realizability problem for constructing the tropical Hodge bundle, posed in~\cite{Tropical}. Finally in Section~\ref{subsec:weierstrass}, as a consequence of our result we provide an efficient description for the degeneration of Weierstra\ss{} divisors on certain binary curve, recovering a main case in the work of \cite{esme}.

\subsection{Organization of the paper}\label{subsec:organization}
In Section~\ref{sec:back_moduli} we review basic properties about moduli spaces of curves and abelian differentials.
We further compare the incidence variety compactification to the Hodge bundle compactification and the Deligne-Mumford compactification, which illustrates what extra information the incidence variety compactification gains. In Section~\ref{sec:examples} we apply Theorem~\ref{thm:main} to analyze explicitly a number of examples that characterize the significance and delicacy of the global residue condition. In Section~\ref{sec:proof} we prove Theorem~\ref{thm:main} by the method of plumbing. Finally in Section~\ref{sec:flatproof} we provide an alternative proof by constructions of flat surfaces. In particular, Sections~\ref{sec:examples}, ~\ref{sec:proof} and~\ref{sec:flatproof} are independent. Depending on the reader's background and interests, they can be read in any order.

\subsection{Acknowledgments}

We are grateful to Curt McMullen for useful conversations on the proof of the main theorem, to Gavril Farkas and Rahul Pandharipande for communications with us at the Arbeitstagung at Max Planck Institute in Bonn, June 2015, and to Adrian Sauvaget and Dimitri Zvonkine for discussions at the Flat Surfaces workshop in Luminy, July 2015. We thank all the organizers and participants of the Flat Surfaces workshop for their interests in this problem and motivating us to work on the proof from the viewpoint of flat geometry. The second author thanks Alex Eskin, Joe Harris, and Anton Zorich for encouraging him over the years to study Teichm\"uller dynamics via algebraic geometry. The fourth author thanks Igor Krichever and Chaya Norton, with whom he has been developing related ideas and techniques for real-normalized meromorphic differentials. Finally we thank the anonymous referees for carefully reading the paper and many useful comments.

\section{Comparison of strata compactifications}\label{sec:back_moduli}
In this section we introduce the basic terminology
about moduli spaces of curves and strata of abelian differentials.
We also define the incidence variety
compactification and compare it to the Hodge bundle and the Deligne-Mumford compactifications.

\subsection{Moduli spaces of curves}\label{subsec:back_moduli}
Denote by $\moduli$ the  {\em moduli space of curves of genus $g$} that parameterizes
smooth and connected complex curves of genus $g$, up to biholomorphism.
Denote by $\modulin$ the {\em moduli space of $n$-pointed genus $g$ curves} that parameterizes
smooth and connected complex curves of genus $g$ together with $n$ distinct (ordered) marked points.
The space $\modulin$ is a complex orbifold of dimension $3g-3+n$.

Recall that a {\em stable $n$-pointed curve} is a connected curve with at worst nodal singularities,
with $n$ distinct marked smooth points, such that the automorphism group of the curve preserving the marked points is finite.
Denote by $\barmoduli[g,n]$ the {\em Deligne-Mumford compactification}
of $\barmoduli[g,n]$ parameterizing stable $n$-pointed genus $g$ curves.
\par
Let $\mathfrak{S}$ be a subgroup of the symmetric group $\mathfrak{S}_{n}$. Then  $\mathfrak{S}$ acts on $\barmoduli[g,n]$ by permuting the marked points. The quotient of $\barmoduli[g,n]$ by $\mathfrak{S}$ is denoted by $\barmoduli[g,\left\{n\right\}]^{\mathfrak{S}}$ or simply $\barmoduli[g,\left\{n\right\}]$ when the group $\mathfrak{S}$ is clear from the context.
\par
The {\em dual graph} of a nodal curve $X$ is the graph~$\Gamma$ whose vertices
correspond to the irreducible components of~$X$. For every node of~$X$ joining
two components $v_1$ and $v_2$ (possibly being the same component) the dual
graph has an edge connecting $v_1$ and $v_2$.
For every marked point there is a leg or equivalently a half-edge attached to
the vertex corresponding to the irreducible component that contains the
marked point.
\par
For a connected nodal curve $X$, if removing a node disconnects $X$, we say that it is a \emph{separating} node. Otherwise we call it
a \emph{non-separating} node. If the two branches of a node belong to the same irreducible component of $X$, we say that it is an
\emph{internal} node.

\subsection{Moduli spaces of abelian differentials}\label{subsec:back_abelian}

The {\em moduli space of abelian differentials $\omoduli^*$}
is the complement of the zero section in the {\em Hodge bundle} $\omoduli\to\moduli$,
which parameterizes pairs $(X, \omega)$ where $X$ is a
smooth and connected curve of genus~$g$, and~$\omega$ is a holomorphic differential on~$X$.
The space $\omoduli^*$ has a natural stratification according to the
orders of zeros of $\omega$.
Let $\mu = (m_1,\ldots, m_n)$ be a partition of $2g-2$ by positive integers. The
\emph{stratum $\omoduli(\mu)$ of abelian differentials of type $\mu$}, as a subspace of
$\omoduli^*$, parameterizes abelian differentials $(X,\omega)$ such that
$\omega$ has $n$ distinct zeros of orders $m_1,\ldots,m_n$, respectively.

A {\em stable differential} on a nodal curve
$X$ is a (possibly meromorphic) differential~$\omega$ on~$X$ which is holomorphic outside
of the nodes of~$X$ and which has at worst simple poles at the
nodes, with opposite residues. The Hodge bundle
extends to a vector bundle $\obarmoduli[g,n]\to\barmoduli[g,n]$,
the total space of the relative dualizing sheaf of the universal family
$f: \calX \to \barmoduli[g,n]$. The fiber of this vector bundle over
a pointed nodal
curve~$(X,z_1,\ldots,z_n)$  parameterizes {\em pointed stable differentials}
$(X, \omega, z_1, \ldots, z_n)$, where  $\omega$ is a stable differential
on~$(X,z_1,\ldots,z_n)$.
\par
There is a natural $\CC^{\ast}$-action on the Hodge bundle by scaling the differentials. This action
preserves the stratification of $\omoduli^*$. The quotient of $\omoduli^*$
under this action is denoted by $\pomoduli$. In general, quotient spaces
by such a $\CC^{\ast}$-action will be denoted by adding the letter $\proj$.

\subsection{The incidence variety compactification}\label{subsec:ivc}

When abelian differentials degenerate, we want to keep track of the information about both the limit stable differentials and the limit positions of the marked zeros. This motivates the definition of the incidence variety.

For a partition $\mu = (m_1, \ldots, m_n)$ of $2g-2$, the {\em (ordered) incidence variety}
$\PP\omoduliinc{\mu}$ is defined to be
\begin{equation}
\PP\omoduliinc{\mu} \,\coloneqq\, \Bigl\{(X, \omega, z_{1},\ldots,z_{n}) \in
 \PP \omodulin:\
\divisor{\omega} \= \sum_{i=1}^{n} m_{i}z_{i}  \Bigr\}.
\end{equation}
The {\em (ordered) incidence variety compactification}
$\PP\obarmoduliinc{\mu}$ is defined to be the closure
of the incidence variety inside $\PP \obarmoduli[g,n]$.

If  $m_i=m_j$ for some $i\neq j$, then interchanging $z_i$ and $z_j$ preserves
the incidence variety. Hence the subgroup $\permGroup[]$ of the symmetric group
$\permGroup[n]$ generated by such transpositions acts on $\PP\obarmoduliinc{\mu}$ by permuting
 $z_1,\ldots,z_n$ accordingly.
We define the {\em (unordered) incidence variety} to be the quotient
\begin{equation}
\PP\omoduliincp{\mu} \,\coloneqq\, \PP\omoduliinc{\mu}
/\mathfrak{S}
\end{equation}
and define the {\em (unordered) incidence variety compactification} to be the quotient
\begin{equation}
\PP\obarmoduliincp{\mu} \,\coloneqq\, \PP\obarmoduliinc{\mu}
/\mathfrak{S}.
\end{equation}

Since the ordered and unordered incidence variety compactifications differ only by permuting the marked points, we
call both of them the incidence variety compactification. In case we need to distinguish them, we do so by specifying $n$ or
$\{ n\}$ in the subscripts.

On a smooth curve $X$, a non-zero differential $\omega$ determines its zeros $z_i$ along with their multiplicities. Hence
$\PP\omoduliincp{\mu}$ is isomorphic to the projectivized stratum of abelian differentials $\PP\omoduli[g](\mu)$. In order to understand degeneration of abelian differentials in $\PP\omoduli[g](\mu)$, we need to describe the boundary points of
the incidence variety compactification.

\subsection{Moduli spaces of pointed meromorphic differentials}\label{subsec:mero}

The moduli spaces of meromorphic differentials from the viewpoint of flat geometry have been investigated by Boissy \cite{boissymero}, including their dimensions and connected components. Let
$$\mu \= (m_1,\ldots, m_r; m_{r+1}, \ldots, m_{r+s};m_{r+s+1},\ldots,m_{r+s+l})$$
be an $n$-tuple of integers such that $\sum_{i=1}^{n} m_i = 2g-2$, where $m_i>0$ for $i \leq r$, $m_{r+1}=\cdots=m_{r+s}=0$, and $m_i <0$ for $i>r+s$. We call such $\mu$ a meromorphic type, and denote by
$\omoduli(\mu)$ the {\em moduli space of meromorphic
differentials of type~$\mu$}. It parameterizes $n$-pointed meromorphic differentials $(X,\omega,z_{1},\cdots,z_{n})$ on
a smooth curve $X$ such that the order of $\omega$ at $z_i$ is equal to $m_i$, which may be a zero,
regular point, or pole, corresponding to whether $m_i>0$, $m_i=0$, or $m_i<0$, respectively.
As there are infinitely many meromorphic types for a fixed genus, these moduli
spaces no longer form a stratification of a fixed ambient space, but by a slight abuse
of language we still call them {\em strata of meromorphic differentials}.

\subsection{The incidence variety compactification in the meromorphic case}\label{subsec:mero_ivc}

To mimic the definition in the abelian case we need to generalize the notion of the Hodge bundle. We denote the polar part of $\mu$
by $\tilde\mu=(m_{r+s+1},\ldots,m_n)$. We then define the {\em pointed Hodge
bundle twisted by $\tilde{\mu}$} to be the bundle
$$K\barmoduli[g,n](\tilde{\mu}) \= f_* \omega_{\calX/\barmoduli[g,n]}
\Bigl(-\sum_{i=r+s+1}^n m_i\seczero_i\Bigr)$$
over $\barmoduli[g,n]$, where we have denoted by $\seczero_i$ the image
of the section of the universal family $f$ given by the $i$'th marked point.
We call points $(X,\omega,z_{1},\ldots,z_{n})\in K\barmoduli[g,n](\tilde{\mu})$
{\em pointed stable differentials (of type~$\tilde{\mu}$)}.
Note that if $\mu$ is a holomorphic type, then the tuple $\tilde\mu$ is empty, so $K\barmoduli(\tilde\mu)=\obarmoduli$, and hence it recovers the preceding setting for the abelian case.
\par
For any meromorphic type $\mu$, we perform the same operations inside the space $K\barmoduli(\tilde\mu)$
as in the abelian case. The {\em (ordered) incidence variety} $\PP\omoduliinc{m_{1},\ldots,m_{n}}$ is defined as
\begin{equation}
\PP\omoduliinc{\mu} \=
\Bigl\{(X, \omega, z_{1},\ldots,z_{n}) \in \PP K\modulin(\tilde{\mu}):\
\divisor{\omega} \= \sum_{i=1}^{n} m_{i}z_{i}  \Bigr\},
\end{equation}
and the {\em (ordered) incidence variety compactification} $\PP\obarmoduliinc{\mu}$ is defined to be its
closure in $\PP K\barmoduli[g,n](\tilde{\mu})$.
We define the
{\em (unordered) incidence variety} to be the quotient
\begin{equation}
\PP\omoduliincp{\mu}\=\PP\omoduliinc{\mu}/\mathfrak{S},
\end{equation}
where $\mathfrak{S}$ is defined as in Section~\ref{subsec:ivc}. Finally, we define the
{\em (unordered) incidence variety compactification}
to be the quotient
\begin{equation}
\PP\obarmoduliincp{\mu} \= \PP\obarmoduliinc{\mu}/\mathfrak{S}.
\end{equation}

Again, by a slight abuse of language we sometimes skip the terms `ordered' and `unordered', and refer to both of them
as the incidence variety compactification of the strata of meromorphic differentials.

\subsection{Comparison to the Hodge bundle and the Deligne-Mumford compactifications}\label{subsec:comparison}

As said earlier, the incidence variety
compactification admits two forgetful maps to the projectivized Hodge bundle and to the Deligne-Mumford space, respectively. In this subsection, we discuss these maps and demonstrate that they indeed forget information. For simplicity of notation, we will only state this in the holomorphic case. The discussion can
be easily generalized to the meromorphic case. Thus we start with a
partition $\mu = (m_{1},\ldots,m_{n})$ of $2g-2$ by positive integers.
\par
The forgetful map
\begin{equation}\label{eq:forget1}
\pi_{1}:\PP\obarmoduliincp{\mu}\to\PP\obarmoduli(\mu)
\end{equation}
forgets the marked points $z_1, \ldots, z_n$. More precisely, the image of $(X,\omega,z_{1},\ldots,z_{n})$ under $\pi_{1}$ is
$(X',\omega')$, where $X'$ is obtained from $X$ (as an unmarked curve) by blowing down all $\PP^1$ tails and bridges, and $\omega'$ can be identified with the restriction of $\omega$ to the remaining components (see \cite[Lemma~2.4]{gendron} for details). The other  forgetful map
\begin{equation}\label{eq:forget2}
 \pi_{2}:\PP\obarmoduliincp{\mu}\to
\barmoduli[g,\left\{n\right\}]
\end{equation}
forgets the stable differential $\omega$, hence the image of $(X,\omega,z_{1},\ldots,z_{n})$ under $\pi_{2}$ is
just $(X, z_1, \ldots, z_n)$.

Since a differential on a compact curve is determined uniquely (up to scaling) by the locations and orders of its zeros, both maps $\pi_1$ and $\pi_2$, when restricted to $\PP\omoduliincp{\mu}$, are isomorphisms onto their respective images. However, over the boundary the fibers of both maps can be more complicated. In particular, they may no longer be finite, and neither image dominates the other.

\begin{prop} \label{prop:relationClosures}
For $g\geq3$ and $n\geq2$, the following properties hold:
\begin{itemize}
 \item[(i)] The map $\pi_1$ is not finite, and there does not exist a map
$$f:\PP\obarmoduli(\mu)\to\barmoduli[g,\left\{n\right\}]$$
such that $f\circ\pi_{1}=\pi_{2}$.
 \item[(ii)]  The map $\pi_2$ is not finite, and there does not exist a map
$$h:\pi_{2}\Bigl(\PP\obarmoduliincp{\mu}\Bigr)\to\PP\obarmoduli(\mu)$$
such that $h\circ\pi_{2}=\pi_{1}$.
\end{itemize}
\end{prop}
This proposition is not a priori clear. Indeed it uses the full strength of Theorem~\ref{thm:main} about characterizing the boundary points of $\PP\obarmoduliincp{\mu}$. As a result we see that the incidence variety compactification contains more information than both the Hodge bundle compactification and the Deligne-Mumford compactification of the strata.
\begin{proof}
We first prove (i). Suppose $X$ is the union of an elliptic curve $X_{1}$ with a curve~$X_{g-1}$ of genus $g-1$ intersecting at a node $q^{+}\sim q^{-}$ such that $(2g-4)q^{+}\sim K_{X_{g-1}}$ and that $X_1$ contains all the marked points
$z_1, \ldots, z_n$. Put $X_{g-1}$ on a higher level than~$X_1$. The corresponding level graph $\overline{\Gamma}$ of $X$ is represented on the left side of Figure~\ref{cap:relationClosure}. Take a stable differential $\omega$ on $X$ such that
$\omega|_{X_{g-1}}$ is a holomorphic differential with a unique zero at $q^{+}$ of multiplicity $2g-4$, and $\omega|_{X_1}$ is identically zero. Take a twisted differential $\eta$ on $X$ such that $\eta_{X_{g-1}} = \omega|_{X_{g-1}}$ and such that $\eta_{X_1}$ is a meromorphic differential with $\divisor{\eta_{X_1}} = \sum_{i=1}^n m_i z_i - (2g-2) q^{-}$.

One checks that $\eta$ satisfies all the conditions in Definitions~\ref{def:twistedMeroDiff} and~\ref{def:twistedAbType}.
In particular, the global residue condition follows from the residue theorem on $X_1$, because $\eta_{X_1}$ has a unique pole at $q^{-}$, and hence $\Res_{q^{-}} \eta_{X_1} = 0$. Since $\eta$ is compatible with $\overline{\Gamma}$ and $(X, \omega, z_1, \ldots, z_n)$
is the associated pointed stable differential of $\eta$ (see Definition~\ref{def:associated-diff}), by Theorem~\ref{thm:main}
$(X, \omega, z_1,\ldots, z_n)$ is contained in the incidence variety compactification $\PP\obarmoduliincp{\mu}$.
It implies that the $\pi_1$-preimage of the stable differential $(X, \omega)$
is isomorphic to
$$\Bigl\{(z_{1},\ldots,z_{n})\in (X_{1})^{n}\setminus\Delta:\sum_{i=1}^n m_{i}z_{i} \= (2g-2)q^{-}\Bigr\} \, / \mathfrak{S}\,,$$
where $\Delta$ is the big diagonal parameterizing the tuples where at least two marked points
coincide. This preimage has dimension $n-1$. Since the $\pi_2$-image retains the information about the positions
of the marked points, this example implies that there does not exist a map $f$ such that $f\circ\pi_{1}=\pi_{2}$.
\par
\begin{figure}[h]
\begin{tikzpicture}[scale=1,decoration={
    markings,
    mark=at position 0.5 with {\arrow[very thick]{>}}}]
\fill (0,0) coordinate (y1) circle (2pt);  \node [above] at (y1) {$X_{g-1}$};\node [below right] at (y1) {$q^{+}$};
\fill (0,-1.5) coordinate (y2) circle (2pt); \node [ below] at (y2) {$X_{1}$};\node [above right] at (y2) {$q^{-}$};

 \draw[postaction={decorate}] (y1) -- (y2);

\fill (5,0) coordinate (x3) circle (2pt);\node [above] at (x3) {$X_{g-2}$};\node [below ] at (x3) {$q_{2}^{+}$};
\fill (3,0) coordinate (x1) circle (2pt); \node [above] at (x1) {$X_{1}$};\node [below ] at (x1) {$q_{1}^{+}$};
\fill (4,-1.5) coordinate (x2) circle (2pt);  \node [ below] at (x2) {$X_{2}$};\node [ left] at (x2) {$q_{1}^{-}$};\node [ right] at (x2) {$q_{2}^{-}$};

 \draw[postaction={decorate}] (x3) .. controls (4.1,-.7) .. (x2);
 \draw[postaction={decorate}] (x1) .. controls (3.9,-.7) ..  (x2);
\end{tikzpicture}
\caption{The level graphs used in the proof of Proposition~\ref{prop:relationClosures}}\label{cap:relationClosure}
\end{figure}
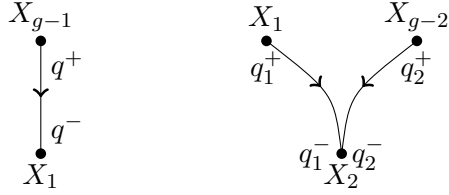
\par
Next we prove (ii). Let $X$ be the union of two elliptic curves $X_{1}$, $X_{2}$ and a curve $X_{g-2}$ of genus $g-2$, whose dual graph with a chosen full order $\overline{\Gamma}$ is represented on the right side of Figure~\ref{cap:relationClosure}. In this level graph $X_1$ and $X_{g-2}$ are on the top level, both higher than $X_2$. Further suppose that $(2g-6)q_{2}^{+}\sim K_{X_{g-2}}$ and that the points $z_{1},\ldots,z_{n}\in X_{2}$ are chosen such that there exists a meromorphic differential $\eta_{X_2}$ on $X_2$ with
$ {\rm div}(\eta_{X_2}) = \sum_{i=1}^n m_i z_i - 2q_{1}^{-} - (2g-4) q_{2}^{-}$
and such that $\Res_{q_1^-}\eta_{X_2} = \Res_{q_2^-}\eta_{X_2} = 0$. The existence of such $\eta_{X_2}$ is proved in \cite{boissy}.
Then $(X, z_1, \ldots, z_n)$ specifies a point in $\barmoduli[g,\left\{n\right\}]$.
Take a twisted differential $\eta$ on $X$ such that $\eta_{X_1}$ on $X_1$ is holomorphic and nowhere vanishing,
$\eta_{X_2}$ is given as above on $X_2$, and $\eta_{X_{2g-2}}$ on $X_{2g-2}$ satisfies ${\rm div}(\eta_{X_{2g-2}}) = (2g-6)q_{2}^{+}$.

One checks that $\eta$ satisfies all the required conditions to be compatible with $\overline{\Gamma}$. In particular, the global residue condition holds because $\eta_{X_2}$ has zero residues at both $q_1^-$ and~$q_2^-$. Hence
by Theorem~\ref{thm:main} the associated pointed stable differential
$(X, \omega, z_1,\ldots, z_n)$ is contained in the incidence variety compactification $\PP\obarmoduliincp{\mu}$,
where $\omega|_{X_1} = \eta_{X_1}$, $\omega|_{X_{g-2}} = \eta_{X_{g-2}}$, and $\omega$ is identically zero on $X_2$.
Note that one can scale $\eta$ on the top level components $X_{1}$ and $X_{g-2}$
by a pair of nonzero scalars $(\lambda_1, \lambda_{g-2})$, which does not affect its compatibility with $\overline{\Gamma}$. The associated stable differential $\omega$ is scaled accordingly on $X_{1}$ and $X_{g-2}$, but the underlying marked curve
$(X, z_1, \ldots, z_n)$ remains the same. In other words, the $\pi_2$-preimage of $(X,z_{1},\ldots,z_{n})$ in the incidence variety compactification contains the space of the projectivized pairs $[\lambda_1, \lambda_{g-2}]$, which is one-dimensional. Since the image of $(X, \omega, z_1,\ldots, z_n)$ under $\pi_1$ retains the scaling information, there does not exist a map $h$ such that $h\circ\pi_{2}=\pi_{1}$ holds.
\end{proof}

\begin{rem} \label{rem:noninj}
Conceptually speaking, the map $\pi_{2}$ fails to be injective for two reasons. First, for a component of $X$, there may exist a nontrivial linear equivalence relation between the marked points in that component. For example, let $X$ be the union of a hyperelliptic curve $Y$ of genus $g-1$, and a $\PP^1$ component, intersecting at two points $q_1$ and $q_2$ which are Weierstra\ss{} points of $Y$. Moreover suppose that all the marked points are contained in $\PP^1$.  Then $(X, z_1, \ldots, z_n)$
can be the image of $(X,\omega,z_1, \ldots, z_n)$ under $\pi_2$, where $\omega$ is identically zero on $\PP^{1}$ and restricts to $Y$ as a differential with a zero of order $2k$ at $q_{1}$ and a zero of order $2(g-k-2)$ at $q_{2}$, for any $k\in\lbrace 0, \ldots, g-2\rbrace$. The other reason is that some scaling factors for the differentials on the top level components are lost, as discussed in the second part of the proof of Proposition~\ref{prop:relationClosures} (also see Lemma~\ref{lm:dimFibrePiDeux} below).
\end{rem}

When we analyze the examples presented in Section~\ref{sec:examples}, the information of the dimension of fibers of $\pi_{2}$
will play a significant role. Hence we conclude this section by the following observation.
\begin{lm}\label{lm:dimFibrePiDeux}
Let $\DC$ be an open boundary stratum of $\barmoduli[g,n]$ parameterizing nodal curves with a given dual graph.
Let $(X,z_{1},\ldots,z_{n})$ be a curve in the intersection of the locus $\pi_{2}(\PP\obarmoduliinc{\mu})$ with $\DC$. Then the dimension of the fiber of $\pi_{2}$ over $(X,z_{1},\ldots,z_{n})$ is one less than the maximal number of connected components of the graph $\Gamma_{=0}$, where the maximum is taken over all level graph structures $\overline\Gamma$ on $\Gamma$,  such that there exists a compatible twisted differential of type $\mu$ on $(X,z_{1},\ldots,z_{n})$.
\end{lm}
\begin{proof}
Take a twisted differential $\eta$ on $X$ compatible with a chosen level graph $\overline{\Gamma}$. Suppose $(X, \omega, z_1, \ldots, z_n)$ is the pointed stable differential associated to $(\eta,\overline{\Gamma})$. Then~$\omega$ is equal to $\eta$ restricted to all top level components, and is identically zero elsewhere. The image of $(X, \omega, z_1, \ldots, z_n)$ under $\pi_2$ then only retains the information on the zeroes of $\omega$ on all the irreducible components of top level, and thus loses the information of the individual scaling factors
of $\eta$ on each such irreducible component. However, if two such top level irreducible components are connected by an edge, the matching residue condition at the corresponding node prescribes that the scale of $\eta$ on these two components is equal, and thus there is only one scale parameter lost for each connected component of the graph of the top level components. The space of such scaling factors has projective dimension equal to
the number of top level components minus one. The desired conclusion thus follows from applying this analysis to all possible level graphs and compatible twisted differentials.
\end{proof}

\section{Examples of the incidence variety compactification }\label{sec:examples}
To illustrate all the aspects of the incidence variety compactification, in
this section we study many examples. We start with examples discussing possible choices of level graphs and
compatible twisted differentials for a fixed dual graph, thus showing that there are indeed choices involved, and that all our data are necessary. We then describe in detail the incidence variety compactification for a number of strata in low genus.

Throughout this section, we use the following notation. Denote by $K_{X_i}$ the canonical line bundle of an irreducible component $X_i$ of a nodal curve $X$. A node joining two irreducible components $X_{i}$ and~$X_{j}$ is denoted by $q_{k}$. Moreover, if $X_{i}\succ X_{j}$, then the node $q_{k}$ is obtained by identifying the points $q_{k}^{+}\in X_{i}$ with  $q_{k}^{-}\in X_{j}$. We denote by $\Gamma$ the dual graph of $X$, and by $\overline{\Gamma}$ a full order on $\Gamma$. For a twisted differential $\eta$ and a stable differential~$\omega$ on~$X$, we use $\eta_i$ and $\omega_i$ to denote their restrictions to the component $X_i$, respectively. We also remind the reader to review Definition~\ref{def:associated-diff} for $\eta$ and $\omega$ being associated with each other.

\subsection{Cautionary examples} \label{sec:cautionary}
We present some examples that serve as an illustration for
our formulation of the conditions on twisted differentials
as well as a warning regarding the extent to which the choices of~$\overline\Gamma$ and $\eta$
determine each other.
\begin{exa}({\em Twisted differentials do not automatically satisfy the global residue condition})\label{exa:GRCneeded}
  Let $X$ be a curve with three components $X_{1}\asymp X_{2}\succ X_{3}$ as represented in Figure~\ref{cap:GRCneeded}. Suppose that $X_1$ and $X_2$ contain no marked poles, and suppose that $g_3=0$. Suppose that the only marked point
  on~$X_3$ is a marked zero~$z_1$, i.e.~$m_1>0$.
\begin{figure}[ht]
\begin{tikzpicture}[scale=1,decoration={
    markings,
    mark=at position 0.5 with {\arrow[very thick]{>}}}]
\draw (-4,0.1) coordinate (x1) .. controls (-3.8,-.3) and (-4.2,-.6) .. (-4,-1.1) coordinate (q1) coordinate [pos=.2] (z1) coordinate [pos=.35] (z2) coordinate [pos=.65] (z3);
\fill (z2) circle (1pt);\fill (z3) circle (1pt);\fill (z1) circle (1pt);
\draw (-2,0.1)  coordinate (x2).. controls (-1.8,-.3) and (-2.2,-.6) .. (-2,-1.1) coordinate[pos=.7] (q2) coordinate [pos=.65] (z1) coordinate [pos=.35] (z2) ;
\fill (z2) circle (1pt);\fill (z1) circle (1pt);

\draw (-4.2,-1) .. controls (-3.3,-1.2) and (-2.6,-.8) .. (-1.8,-1) coordinate (x3)coordinate [pos=.5] (z1);
\fill (z1) circle (1pt);\node[above] at (z1) {$z_{1}$};
\node [right] at (x1) {$X_{1}$};\node [right] at (x2) {$X_{2}$};\node [right] at (x3) {$X_{3}$};\node [above right] at (q1) {$q_{1}$};\node [left] at (q2) {$\, q_{2}$};

\fill (0,0) coordinate (x1) circle (2pt);\node [left] at (x1) {$X_{1}$};
\fill (3,0) coordinate (x2) circle (2pt); \node [right] at (x2) {$X_{2}$};
\fill (1.5,-1) coordinate (x3) circle (2pt);
\node [above] at (x3) {$X_{3}$};

 \draw[postaction={decorate}] (x1) -- (x3);
 \draw[postaction={decorate}] (x2) -- (x3);

\end{tikzpicture}
\caption{The curve and level graph used in Example~\ref{exa:GRCneeded} }\label{cap:GRCneeded}
\end{figure}
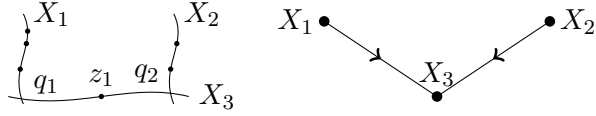
\par
Let $\eta$ be a twisted differential on $X$ compatible with this level graph. Since $X_3$ is on the bottom level, $\eta_3$ has poles
at $q_{1}^{-}$ and $q_{2}^{-}$. Let $r_{1}$ and $r_{2}$ be the residues of $\eta_{3}$ at $q_{1}^{-}$ and $q_{2}^{-}$, respectively. The residue theorem on $X_3\cong \mathbb{P}^1$ says that $r_1+r_2=0$, with no further constraints. However, the global residue condition applied to the level of $X_3$ implies that $r_1=0$ and $r_2=0$, which does not follow from the relation $r_1+r_2=0$.

\end{exa}
\par
\begin{exa}\label{exa:etanotunique}
({\em Non-uniqueness of associated twisted differentials}) Theorem~\ref{thm:main} says that a pointed stable differential lies in the incidence variety compactification of a given stratum if and only if there exists an associated twisted differential of the given type compatible with certain level graph
 (see Definition~\ref{def:associated-diff}). However, such a twisted differential may not be unique (modulo scaling), even for a fixed level graph. For example, suppose $X$ has three irreducible components $X_1\succ X_2\succ X_3$, where $X_1$ intersects $X_2$ at one point~$q_{1}$, and $X_2$ and $X_3$ intersect at two points $q_2$ and $q_3$, and suppose all marked points $z_1, \ldots, z_n$ lie on~$X_3$, see Figure~\ref{cap:etanotunique}.
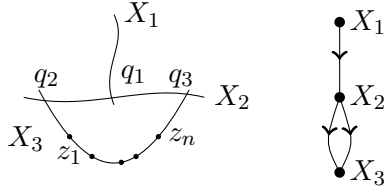
\begin{figure}[ht]
\begin{tikzpicture}[scale=1,decoration={
    markings,
    mark=at position 0.5 with {\arrow[very thick]{>}}}]
\draw (-3,0.1) coordinate (x1) .. controls (-2.8,-.3) and (-3.2,-.6) .. (-3,-1.1) coordinate[pos=.9] (q1);
\draw (-4.2,-1)  coordinate (q2) .. controls (-3.3,-1.2) and (-2.6,-.8) .. (-1.8,-1) coordinate (x3);
\draw (-4,-0.9) .. controls (-3.3,-2.2) and (-2.7,-2.2) .. (-2,-.9)
node [pos=.1,below left]{$X_{3}$} coordinate [pos=.2] (z1) coordinate [pos=.35] (z2) coordinate [pos=.65] (z3) coordinate [pos=.8] (z4) coordinate [pos=.55] (z5);

\node [right] at (x1) {$X_{1}$};\node [right] at (x3) {$X_{2}$};
\node [above right] at (q1) {$q_{1}$};\node [above right] at (q2) {$q_{2}$};\node [above left] at (x3) {$q_{3}$};

\fill (z1) circle (1pt); \node [below] at (z1) {$z_{1}$};
\fill (z2) circle (1pt);\fill (z3) circle (1pt);\fill (z5) circle (1pt);
\fill (z4) circle (1pt); \node [right] at (z4) {$z_{n}$};

\fill (0,0) coordinate (x1) circle (2pt);\node [right] at (x1) {$X_{1}$};
\fill (0,-1) coordinate (x2) circle (2pt); \node [right] at (x2) {$X_{2}$};
\fill (0,-2) coordinate (x3) circle (2pt); \node [right] at (x3) {$X_{3}$};

 \draw[postaction={decorate}] (x1) -- (x2);
 \draw[postaction={decorate}] (x2) ..controls (.2,-1.6)  ..  (x3);
 \draw[postaction={decorate}] (x2) ..controls (-.2,-1.6)  ..  (x3);
\end{tikzpicture}
\caption{The curve and level graph used in Example~\ref{exa:etanotunique} }\label{cap:etanotunique}
\end{figure}
\par

Suppose $\eta$ is a twisted differential compatible with the corresponding level graph $\overline{\Gamma}$. Because
$X_1$ has no marked points and it is on the top level,  $\eta_1$ on $X_1$ is holomorphic, and it has a unique zero at $q_{1}^{+}$ whose order is $2g_1-2$. Moreover, $\eta_3$ on $X_3$ has all the prescribed zeros or poles at $z_i$, hence the sum
of its pole orders at $q_2^-$ and $q_3^-$ is equal to $2g-2g_3$. Suppose that
$\eta_3$ has a pole of order $k$ at $q_2^-$ and a pole of order $2g-2g_3-k$ at $q_3^-$. Then
$\eta_2$ as a differential on $X_2$ has a pole of order $2g_1$
at $q_{1}^-$, and has zeros of orders $k-2$ and $2g-2g_3-k-2$ at $q_2^+$ and $q_3^+$,
respectively. Finally applying the global residue condition to each level of $\overline{\Gamma}$, respectively, it says that the sum of residues of $\eta_i$ on each irreducible component $X_i$ is zero, which follows from the residue theorem. Hence in this case the global residue condition imposes no further constraints on $\eta$.

By the above analysis, there exists a twisted differential $\eta$ compatible with
$\overline{\Gamma}$ if and only if the pointed curve $(X,z_{1},\ldots,z_{n})$
satisfies
\bas
K_{X_1}& \sim (2g_1-2)\;q_{1}^{+}, \\
K_{X_2}& \sim (k-2) q_2^{+}\+(2g-2g_3-k-2)q_3^{+}\,-\,2g_1 q_{1}^{-}, \\
K_{X_3}&\sim \sum_{i=1}^n m_i z_i\,-\,k q_2^{-} \,-\,(2g-2g_3-k) q_3^{-}.
\eas
For special curves $X$ these conditions can be satisfied for different values of~$k$. For instance, take
all $X_i$ to be hyperelliptic curves of high genus, all $q_i^{\pm}$ to be Weierstra\ss{} points, and $k$ to be even. In that case we obtain a number of distinct twisted differentials~$\eta$, all of which are the same on $X_1$ (modulo scaling) but are different on~$X_2$ and~$X_3$ (even after modulo scaling).

Suppose $\omega$ is the associated stable differential of $\eta$. Then $\omega_1 = \eta_1$ is determined on the top level component $X_1$, which is the same for all the different choices of $\eta$ (modulo scaling). Nevertheless, $\omega_2$ and $\omega_3$ are identically zero on the lower level components $X_2$ and $X_3$, respectively. Hence in this case different values of $k$ in the above give rise to distinct twisted differentials $\eta$, but the associated stable differential $(X, \omega, z_1, \ldots, z_n)$ remains to be the same.
\end{exa}

\begin{exa}\label{exa:nopartialorder}({\em Pointed stable differentials do
not determine the level graph})  The previous example shows that there may be many different twisted differentials that are associated with a given pointed stable differential in the incidence variety compactification. We now show that a pointed stable differential $(X, \omega, z_1, \ldots, z_n)$ does not necessarily determine a full order on the dual graph. The reason is that $\omega$ is identically zero
on the lower level components of $X$, hence the lower level components may be ordered differently. An example illustrated in Figure~\ref{cap:nopartialorder} is given by a triangular subgraph on lower levels, attached to a top level component $X_1$. The only two marked points are $z_1\in X_3$ and $z_2\in X_4$. The two different level graphs are obtained by switching the ordering of $X_3$ and $X_4$ on the bottom two levels.
\begin{figure}[ht]
\begin{tikzpicture}[scale=1,decoration={
    markings,
    mark=at position 0.5 with {\arrow[very thick]{>}}}]

\draw (-3,0.1) coordinate (x1) .. controls (-2.8,-.3) and (-3.2,-.6) .. (-3,-1.1) coordinate [pos=.8] (q1);
\draw (-4.5,-1) .. controls (-3.3,-1.2) and (-2.6,-.8) .. (-1.5,-1) coordinate (x3);

\draw (-4.3,-.8) coordinate (x3).. controls (-3.8,-1.5) and (-3.2,-2) .. (-2.9,-2.6) coordinate (q4)
coordinate [pos=.5] (z1);
\draw (-1.7,-.8) coordinate (x4).. controls (-2,-1.2) and (-2.6,-1.6) .. (-3.1,-2.6)
coordinate [pos=.5] (z2);

\node [left] at (z1) {$X_{3}$};\node [above right] at (x4) {$X_{4}$};
\fill (z1) circle (1pt); \node [right] at (z1) {$z_{1}$};
\fill (z2) circle (1pt); \node [right] at (z2) {$z_{2}$};
\node [right] at (q1) {$q_{1}$};\node [right] at (x3) {$q_{2}$};
\node [left] at (x4) {$q_{3}$};\node [right] at (q4) {$q_{4}$};

\node [right] at (x1) {$X_{1}$};\node [below left] at (x3) {$X_{2}$};

\fill (0,0) coordinate (x1) circle (2pt);\node [right] at (x1) {$X_{1}$};
\fill (0,-1) coordinate (x2) circle (2pt); \node [right] at (x2) {$X_{2}$};
\fill (1,-2) coordinate (x3) circle (2pt); \node [below] at (x3) {$X_{4}$};
\fill (-1,-3) coordinate (x4) circle (2pt); \node [left] at (x4) {$X_{3}$};

 \draw[postaction={decorate}] (x1) -- (x2);
 \draw[postaction={decorate}] (x2) -- (x3);
 \draw[postaction={decorate}] (x2) -- (x4);
\draw[postaction={decorate}]  (x3) -- (x4);

\fill (3,0) coordinate (x1) circle (2pt);\node [right] at (x1) {$X_{1}$};
\fill (3,-1) coordinate (x2) circle (2pt); \node [right] at (x2) {$X_{2}$};
\fill (4,-3) coordinate (x3) circle (2pt); \node [right] at (x3) {$X_{4}$};
\fill (2,-2) coordinate (x4) circle (2pt); \node [below] at (x4) {$X_{3}$};

 \draw[postaction={decorate}] (x1) -- (x2);
 \draw[postaction={decorate}] (x2) -- (x3);
 \draw[postaction={decorate}] (x2) -- (x4);
\draw[postaction={decorate}] (x4) -- (x3);
\end{tikzpicture}
\caption{The curve with two different level graphs used in Example~\ref{exa:nopartialorder}}\label{cap:nopartialorder}
\end{figure}
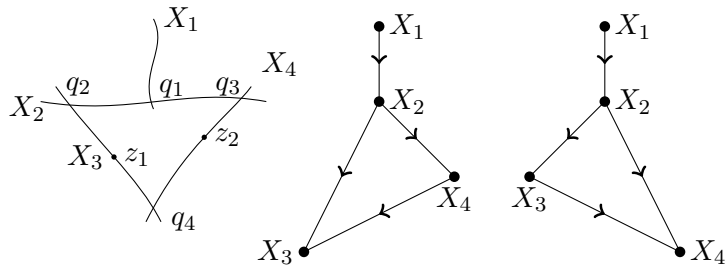

Suppose the genera of $X_1$ and $X_2$ are sufficiently high so that  $m_1>2g_3-2$ and $m_2>2g_4-2$.
Take a twisted differential $\eta$ such that
\bas
{\rm div}(\eta_1) & = (2g_1-2) q_1^+,  \\
 {\rm div}(\eta_2) & = k q_2^{+} + (2g_2+2g_1-k-2)q_3^{+} - 2g_{1} q_{1}^{-},  \\
 {\rm div}(\eta_3) & = m_1 z_1+(2g_3+k-m_1) q_4 - (k+2) q_2^{-}, \\
  {\rm div}(\eta_4) & = m_2 z_2+(k-2g_1-2g_2) q_3^{-}+(m_1-2-k-2g_3) q_4,
\eas
where $0\leq k \leq 2g_2+2g_1-2$. The existence of such $\eta$ is equivalent to the linear equivalence conditions
${\rm div}(\eta_i) \sim K_{X_i}$ for all $i$. One can take $X_i$ to be hyperelliptic curves, $z_i$ and $q_j$ to be Weierstra\ss{} points, $m_i$ and $k$ to be even, so that these conditions hold.
In that case, if we choose suitable $k$ such that $2g_3+k-m_1 > -1$, then $\eta_3$ is holomorphic at $q_4$ on $X_3$, hence $\eta$ is compatible with the level graph on the right side of Figure~\ref{cap:nopartialorder}, where the global residue condition follows from the residue theorem
on each lower level component. Conversely if $2g_3+k-m_1 < -1$, then $\eta_3$ has a higher order pole at $q_4$ on~$X_3$, hence $\eta$ is compatible with the level graph in the middle of Figure~\ref{cap:nopartialorder}. Nevertheless, in both cases the associated stable differentials $\omega$ are the same (modulo scaling on the top level component $X_1$), because $\omega$ is identically zero on the components $X_2$, $X_3$, and~$X_4$.
\end{exa}

\begin{exa}\label{exa:differentGRCIVC}
({\em Different level graphs give different global residue conditions})
Consider two level graphs $\overline{\Gamma}_{1}$ and $\overline{\Gamma}_{2}$ with the same underlying dual graph $\Gamma$ of a curve $X$, presented in Figure~\ref{cap:Order}. The orientation of the edges is going downwards.
Because $X_5$ and $X_6$ are disjoint, $\overline{\Gamma}_{1}$ and $\overline{\Gamma}_{2}$ determine the same partial order on $\Gamma$. Further suppose there is no marked pole in the smooth locus of $X$.
\par
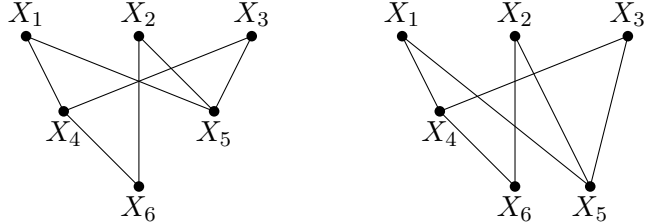
\begin{figure}[ht]
\begin{tikzpicture}[scale=1]

\fill (0,0) coordinate (x1) circle (2pt);\node [above] at (x1) {$X_{1}$};
\fill (1.5,0) coordinate (x2) circle (2pt); \node [above] at (x2) {$X_{2}$};
\fill (3,0) coordinate (x3) circle (2pt); \node [above] at (x3) {$X_{3}$};
\fill (.5,-1) coordinate (y1) circle (2pt);\node [below] at (y1) {$X_{4}$};
\fill (2.5,-1) coordinate (y2) circle (2pt);\node [below] at (y2) {$X_{5}$};
\fill (1.5,-2) coordinate (y3) circle (2pt); \node [below] at (y3) {$X_{6}$};

 \draw[] (x1) -- (y1);
 \draw[] (x1)  --  (y2);
 \draw[] (x2)  --  (y2);
 \draw[] (x2) -- (y3);
  \draw[] (y1) -- (y3);
 \draw[] (x3) -- (y1);
 \draw[] (x3)  --  (y2);

 \coordinate (a) at (5,0);
\fill (a)+(0,0) coordinate (x1) circle (2pt);\node [above] at (x1) {$X_{1}$};
\fill (a)+(1.5,0) coordinate (x2) circle (2pt); \node [above] at (x2) {$X_{2}$};
\fill (a)+(3,0) coordinate (x3) circle (2pt); \node [above] at (x3) {$X_{3}$};
\fill (a)+(.5,-1) coordinate (y1) circle (2pt);\node [below] at (y1) {$X_{4}$};
\fill (a)+(2.5,-2) coordinate (y2) circle (2pt);\node [below] at (y2) {$X_{5}$};
\fill (a)+(1.5,-2) coordinate (y3) circle (2pt); \node [below] at (y3) {$X_{6}$};

 \draw[] (x1) -- (y1);
 \draw[] (x1)  --  (y2);
 \draw[] (x2)  --  (y2);
 \draw[] (x2) -- (y3);
  \draw[] (y1) -- (y3);
 \draw[] (x3) -- (y1);
 \draw[] (x3)  --  (y2);
\end{tikzpicture}
 \caption{Two level graphs $\overline{\Gamma}_{1}$ and $\overline{\Gamma}_{2}$ used in Example~\ref{exa:differentGRCIVC}} \label{cap:Order}
\end{figure}

We will now determine the conditions necessary for a twisted differential $\eta$ to be compatible with $\overline{\Gamma}_{1}$ or $\overline\Gamma_2$, respectively. For a node $q_k$ joining $X_i$ and $X_j$ with $X_i\succ X_j$, we denote by $r_{i,j}$ the residue of~$\eta_j$ at~$q_{k}^-$. In both cases $\eta_1$,  $\eta_2$, and $\eta_3$ are holomorphic, and
$\eta_4$,  $\eta_5$, and $\eta_6$ are meromorphic. Applying the residue theorem to $\eta_i$ on $X_i$ for $i = 4, 5, 6$, we obtain that
$$r_{1,4}+r_{3,4}=0, \quad r_{1,5}+r_{2,5}+r_{3,5}=0, \quad r_{2,6}+r_{4,6}=0. $$

The global residue condition can be imposed at the bottom level and at the middle level of $\overline{\Gamma}_{i}$. For $\overline\Gamma_1$ the bottom level global residue condition implies that
$$
 r_{2,6}+r_{4,6}=0
$$
(since the graph above the bottom level is connected), which follows from the residue theorem. The global residue condition applied to the middle level
of $\overline\Gamma_1$ implies that
$$
 r_{1,4}+r_{1,5}\=0,  \quad r_{2,5}\=0, \quad r_{3,4}+r_{3,5}\=0
$$
where the third condition follows from the first two and the residue theorem. Thus for~$\overline\Gamma_1$ the global residue condition gives two extra relations among the residues of $\eta$, in addition to the residue theorem. In particular, it implies that $r_{2,5}=0$.

For $\overline\Gamma_2$ the global residue condition applied to the bottom level implies that
$$
 r_{4,6}+r_{1,5}+r_{3,5}\=0, \quad r_{2,5}+r_{2,6}\=0
$$
(since the graph above the bottom level has two connected components, one being $X_2$, and the other everything else), where the two conditions are equivalent by the residue theorem. The global residue condition applied to the middle level of $\overline\Gamma_2$ implies that $r_{1,4}=0$
(or equivalently, $r_{3,4}=0$, by the residue theorem). Thus for $\overline\Gamma_2$ the global residue condition gives two extra relations among the residues of $\eta$, in addition to the residue theorem. In particular, it implies that $r_{1,4}=0$, which is different from the conditions imposed by $\overline\Gamma_1$ in the preceding paragraph.
\end{exa}

\begin{exa}\label{exa:nonLinearResidue}
({\em The space of residues cut out by the global residue condition})
In this example we illustrate how the scaling factors of a twisted differential on lower level components of $X$ come into play in the global residue condition. Consider the curve $X$ with the level graph $\overline{\Gamma}$ in Figure~\ref{cap:nonLinearResidue}. Suppose there is no marked pole in the smooth locus of $X$.
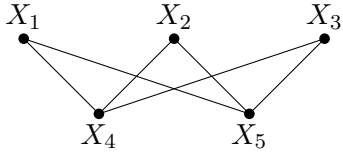
\begin{figure}[ht]
\begin{tikzpicture}

\fill (0,0) coordinate (x1) circle (2pt);\node [above] at (x1) {$X_{1}$};
\fill (2,0) coordinate (x2) circle (2pt); \node [above] at (x2) {$X_{2}$};
\fill (4,0) coordinate (x3) circle (2pt); \node [above] at (x3) {$X_{3}$};

\fill (1,-1) coordinate (y1) circle (2pt);  \node [below] at (y1) {$X_{4}$};
\fill (3,-1) coordinate (y2) circle (2pt); \node [below] at (y2) {$X_{5}$};

 \draw(x1) -- (y1);
 \draw (x1) -- (y2);
 \draw (x2) -- (y1);
 \draw (x2) -- (y2);
  \draw (x3) -- (y1);
 \draw (x3) -- (y2);
\end{tikzpicture}
 \caption{The curve and level graph used in Example~\ref{exa:nonLinearResidue}} \label{cap:nonLinearResidue}
\end{figure}
\par
Let $\eta$ be a twisted differential on $X$ whose associated stable differential is $\omega$. Suppose further that $\eta$ satisfies conditions (0)--(3) with respect to $\overline{\Gamma}$ (see Definitions~\ref{def:twistedMeroDiff} and~\ref{def:twistedAbType}). Take a pair of scalars $\lambda = (\lambda_4, \lambda_5)$. Define a twisted differential
$$\eta(\lambda) \= \{ \eta_1, \eta_2, \eta_3, \lambda_4 \eta_4, \lambda_5 \eta_5\}. $$
Because $X_4$ and $X_5$ are the lower level components, the associated stable differential of $\eta(\lambda)$ is $\omega$, and $\eta(\lambda)$
also satisfies conditions (0)--(3) with respect to $\overline{\Gamma}$. Conversely, any twisted differential with associated stable differential equal to $\omega$ is of the type $\eta(\lambda)$. By Theorem~\ref{thm:main}, $(X, \omega, z_1, \ldots, z_n)$ is contained in the incidence variety compactification if and only if there exists one such $\eta(\lambda)$ satisfying the global residue condition.

Use the notation $r_{i,j}$ for residues of $\eta$ as
in the preceding example. Then the residues of $\eta(\lambda)$, compared to $\eta$, are multiplied by $\lambda_4$ and $\lambda_5$ on $X_4$ and $X_5$, respectively.
Applying the residue theorem to $X_4$ and $X_5$ gives
$$
 r_{1,4}+r_{2,4}+r_{3,4}\=0, \quad r_{1,5}+r_{2,5}+r_{3,5}\=0.
$$
The global residue condition applied to $\eta(\lambda)$ and $\overline{\Gamma}$ imposes the conditions
$$
 \lambda_4 r_{1,4}+ \lambda_5 r_{1,5}\=0, \quad \lambda_4 r_{2,4}+\lambda_5 r_{2,5}\=0, \quad \lambda_4 r_{3,4}+ \lambda_5 r_{3,5}\=0.
$$
By the residue theorem we can express $r_{3,4}$ and $r_{3,5}$ in terms of the other residues, hence the imposed conditions reduce to
\[\left\{\begin{matrix}
\lambda_{4}r_{1,4} + \lambda_{5}r_{1,5} \= 0, \\
\lambda_{4}r_{2,4} + \lambda_{5}r_{2,5} \= 0.
\end{matrix}
\right.
\]
Eliminating $\lambda_4$ and $\lambda_5$ in the above further reduces the conditions to
\[r_{1,4}r_{2,5}-r_{2,4}r_{1,5}\=0.\]
Hence there exists a twisted differential $\eta(\lambda)$ compatible with $\overline{\Gamma}$ only if the residues of~$\eta$ satisfy this \emph{quadratic} equation. This phenomenon will be addressed in full generality in our forthcoming work \cite{BCGGM2}, where we study in detail the boundary structure of the incidence variety compactification.
\end{exa}

\subsection{Some preliminary results}\label{subsec:prelim}

We first remark that the incidence variety compactification $\PP\obarmoduliincp{m_{1},\ldots,m_{n}}$ of any stratum has divisorial boundary (i.e.,~the boundary is of complex codimension one in the compactification). This follows from the fact that $\barmoduli[g,n]$ and the projectivized Hodge bundle have divisorial boundary, and this is where we take the closure of the stratum. In this section we describe explicitly the boundary of $\PP\obarmoduliincp{m_{1},\ldots,m_{n}}$ for a few cases. The general method is as follows. First, we enumerate the boundary strata in $\barmoduli[g,n]$ parameterizing stable curves with a given dual graph. For each such stratum, we investigate all possible level graphs. Then we write down the space of all possible twisted differentials $\eta$ that are compatible with such a level graph, which we write as a formula saying that $\eta$ lies in a suitable product of strata of differentials, recording its zeros and poles on the irreducible components of $X$. Given such $\eta$, the associated pointed stable differential $\omega$ coincides with $\eta$ on the top level components, and is identically zero elsewhere. To conclude, we use Lemma~\ref{lm:dimFibrePiDeux} to compute the dimension of the fiber of the map $\pi_2$ from the incidence variety compactification to its image in the Deligne-Mumford compactification (see the setting~\eqref{eq:forget2}).

For later use we denote by $\DC_{i}$ the {\em open} divisorial components of the boundary of $\barmoduli$. More precisely, $\DC_i$ for $1\leq i\leq [g/2]$ parameterizes reducible curves consisting of a genus $i$ smooth component and a genus $g-i$ smooth component attached at one node, and $\DC_0$ parameterizes irreducible nodal curves of geometric genus $g-1$.

We first consider the strata in genus zero, since they play a
central role in many cases. Recall that there exists a meromorphic differential on $\PP^1$ with any prescribed collection of zeros and poles, as long as the sum of their orders is equal to $-2$. As a consequence of Theorem~\ref{thm:main}, we deduce the following statement about residues of meromorphic differentials on~$\PP^{1}$.
\par
\begin{lm}\label{lm:gzeroontresidu}
Let $\omoduli[0](m_{0},\ldots,m_{n})$ be a stratum of meromorphic differentials
in genus zero such that  $m_{0}>0$ and $m_{i}<0$ for $1\leq i\leq n$. If~$n \geq 2$, then
every differential in this stratum has a nonzero residue at some pole.
\end{lm}
\begin{proof}
If some $m_i = -1$, then the corresponding pole is simple, hence it has a nonzero residue. From now on assume that $m_i \leq -2$
for all $1\leq i \leq n$.
\par
We argue by contradiction. Suppose this stratum contains a differential~$\eta_0$ on $X_0 \cong \PP^1$
with zero residue at every pole. Denote by $z_{0}$ the zero of $\eta_0$ and $z_{i}$ the pole of~$\eta_0$
of order~$m_{i}$ for $1\leq i \leq n$. For every $i\geq 1$ we attach to $X_0$
at $z_{i}$ a curve $X_{i}$ as follows. If $m_{i}$ is even, we take a curve
of genus $g_{i}=-m_{i}/2$ such that there exists a differential $\eta_{i}$ on $X_{i}$
with a unique zero of order $-m_i-2$ at $z_0$. If $m_{i}$ is odd, we attach
a curve of genus $g_{i}=(-m_{i}+1)/2$ such that there exists a differential $\eta_{i}$
on $X_{i}$ with a zero of order $-m_{i}-2$ at $z_i$ and a simple zero elsewhere.
Let~$X$ be the resulting stable curve.
\par
The collection of $\eta_0, \eta_1, \ldots, \eta_n$ defines a twisted differential~$\eta$ of type
$(m_0,1,\ldots,1)$ on $X$, where the number of ones is equal to the number $d$ of odd negative $m_{i}$. Consider the level graph
$\overline{\Gamma}$ on $X$ such that $X_1\asymp\cdots \asymp X_n \succ X_0$. Given the hypothesis that $\eta_0$ has zero residues at $z_1, \ldots, z_n$, the \twd $\eta$ is compatible with $\overline{\Gamma}$. Hence by Theorem~\ref{thm:main}, all stable pointed differentials associated to such~$\eta$
are in the incidence variety compactification of the stratum
$\proj\omoduli(m_{0},1,\ldots,1)$, where $g=\sum_{i=1}^{n}\lfloor\frac{-m_{i}+1}{2}\rfloor$.
\par
Note that
$$\dim \proj\omoduli(m_{0},1,\ldots,1) \= 2g+d-1.$$
On the other
hand, the dimension of the space of stable pointed differentials associated to the
twisted differentials we constructed, with the projectivization of~$\eta_0$
on~$X_0$ fixed, has dimension
$$ -1 \+ \sum_{i>0\,,\, m_{i}\text{ even}}2g_{i}+\sum_{i>0\,,\,m_{i}\text{ odd}}
(2g_{i}+1)\=2g+d-1\,,$$
where the $-1$ results from simultaneous projectivization of the differentials on the top level components. Hence it has the
same dimension as the stratum, contradicting the fact that the boundary of a closed variety has smaller dimension compared to that of
the interior.
\end{proof}

\begin{rem}
\label{rem:referee}
The lemma can also be proved directly in two ways. If $\omega$ is such a differential with no residue, then $\omega = df$ for some meromorphic function $f$ on $\PP^1$. Regarding $f$ as a map from $\PP^1$ to $\PP^1$, the unique zero of $\omega$ corresponds to the unique ramification point of $f$ over $\PP^1\setminus \{\infty\}$, and the $n$ poles of $\omega$ map to $\infty$. It follows from the Riemann-Hurwitz formula that $2d-2 = (d-1) + (d-n)$ where $d$ is the degree of $f$, which contradicts that $n\geq 2$. Alternatively if $\omega$ has no residue, then the flat geometric representation of $\omega$ in the sense of \cite{boissy} has no saddle connections, hence the surface would be a wedge sum of $n$ spheres attached at the unique zero, leading to a contradiction.
\end{rem}

The above lemma yields a useful criterion when we apply the global residue condition to twisted differentials on a stable curve with a
rational component.

\begin{cor}\label{cor:pasSep}
Suppose $\mu = (m_1, \ldots, m_n)$ is a holomorphic type, i.e. all~$m_i$ are positive. Let $(X,z_1,\ldots,z_n)$ be a stable curve with an irreducible component $X_v$ of genus zero, such that all the nodes $q_1,\ldots,q_k$ contained in $X_v$ are separating, and $k\ge 2$.
If $\eta$ is a twisted differential of type $\mu$ such that $\eta_{v}$ on $X_v$ has a unique zero and has a pole at each $q_i$, then
$\eta$ is not compatible with any level graph on $X$.
\end{cor}

\begin{proof} Suppose, on the contrary, that there is a level graph $\overline{\Gamma}$ on $X$ such that $\eta$ is compatible with
$\overline{\Gamma}$. Since $\eta_v$ has a pole at each node $q_i$, $X_v$ is on a lower level compared to the component on the other branch of $q_i$. Since $q_i$ is separating, the global residue condition applied to the level of $X_v$ implies that
$\Res_{q_i} \eta_v = 0$ for all $i$, contradicting Lemma~\ref{lm:gzeroontresidu}.
\end{proof}

Next we describe a relationship between pointed stable differentials and compatible level graphs. An irreducible component $X_{v}$ of $X$ is called a {\em non-strict local minimum} if for every component $X_{u}$ intersecting $X_{v}$, the inequality $X_{u}\preccurlyeq  X_{v}$ holds.
\begin{lm}\label{lm:minimumComp}
Let $(X,\omega,z_{1},\cdots,z_{n})$ be a pointed stable differential
and suppose that $g \neq 1$ or $n \neq 0$. If an irreducible component $X_v$
of $X$  is a (non-strict) local minimum in a level graph~$\overline\Gamma$,
then any twisted differential $\eta$ compatible with $\overline\Gamma$
that agrees with $\omega$ on the top level components,
has at least one zero or marked point at a smooth point of $X_v$.
\end{lm}
\par
\begin{proof}
Since $X_v$ is a (non-strict) local minimum, the differential $\eta_v$ has a pole at any node lying on~$X_{v}$. If $\eta_v$ has no marked zero
in the smooth locus of~$X_v$, then $X_v \cong \mathbb{P}^1$, because the only Riemann surfaces on which there exist a meromorphic differential without zeros are $\PP^1$ and a genus $1$ curve --- and in the latter case the differential must be holomorphic, with no zeroes or poles, and thus $X_v$ must equal to $X$, which is the excluded case. In that case $\eta_v$ either has two simple poles or a single double pole; hence  for the curve to be stable, $X_v$ must contain at least one marked point.
\end{proof}

This lemma implies the following useful result.

\begin{cor}\label{cor:oneminimum}
For $g\geq 1$ any pointed stable differential $(X,\omega,z)$ in the incidence variety compactification $\pobarmoduliinc[g,1]{2g-2}$
is compatible with a level graph that has a unique local minimum, which is the irreducible component of $X$ containing the marked point~$z$.
\end{cor}

We conclude this section by the following lemma for later use. Recall the map $\pi_2$ defined in~\eqref{eq:forget2} from the incidence variety compactification to the Deligne-Mumford compactification.

\begin{lm}\label{lm:degenloop}
Let $S$ be an open boundary stratum of $\barmoduli[g,n]$ parameterizing curves with a fixed dual graph. Let $\tilde{S}$ be the boundary stratum corresponding to the dual graph obtained from the dual graph of $S$ by adding a loop at a vertex $v$ and decreasing the geometric genus of $X_v$ by one (i.e., adding a non-separating node $q$ for an irreducible component $X_v$ of curves in $S$). Then the dimension of $\pi_2^{-1}(\tilde{S})$ in the incidence variety compactification of any stratum of differentials is either strictly smaller than the dimension of $\pi_2^{-1}(S)$ or both preimages are empty.
\end{lm}
\begin{proof}
To prove this result it suffices to show that $\pi_{2}^{-1}(\tilde{S})$ is contained in the closure of $\pi_{2}^{-1}(S)$. Let $(X,\omega,z_{1},\cdots,z_{n})$ be a stable pointed differential in $\pi_{2}^{-1}(\tilde{S})$. Choose an associated twisted differential $\eta$ in the sense of Definition~\ref{def:associated-diff}. Since the newly added node $q$ is an internal node, $\eta$ has simple poles with opposite residues at the two branches of $q$ by conditions (1), (2), and (3). Then one can locally smooth out $q$ to obtain a family of twisted differentials $\eta_{t}$, such that the stable differential associated to the general $\eta_{t}$ lies in $\pi_{2}^{-1}(S)$. This operation of locally smoothing a simple polar node can be performed by classical plumbing (see Lemma~\ref{prop:classical_plumbing}).
\end{proof}

\subsection{The incidence variety compactification of $\PP\omoduli[2](2)$}

In this section we work out in detail the irreducible components of the boundary
$\partial\PP\obarmoduliinc[2,1]{2}$
in the incidence variety compactification of $\PP\omoduli[2](2)$. This
stratum is connected. Moreover,~$\pi_2$ maps it to the locus of Weierstra\ss{} points in $\moduli[2,1]$.
The closure of this stratum was described already in \cite[Section~6]{gendron} by using admissible double covers.
Here we match these results and demonstrate how our current machinery works.  Coordinates around
part of this boundary were described in \cite[Section~6.6]{bainbridge_euler}.
\par
\begin{prop} \label{prop:str(2)}
The boundary of $\PP\obarmoduliinc[2,1]{2}$ has three irreducible components, given by
the closures of the loci $(\rom{1})$, $(\rom{2})$, and $(\rom{3})$ defined below.
\end{prop}
\par
The strategy of the proof is as follows. Recall that the boundary of
$\PP\obarmoduliinc[2,1]{2}$ is divisorial (see the beginning of Section~\ref{subsec:prelim}), hence its irreducible components are two-dimensional.
It thus suffices to locate all two-dimensional boundary strata of $\PP\obarmoduliinc[2,1]{2}$. We perform the search according to the dimension of the boundary strata
in~$\barmoduli[2,1]$.

We first  determine the $\pi_2$-preimage in $\PP\obarmoduliinc[2,1]{2}$ of every
open divisorial stratum of $\barmoduli[2,1]$. Let $X$ be a stable curve in the stratum~$\DC_0$, see Figure~\ref{cap:type-I}.
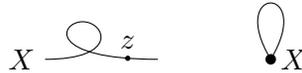
\begin{figure}[ht]
\begin{tikzpicture}[scale=1,decoration={
    markings,
    mark=at position 0.5 with {\arrow[very thick]{>}}}]

\draw [] (-3, 0) coordinate (x1)
  .. controls ++(0:1) and ++(0: .3) .. ( -2.5, .5)
  .. controls ++(180:.3) and ++(180:1.5) .. ( -1.5, 0) coordinate [pos=.9] (z);
\fill (z) circle (1pt);\node[above] at (z) {$z$};

\node [left] at (x1) {$X$};
\fill (0,0) coordinate (x1) circle (2pt);\node [right] at (x1) {$X$};
   \draw(x1) .. controls (.6,1) and (-.6,1) .. (x1) ;
\end{tikzpicture}
 \caption{A curve parameterized by $\DC_0$ and its level graph} \label{cap:type-I}
\end{figure}
Since $X$ is irreducible, there is no order to consider on the dual graph. Given a twisted differential $\eta$, the associated stable differential $\omega$ coincides with $\eta$, which is not identically zero on $X$. Such a differential has a double zero and two simple poles on the normalization
of $X$, which is of genus one. Thus the $\pi_2$-preimage of~$\DC_{0}$ in $\pobarmoduli[2,1]$ can be identified with the locus
$$
 (\rom{1})\,\coloneqq\,\PP\omoduliinc[1,\{3\}]{2,-1,-1},
$$
where the two simple poles are not ordered. This is the locus of elliptic
curves $E$ with three marked points $p_1,p_2,p_3$ such that $2p_1=p_2+p_3$.
It is irreducible according to \cite{boissymero}.
The locus $(\rom{1})$ is two-dimensional, and hence it gives an irreducible
component of the boundary of $\pobarmoduliinc[2,1]{2}$.
\par
In the open boundary divisor $\DC_1$, a stable curve~$X$ is the union of two elliptic curves~$X_{1}$ and $X_{2}$ intersecting at $q$, where $z\in X_{2}$. Suppose $\eta$ is a twisted differential compatible with a level graph $\overline{\Gamma}$ on $X$. Then $\eta_2$ on $X_2$ has a unique zero of order two at $z$, hence $\eta_2$ has a double pole at $q^-$, and the component $X_{2}$ is of lower level than~$X_{1}$ in $\overline{\Gamma}$, see Figure~\ref{cap:type-2}. Consequently $\eta_1$ on $X_1$ is holomorphic and nowhere vanishing.
\begin{figure}[ht]
\begin{tikzpicture}[scale=1,decoration={
    markings,
    mark=at position 0.5 with {\arrow[very thick]{>}}}]

\draw (-2.2,0) coordinate (x1).. controls (-1.7,-.2) and (-1.3,.2) .. (-1,0) coordinate[pos=.78](q1);
\draw (-1.2,0.2) .. controls (-1.4,-.3) and (-1,-.6) .. (-1.2,-1.1) coordinate (x2)
coordinate [pos=.6] (z) coordinate [pos=.22] (q2);
\fill (z) circle (1pt);\node[right] at (z) {$z$};

\node [left] at (x1) {$X_{1}$};\node [left] at (x2) {$X_{2}$};
\node [above] at (q1) {$q^{+}$};\node [right] at (q2) {$q^{-}$};

\fill (0,0) coordinate (x1) circle (2pt);\node [right] at (x1) {$X_{1}$};
\fill (0,-1) coordinate (x2) circle (2pt); \node [right] at (x2) {$X_{2}$};

 \draw[postaction={decorate}] (x1) -- (x2);
\end{tikzpicture}
 \caption{A curve parameterized by $\DC_1$ and its level graph} \label{cap:type-2}
\end{figure}
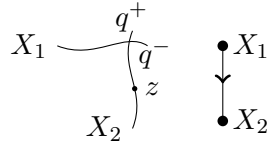
The global residue condition requires that $\Res_{q^-} \eta_2 = 0$, which follows from the residue theorem for $\eta_2$ on $X_2$. We have thus verified that in the above setting $\eta$ is compatible with $\overline{\Gamma}$. Such $(X, \eta)$ is equivalent to the collection of $(X_1,\eta_1,q^+)\in\omoduli[1,1]$ and $(X_2,\eta_2,z,q^-)\in\omoduliinc[1,2]{2,-2}$. Suppose $\omega$ is the associated stable differential of~$\eta$. Then $\omega|_{X_1} = \eta_1$ and $\omega$ is identically zero on the lower level component $X_2$. In other words, the associated pointed stable differential $(X, \omega, z)$ consists of $(X_1,\eta_1,q^+)$ and $(X_2,z,q^-)$, which forgets $\eta_2$
compared to $(X, \eta)$. In this sense we can identify the $\pi_2$-preimage of~$\DC_{1}$ in $\pobarmoduli[2,1]$ with  the locus
\bes
(\rom{2})\,\coloneqq\,\PP\omoduli[1,1]\times \pi_2(\PP\omoduliinc[1,2]{2,-2}).
\ees
Since we are working on the open part $\DC_1$ of the boundary divisor, the points $z$ and $q^-$ are not allowed to coincide (the degeneration when $z$ and $q^-$ coincide corresponds to stable curves that have three irreducible components, as depicted in figure~\ref{cap:rational-bridge}, and this case will be treated separately below). Thus the second factor is the locus in $\moduli[1,2]$ where $z$ and $q^-$ differ by a non-zero two-torsion point. According to the classification of \cite{boissymero},  this locus is irreducible. Together with the fact that $\PP\omoduli[1,1]$ is irreducible, this implies that the locus $(\rom{2})$ is irreducible and two-dimensional, and hence it gives an irreducible component of the boundary of $\PP\obarmoduliincp[2,1]{2}$.

We now study the $\pi_2$-preimages in $\PP\obarmoduliincp[2,1]{2}$ of the codimension two boundary strata of $\barmoduli[2,1]$. Consider the locus of irreducible stable curves with two nodes, see Figure~\ref{cap:2-nodal}.
\begin{figure}[ht]
\begin{tikzpicture}[scale=1,decoration={
    markings,
    mark=at position 0.5 with {\arrow[very thick]{>}}}]

\draw [] (-3, 0) coordinate (x1)
  .. controls ++(0:1) and ++(0: .3) .. ( -2.5, .6)
  .. controls ++(180:.3) and ++(180:1) .. ( -2, 0) coordinate  (z)
  .. controls ++(0: 1) and ++(0:.3) .. (-1.5, .6)
  .. controls ++(180: .3) and ++(180: 1) .. ( -1, 0);
\fill (z) circle (1pt);\node[above] at (z) {$z$};

\node [left] at (x1) {$X$};
\end{tikzpicture}
 \caption{An irreducible curve with two nodes} \label{cap:2-nodal}
\end{figure}
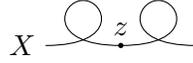
Since this stratum is obtained by adding a loop to the dual graph of curves in~$\DC_{0}$, by Lemma~\ref{lm:degenloop} its $\pi_2$-preimage has higher codimension in $\PP\obarmoduliincp[2,1]{2}$, hence does not give a boundary component.
\par
Next, consider the locus of stable curves as represented in Figure~\ref{cap:elliptic-rational-nodal}, where one component
is an elliptic curve and the other component is a rational nodal curve. The marked point $z$ is contained in one of the two components.
\begin{figure}[ht]
\begin{tikzpicture}[scale=1,decoration={
    markings,
    mark=at position 0.5 with {\arrow[very thick]{>}}}]
    \begin{scope}[xshift=-6cm]

\draw (-2.5,0) coordinate (x1)  .. controls (.5,.4) and (-3.5,.4) .. (-.7,0) ;
\draw (-.8,.2).. controls (-1.2,-.2) and (-.8,-.6) .. (-1,-1) coordinate (x2)
 coordinate [pos=.6] (z);
\fill (z) circle (1pt);\node[right] at (z) {$z$};

\node [left] at (x1) {$X_{1}$};\node [left] at (x2) {$X_{2}$};

\fill (0,0) coordinate (x1) circle (2pt);\node [right] at (x1) {$X_{1}$};
\fill (0,-1) coordinate (x2) circle (2pt); \node [right] at (x2) {$X_{2}$};
   \draw(x1) .. controls (.6,.4) and (-.6,.4) .. (x1) ;
 \draw[postaction={decorate}] (x1) -- (x2);
\end{scope}

\draw (-2.7,0) coordinate (x1).. controls (-2.2,.2) and (-1.8,-.2) .. (-.8,0) ;
\draw (-1.2,0.2) .. controls ++(270:.6) and ++(270:.4) .. (-.5,-.5)
                 .. controls ++(90:.4) and ++(90:.6) .. (-1.2,-1.2) coordinate (x2) coordinate [pos=.8] (z);
\fill (z) circle (1pt);\node[right] at (z) {$z$};

\node [left] at (x1) {$X_{1}$};\node [left] at (x2) {$X_{2}$};

\fill (0,0) coordinate (x1) circle (2pt);\node [right] at (x1) {$X_{1}$};
\fill (0,-1) coordinate (x2) circle (2pt); \node [right] at (x2) {$X_{2}$};
   \draw(x2) .. controls (.6,-1.4) and (-.6,-1.4) .. (x2) ;
 \draw[postaction={decorate}] (x1) -- (x2);
\end{tikzpicture}
 \caption{An elliptic curve union a rational nodal curve  and the level graph} \label{cap:elliptic-rational-nodal}
\end{figure}
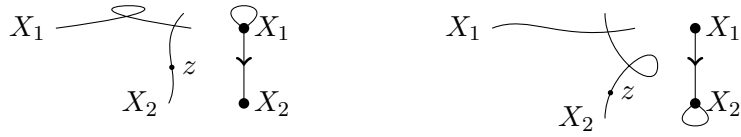
Clearly both dual graphs are obtained by adding a loop to the dual graph of curves in $\DC_{1}$. Hence Lemma~\ref{lm:degenloop} implies that the $\pi_2$-preimages of these strata do not give boundary components of $\PP\obarmoduliincp[2,1]{2}$.
\par
Another codimension two boundary stratum of $\bM_{2,1}$
parameterizes stable pointed curves that have a rational bridge~$X_3$ connecting two elliptic curves~$X_1$ and~$X_2$, where the marked point $z$ lies on~$X_{3}$ to ensure stability of $X$, see Figure~\ref{cap:rational-bridge}. It follows from Corollary~\ref{cor:pasSep} that this stratum does not intersect the projection of $\PP\obarmoduliincp[2,1]{2}$ under~$\pi_{2}$.
\par
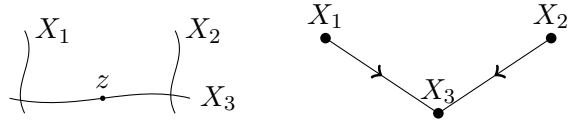
\begin{figure}[ht]
\begin{tikzpicture}[scale=1,decoration={
    markings,
    mark=at position 0.5 with {\arrow[very thick]{>}}}]

\draw (-4,0.1) coordinate (x1) .. controls (-3.8,-.3) and (-4.2,-.6) .. (-4,-1);
\draw (-2,0.1)  coordinate (x2).. controls (-1.8,-.3) and (-2.2,-.6) .. (-2,-1);

\draw (-4.2,-.8) .. controls (-3.3,-1) and (-2.6,-.6) .. (-1.8,-.8) coordinate (x3)
coordinate [pos=.5] (z);
\fill (z) circle (1pt);\node [above] at (z) {$z$};
\node [right] at (x1) {$X_{1}$};\node [right] at (x2) {$X_{2}$};\node [right] at (x3) {$X_{3}$};

\fill (0,0) coordinate (x1) circle (2pt);\node [above] at (x1) {$X_{1}$};
\fill (3,0) coordinate (x2) circle (2pt); \node [above] at (x2) {$X_{2}$};
\fill (1.5,-1) coordinate (x3) circle (2pt);
\node [above] at (x3) {$X_{3}$};

 \draw[postaction={decorate}] (x1) -- (x3);
 \draw[postaction={decorate}] (x2) -- (x3);

\end{tikzpicture}
 \caption{Two elliptic curves connected by a pointed rational curve} \label{cap:rational-bridge}
\end{figure}

The only remaining codimension two boundary stratum of $\bM_{2,1}$ parameterizes curves consisting of two irreducible components~$X_1$ and $X_2$ intersecting at two points, where~$X_1$ has genus one and $X_2$ has genus zero, and the marked point $z$ is contained in $X_2$ to ensure stability of $X$, see Figure~\ref{cap:banana}.
\par
\begin{figure}[ht]
\begin{tikzpicture}[scale=1,decoration={
    markings,
    mark=at position 0.5 with {\arrow[very thick]{>}}}]

\draw (-4.2,-0.4) .. controls (-3.3,-1.1) and (-2.6,-1.1) .. (-1.8,-0.4) coordinate (x3)coordinate [pos=.1] (x2)
coordinate [pos=.5] (z);
\fill (z) circle (1pt);\node [above] at (z) {$z$};\node [below] at (x2) {$X_{2}$};
\draw (-4.2,-0.6) coordinate (q1).. controls (-3.3,.1) and (-2.6,.1) .. (-1.8,-.6)  coordinate (q2) node [pos=.7,above]{$X_{1}$};

\node [left] at (q1) {$q_{1}$};
\node [right] at (q2) {$q_{2}$};

\fill (0,0) coordinate (x2) circle (2pt); \node [right] at (x2) {$X_{1}$};
\fill (0,-1) coordinate (x3) circle (2pt); \node [right] at (x3) {$X_{2}$};

 \draw[postaction={decorate}] (x2) ..controls (.2,-.5)  ..  (x3);
 \draw[postaction={decorate}] (x2) ..controls (-.2,-.5)  ..  (x3);
\end{tikzpicture}
 \caption{Two curves intersecting at two nodes  and the level graph} \label{cap:banana}
\end{figure}
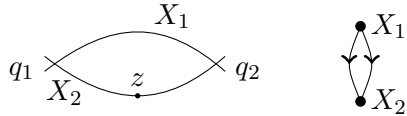
Suppose $\eta$ is a twisted differential compatible with a level graph $\overline{\Gamma}$ on $X$. It follows from Corollary~\ref{cor:oneminimum} that $X_1\succ X_2$. Since $\eta_1$ on $X_1$ has no zeros, it is holomorphic and nowhere vanishing, and then~$\eta_2$ has a double pole at each of the two nodes with a
double zero at $z$. In this case the global residue condition requires that $\Res_{q_1^-}\eta_2 +\Res_{q_2^-}\eta_2 = 0$, which follows from the residue theorem.
The associated stable differential~$\omega$ is equal to $\eta_1$ on $X_1$, and $\omega$ on $X_2$ is identically zero. Hence the locus of such $(X, \omega, z)$ in
$\pobarmoduliinc[2,1]{2}$ can be identified with
$$
 (\rom{3})\,\coloneqq\,\PP\omoduli[1,2]\times\moduli[0,3],
$$
where the second factor corresponds to the one-point space $\pi_2(\PP\omoduliinc[0,\{3\}]{2,-2,-2})$. Since the locus $(\rom{3})$ is clearly irreducible and two-dimensional, it gives an irreducible component of the boundary of $\PP\obarmoduliinc[2,1]{2}$.
\par
It remains to consider the boundary strata of $\barmoduli[2,1]$ that have codimension greater or equal to three. If there exists a boundary component of $\PP\obarmoduliinc[2,1]{2}$ mapping to such a stratum, then the fibers of $\pi_{2}$ over this stratum have positive dimension. By Lemma~\ref{lm:dimFibrePiDeux}, for this to happen, the graph $\overline{\Gamma}_{=0}$ must have at least two connected components. In our case, the components of top level would have to be of genus at least one. Hence such a stratum would have to parameterize curves with two irreducible components of  (arithmetic) genus at least one, neither of them containing the marked point $z$.  The only possibility is that such curves are degenerations of the curves in Figure~\ref{cap:rational-bridge}. Such degenerations can happen only if (at least) one of the two elliptic components degenerates to a rational nodal curve. However, in that case Corollary~\ref{cor:pasSep} still applies, hence the locus of such curves is disjoint from the image of $\PP\obarmoduliinc[2,1]{2}$ under $\pi_2$. This completes the proof of Proposition~\ref{prop:str(2)}.

\subsection{The incidence variety compactification of $\pomoduli[3](4)$}
The stratum $\pomoduli[3](4)$ has two irreducible components $\pomoduli[3](4)^{\rm hyp}$ and $\pomoduli[3](4)^{\rm odd}$ (see \cite{kozo1}). The hyperelliptic component parameterizes hyperelliptic curves with a marked Weierstra\ss{} point.
The odd-spin component parameterizes nonhyperelliptic curves whose canonical embeddings (as plane quartics) have a hyperflex.
Below we will describe the closure of the entire stratum.  The closure of
each connected component will follow from the results in the
forthcoming work~\cite{BCGGM2}.
\par
For a pointed stable differential $(X,\omega,z)$ in the boundary of $\pobarmoduliinc[3,1]{4}$, let $\overline\Gamma$ be a level graph on $X$ and $\eta$ be a compatible
twisted differential such that the associated pointed stable differential is $(X,\omega,z)$. By Corollary~\ref{cor:oneminimum}, $\overline\Gamma$ has a unique minimal vertex, and the corresponding irreducible component of $X$ contains the marked point~$z$.
\par
We first determine the $\pi_2$-preimages in $\pobarmoduliinc[3,1]{4}$ of the open boundary divisors of $\barmoduli[3,1]$. The computation over $\DC_0$ is similar to the previous case, and we obtain the corresponding locus as
$$
 ({\rm i})\,\coloneqq\, \PP\omoduliinc[2,\{3\}]{4,-1,-1},
$$
where the two simple poles are unordered. By~\cite[Theorem~1.2]{boissymero}, this locus has two connected components that can be distinguished by their parity. Both these components are four-dimensional and therefore each of them is an irreducible component of the boundary of $\pobarmoduliinc[3,1]{4}$. In fact,
by the discussion in~\cite[Section~4]{gendron} each component of $\PP\omoduliinc[2,\{3\}]{4,-1,-1}$ is a boundary component of the connected component of $\obarmoduliinc[3,1]{4}$ of the same parity.
\par
For $i=1,2$, denote by $\DC_{1,i}$ the open boundary divisor of $\barmoduli[3,1]$ parameterizing nodal unions of a genus one curve with a genus two curve, where the marked point is contained in the genus $i$ component, respectively. By Lemma~\ref{lm:minimumComp}, the component containing the marked point is of lower level compared to the unmarked component in $\overline\Gamma$. As in the previous section, we obtain that
$$ ({\rm ii})\,\coloneqq\,\pi_2^{-1}(\DC_{1,1})\=\PP\omoduliinc[2,1]{2}\times\pi_2(\PP\omoduliinc[1,2]{4,-4}), $$
$$({\rm iii})\,\coloneqq\, \pi_2^{-1}(\DC_{1,2})\=\pomoduli[1,1]\times\pi_2(\PP\omoduliinc[2,2]{4,-2}), $$
both of which are again purely four-dimensional. By~\cite{boissymero}, each of these two loci has two connected components: a hyperelliptic and a non-hyperelliptic one, and each of these is then an irreducible component of the boundary of the corresponding component of $\pobarmoduliinc[3,1]{4}$.

Next we investigate the situation over the codimension two boundary strata of $\barmoduli[3,1]$. By Lemma~\ref{lm:degenloop}, the $\pi_2$-preimage of the locus of irreducible curves with two nodes has codimension bigger than one, hence it does not provide a boundary component of $\pobarmoduliinc[3,1]{4}$.
\par
Let $\Theta_{0}\subset \barmoduli[3,1]$ denote the locus of curves with two components $X_0$ and $X_2$ intersecting at two nodes, where $X_0 \cong \PP^1$, $X_2$ has genus two, and $X_0$ contains the marked point. In this case $X_2 \succ X_0$ in $\overline\Gamma$, hence the stable differential $\omega$ is identically zero on~$X_0$. For $\omega$ on $X_2$, it is given by $\eta_2$, and there are two possibilities. If $\eta_2$ has simple zeros at both nodes, then $\eta_0$ on~$X_0$ has triple poles at both nodes, and hence $(X, \omega, z)$ is parameterized by the irreducible locus
$$
  (\rm{iv})\,\coloneqq\,\PP\omoduliinc[2,\{2\}]{1,1}\times \pi_2(\PP\omoduliinc[0,\{3\}]{4,-3,-3}).
$$
On the other hand, if $\eta_2$ has a double zero at one node and is regular at the other, then $(X, \omega, z)$ is parameterized by the irreducible locus
$$
 (\rm{v})\,\coloneqq\,\PP\omoduliinc[2,2]{2,0}\times \pi_2(\PP\omoduliinc[0,3]{4,-4,-2}).
$$
Both loci are four-dimensional irreducible, and thus
$$ \pi_2^{-1}(\Theta_0) \= (\rm{iv}) \cup (\rm{v}), $$
providing (irreducible) boundary components of $\pobarmoduliinc[3,1]{4}$.
\par
Let $\Theta_{1}\subset \barmoduli[3,1]$ denote the locus of curves with two components $X_1$ and $X_2$ intersecting at two nodes, where
both components have genus one and $X_1$ contains the marked point. A similar analysis as in the case of $\Theta_0$ implies that
the pointed stable differential $(X, \omega, z)$  in $\pobarmoduliinc[3,1]{4}\cap \pi_{2}^{-1}(\Theta_1)$ then lies in the locus
$$
(\rm{vi})\coloneqq \PP\omoduli[1,\lbrace2\rbrace]\times\pi_2(\PP\omoduliinc[1,\{3\}]{4,-2,-2}),
$$
which has two irreducible components, both of dimension four, hence providing boundary components of $\pobarmoduliinc[3,1]{4}$. As remarked in \cite[Section~7]{gendron}, each of these components two components lies in the boundary of both components of $\pobarmoduliinc[3,1]{4}$.
\par
For the codimension two boundary strata of $\barmoduli[3,1]$ parameterizing curves with one separating node and one non-separating node, they are contained in the closure of $\DC_{1,i}$ for $i=1,2$, by pinching a non-separating loop on an irreducible component of curves in $\DC_{1,i}$. This is the operation described by Lemma~\ref{lm:degenloop}, and thus the corresponding boundary loci in $\pobarmoduliinc[3,1]{4}$ are contained in $(\rm{ii})$ and $(\rm{iii})$ defined above. Therefore, the $\pi_2$-preimages of such codimension two boundary strata do not give rise to any boundary components of $\PP\obarmoduliinc[3,1]{4}$.
\par
Let us consider the other codimension two boundary strata of $\barmoduli[3,1]$ parameterizing curves with two separating nodes. There are several possibilities. If all three components $X_1$, $X_2$, and $X_3$ of the curve $X$ have genus one and the marked point lies in the middle component $X_2$, then the \twd~$\eta$ lies in the locus contained in
$$
\omoduli[1,1]\times\omoduliinc[1,\lbrace3\rbrace]{4,-2,-2}\times\omoduli[1,1],
$$
which is cut out by the global residue condition that $\eta_2$ has zero residues at both nodes on the middle component $X_2$. Here we encounter the situation that the map $\pi_2$ is no longer finite. In this case the full order is $X_1\asymp X_3\succ X_2$. Since $X_1$ and $X_3$ are both of top level,
by Lemma~\ref{lm:dimFibrePiDeux} the fibers of $\pi_2$ are one-dimensional, which record the projectivization of the pair of scaling factors for $\eta_1$ and $\eta_3$. It follows that the associated pointed stable differential $(X, \omega, z)$
lies in the codimension one locus
$$
(\rm{vii})\subsetneq\PP(\omoduli[1,1]\times\omoduli[1,1])\times\pi_2(\omoduliinc[1,\lbrace3\rbrace]{4,-2,-2})
$$
given by the condition that $\eta_2$ has zero residues. By the results of~\cite{boissy} it follows that the locus $(\rm{vii})$ is non-empty, and thus its codimension in this product is equal to one. The locus (\rm{vii}) is thus purely four-dimensional; the number of its irreducible components is not known, but each such irreducible component gives a boundary component of $\PP\omoduliinc[3,1]{4}$.
\par
On the other hand if the marked point $z$ lies in $X_1$, then the full order is given by  $X_3\succ X_2 \succ X_1$.
Since $X_3$ is the unique irreducible component of top level, the map $\pi_2$ is finite. The \twd~$\eta$ lies in the locus
$$
 \omoduli[1,1]\times\omoduliinc[1,2]{2,-2}\times\omoduliinc[1,2]{4,-4}
$$
and the associated pointed stable differential lies in the locus
$$
 \PP\omoduli[1,1]\times\pi_2(\omoduliinc[1,2]{2,-2})\times\pi_2(\omoduliinc[1,2]{4,-4}),
$$
which is three-dimensional, and thus not a boundary component of $\PP\obarmoduliinc[3,1]{4}$.

Suppose now that one of the three irreducible components of $X$ has genus zero. By stability of $X$, the $\PP^1$ component
must be the middle component, and it contains the marked point $z$. It follows from Corollary~\ref{cor:pasSep} that this locus is disjoint with the image of $\PP\obarmoduliinc[3,1]{4}$ under $\pi_2$. This completes our analysis over the codimension two boundary strata of $\barmoduli[3,1]$.
\par
We now investigate over which boundary strata of $\barmoduli[3,1]$ of codimension greater or equal to three there may be boundary components of $\PP\obarmoduliinc[3,1]{4}$. As there are numerous cases, we first make some general observations. If a boundary stratum of
$\barmoduli[3,1]$ parameterizes stable curves that have an internal node, then by Lemma \ref{lm:degenloop}, its $\pi_2$-preimage does not give rise to
a boundary component of $\PP\obarmoduliinc[3,1]{4}$. Therefore, it suffices to consider the case when all irreducible components of the stable curves are smooth. Furthermore, if the $\pi_2$-preimage of a boundary stratum of $\barmoduli[3,1]$ of codimension greater or equal to three
gives a boundary component of $\PP\obarmoduliinc[3,1]{4}$, then the fibers of~$\pi_2$ over this stratum have positive dimension.
\par
Let us first determine which codimension three strata of the boundary of $\barmoduli[3,1]$ are contained in $\pi_2(\PP\obarmoduliinc[3,1]{4})$. As before, Lemma~\ref{lm:degenloop} allows us to assume that no edge of the dual graph~$\Gamma$ is a loop. Moreover, if $h^1(\Gamma)\ge 3$, the stratum has codimension at least 4, and can be discarded. The case of $h^1(\Gamma)=0$, i.e. the $\Gamma$ being a tree, is also easily ruled out. If $h^1(\Gamma)=1$, then $\Gamma$ must be the triangle, with two curves of genus one and one $\PP^1$ containing the marked point. In this case, the admissible twisted differentials gives a non-trivial torsion condition on the marked points on one elliptic curve. Hence the whole stratum is not in the image $\pi_2(\PP\obarmoduliinc[3,1]{4})$. Finally, if $h^{1}(\Gamma)=2$,  the only possibility is the
locus parameterizing curves with two irreducible components intersecting at three nodes, where one component has genus one, and the other is $\PP^1$, containing the marked point $z$. Moreover, the $\PP^1$ must be of lower level in~$\overline{\Gamma}$.
The global residue condition on $\eta$ then holds automatically by the residue theorem. In this case the associated pointed stable differentials are parameterized by
\[
(\rm{viii})\,\coloneqq\,\pomoduli[1,3]\times\pi_2(\PP\omoduliinc[0,\{4\}]{4,-2,-2,-2}).
\]
 This locus is irreducible and four-dimensional, hence it gives an irreducible boundary component of $\pobarmoduliinc[3,1]{4}$.
\par
Let us now look at the cases when $\pi_2$ has positive dimensional fibers. According to Lemma~\ref{lm:dimFibrePiDeux} the level graph $\overline{\Gamma}$ has at least two top level components. Both are smooth and of genus at least one.
If some two top level components are connected by a node, the presence of this node decreases the codimension of such a locus by one, and thus such a locus must be properly contained in a higher-dimensional component of the boundary.
Moreover by Lemma~\ref{lm:minimumComp}, the unique minimal component~$X_v$ of
$\overline{\Gamma}$ contains the marked point $z$. If $X_v$ has genus at least one, then the curve~$X$ is of compact type (which, recall, means that the dual graph is a tree). The only such stable curve $X$ is given by a $\PP^1$ component, denoted by~$X_3$, intersecting three elliptic components
$X_1$, $X_2$, and $X_v$, with the full order given by $X_1\asymp X_2\succ X_3\succ X_v$. Then a compatible \twd~$\eta$ on $X_3\cong \PP^1$ would have a double zero at the node joining $X_v$ and two double poles with zero residues at the other two nodes, which is impossible by Lemma~\ref{lm:gzeroontresidu}.
\par
Now assume that the minimal component $X_v$ has genus zero. Since $X_v$ contains~$z$, in order to be stable, $X_v$ must have at least two nodes. If it has precisely two nodes, then $X\setminus X_v$ has to be connected, for otherwise the global residue condition would imply that~$\eta_v$ has zero residues at the two nodes and a unique zero at $z$, which is impossible by Lemma~\ref{lm:gzeroontresidu}. In this case the only possibility of such $X$ and $\overline{\Gamma}$ is $X_1\asymp X_2\succ X_v$, where~$X_1$ and $X_2$ are of genus one, and $X_v\cong \PP^1$. Then there does not exist a compatible \twd~$\eta$, for otherwise $\eta_{1}$ on the elliptic component~$X_1$ would have a simple pole at the node joining $X_2$ and a simple zero at the node joining~$X_v$, impossible.
\par
Now consider the case when $X_v\cong\PP^1$ contains more than two nodes. If there are three top level components, all of them have to be elliptic curves and have to intersect~$X_v$. Then the global residue condition implies that any compatible \twd~$\eta$ has zero residues at the poles of $\eta_{v}$ on $X_v$. By Lemma~\ref{lm:gzeroontresidu} such $\eta$ does not exist. If $X$ has only two top level components, then both of them have to be elliptic curves, and the dual graph of $X$ contains a circle. Hence all but two of the nodes contained in $X_v$ are separating. In order
for $X$ to be stable, at each separating node the attached curve must be of genus one. Thus there can be only one separating node. In this case the only possibility is that $X_v$ intersects $X_1$ at one point and $X_v$ intersects $X_2$ at two points, where both $X_1$ and $X_2$ are of genus one. Let $S$ be the locus in $\omoduliinc[0,\{4\}]{4,-2,-2,-2}$ parameterizing meromorphic differentials whose residues are zero at the first pole. In this case
the \twds~are parameterized by the locus
$$
 \omoduli[1,1]\times\omoduli[1,2]\times S.
$$
The associated pointed stable differentials $(X, \omega, z)$ are parameterized by
\[
({\rm ix})\,\coloneqq\,\PP(\omoduli[1,1]\times\omoduli[1,2])\times \pi_2(S),
\]
which is four-dimensional. We claim that this locus is connected, and thus it gives an irreducible component of the boundary of $\pobarmoduliinc[3,1]{4}$. Indeed, a differential in $\omoduliinc[0,\{4\}]{4,-2,-2,-2}$ can be written in the form $\omega=\tfrac{dz}{z^2 (z-1)^2(z-a)^2}$ with $a\in\CC\setminus\left\{0,1\right\}$. The residue of $\omega$ at $0$ is $\Res_{0}(\omega)=\tfrac{2(a+1)}{a^{3}}$, which is zero if and only if $a=-1$. This shows that $S$ and hence $({\rm ix})$ is connected.
\par
In summary, we have proved the following result.
\begin{prop}
The boundary of $\pobarmoduliinc[3,1]{4}$ is the union of the closures of the loci $(\rm{i}),$ $(\rm{ii})$, $(\rm{iii})$, $(\rm{iv})$, $(\rm{v})$, $(\rm{vi})$, $(\rm{vii})$,
$(\rm{viii})$, and $(\rm{ix})$ defined above. Furthermore, for each of these cases except $(\rm{vii})$, the irreducible components of the corresponding locus are enumerated above, so that in total we see that the number of irreducible components of the boundary of $\pobarmoduliinc[3,1]{4}$ is equal to $12$ plus the number of irreducible components of the locus $(\rm{vii})$. Moreover, the loci (\rm{vii}) and (\rm{ix}) are contracted under the map $\pi_{2}$.
\end{prop}

We remark that in \cite{gendron} the third author analyzed the incidence variety compactification of $\PP\omoduliinc[3,1]{4}$ over the boundary divisors
$\DC_{0}$ and $\DC_{1}$ of $\barmoduli[3]$. In \cite{chen} the second author characterized the Deligne-Mumford compactification of $\pomoduli[3,1](4)$ in $\barmoduli[3,1]$ for curves with at most two nodes. In \cite{fapa} Farkas and Pandharipande also studied the Deligne-Mumford compactification of $\pomoduli[3,1](4)$ for several cases of stable curves with a fixed dual graph. In~\cite{hu} Hu obtained an explicit modular form defining the locus of hyperflexes in $\barmoduli[3]$, that is, the image of $\pomoduli[3](4)^{\rm{odd}}$ in $\barmoduli[3]$.

\subsection{Some strata of meromorphic differentials in genus one}

In this section we describe the boundary of the incidence variety compactification of $\omoduli[1](m,-m)$ and $\omoduli[1](m_{1},m_{2},m_{3})$ in genus one.
\par
First consider the stratum $\PP\omoduli[1](m,-m)$ for $m\geq2$. If a stable curve in $\barmoduli[1,2]$ has a separating node, then it does not lie in the image $\pi_2(\PP\obarmoduliinc[1,\{2\}]{m,-m})$. Otherwise the subcurve on one side of the separating node would have arithmetic genus zero, and by stability it would have to contain both marked points. Then a compatible \twd~on this component would have a zero and a pole at the two marked points, respectively, and have another pole at the node, where both poles would have zero residues by the global residue condition and the residue theorem. But this is impossible by Lemma~\ref{lm:gzeroontresidu}. Thus we only need to consider the open boundary divisor~$\DC_0$ of $\barmoduli[1,2]$ parameterizing irreducible rational nodal curves and
the locus $\Theta$ parameterizing two~$\PP^1$ components intersecting at two nodes, each containing a marked point.

The $\pi_2$-preimage of $\DC_0$ in $\PP\obarmoduliinc[1,\{2\}]{m,-m})$ can be identified with the locus
\[
({\rm j})\subset\PP\omoduliinc[0,\{4\}]{m,-1,-1,-m},
\]
where the two simple poles are unordered, and the meromorphic differentials parameterized in $(\rm{j})$ has zero residues at the pole of order $m$. Equivalently, $(\rm{j})$ parameterizes marked rational curves $(\PP^1,z,p,q^+,q^-)$ such that the degree $m$ cover
from $\PP^1$ to $\PP^1$ induced by $mz - mp$ contains $q^+$ and $q^-$ in the same fiber.

Now we consider the $\pi_2$-preimage of $\Theta$ in $\PP\obarmoduliinc[1,\{2\}]{m,-m}$. Suppose a pointed stable differential $(X, \omega, z, p)$ lies in the $\pi_2$-preimage of $\Theta$, associated to a \twd~$\eta$ compatible with a level graph $\overline{\Gamma}$. The irreducible component of $X$ containing the marked zero $z$ must be of lower level in $\overline{\Gamma}$.  Since there is a marked pole on the top level component, there is no global residue condition in this case. Suppose the zeros of $\eta$ on the top level component are of orders $a$ and $m-2-a$ at the two nodes. Then such $(X, \eta)$ is parameterized by
$$
 \omoduliinc[0,3]{a,m-2-a,-m}\times\omoduliinc[0,3]{m, -2-a, a-m}.
$$
The associated pointed stable differential $(X, \omega, z, p)$ is parameterized by the locus
\[
({\rm jj}_{a})\,\coloneqq\, \PP\omoduliinc[0,3]{a,m-2-a, -m}\times\moduli[0,3],
\]
which is a single point.

In summary, we have obtained the following result.

\begin{prop}\label{prop:genreundeuxpts}
For  $m\geq 2$ the boundary of $\PP\obarmoduliinc[1,2]{m,-m}$ consists of the union of the loci
$({\rm j})$ and $({\rm jj}_a)$ for $a=0,\ldots, \lfloor \frac{m-2}{2}\rfloor$ defined above.
\end{prop}

Next we study the strata $\omoduli[1](m_{1},m_{2},m_{3})$, where $m_{1}\geq m_{2}>0$, $m_{3}<0$ and $m_1 + m_2 + m_3 = 0$ (so in particular we must have $m_3\le -2$).
The $\pi_2$-preimage of $\DC_0$ in $\PP\obarmoduliinc[1,\{3\}]{m_{1},m_{2},m_{3}}$ can be identified with the locus
\[
({\rm J})\subset\PP\omoduliinc[1,\{5\}]{m_{1},m_{2},m_{3},-1,-1}
\]
defined by the condition that the residue is zero at the pole of order $m_{3}$.

Suppose a pointed stable differential $(X, \omega, z_1, z_2, z_3)$ lies in the $\pi_2$-preimage of the open boundary divisor in $\barmoduli[1,3]$
parameterizing curves with a separating node. If all three marked points are on the $\PP^1$ component of $X$, any associated \twd~$\eta$ must be holomorphic and nowhere vanishing on the elliptic component of $X$. We thus obtain the locus of such $(X, \omega, z_1, z_2, z_3)$ as
 \[
({\rm JJ})\subset\PP\omoduli[1,1]\times\pi_{2}(\PP\omoduliinc[0,4]{m_{1},m_{2},m_{3},-2})
\]
where the residue at the last pole of order two is required to be zero. The locus $({\rm JJ})$ is one-dimensional, hence it gives rise to boundary components of $\PP\obarmoduliinc[1,\{3\}]{m_{1},m_{2},m_{3}}$.

If the $\PP^1$ component contains the two marked zeros $z_1$ and $z_2$, and the elliptic component contains the marked pole $z_3$, then
such $(X, \omega, z_1, z_2, z_3)$ is parameterized by the locus
\[
({\rm JJJ}_{m_{3}})\,\coloneqq\,\PP\omoduliinc[1,2]{-m_{3},m_{3}} \times\pi_2(\PP\omoduliinc[0,3]{m_1, m_2, m_3-2}).
\]
(with a further symmetrization $\PP\omoduliinc[0,\{3\}]{m_1, m_2, m_3-2})$ if $m_1=m_2$).
If the zero of order $m_{i}$ lies on the elliptic component for $i=1$ or $2$, then the $\PP^1$ component must be of higher level. Since it contains a marked pole, by Theorem~\ref{thm:main} there is no global residue condition imposed on an associated \twd~$\eta$. Hence such $(X,\eta)$ are parameterized by
$$
 \omoduliinc[1,2]{m_{i},-m_{i}}\times\omoduliinc[0,3]{m_{3-i},m_{3},-2+m_{i}}.
$$
The associated $(X, \omega, z_1, z_2, z_3)$ are parameterized by the locus.
\[
({\rm JJJ}_{m_{i}})\,\coloneqq\,\pi_2(\PP\omoduliinc[1,2]{m_{i},-m_{i}})\times\PP\omoduliinc[0,3]{m_{3-i},m_{3},-2+m_{i}}.
\]

Now suppose $(X, \omega, z_1, z_2, z_3)$ lies in the $\pi_2$-preimage of the locus of curves with two $\PP^1$ components intersecting at two nodes. If both marked zeros lie on the same component, suppose the zeros of $\eta$ at the two nodes on the other component are of orders $a$ and $-m_3-2-a$, respectively,
for $0\leq a \leq \lfloor \frac{-m_3-2}{2}\rfloor$. The associated pointed stable differentials are parameterized by the locus
\[
({\rm JV}_{m_{3}}^{a})\,\coloneqq\,\PP\omoduliinc[0,3]{a, -m_{3}-2-a, m_{3}}\times\pi_{2}(\PP\omoduliinc[0,4]{m_{1}, m_{2}, -a-2, a+m_3})
\]
(with a further symmetrization if $-a-2=a+m_3$).
Similarly if one component of $X$ contains only the marked zero~$z_i$ for $i=1$ or $2$, suppose the zero orders of $\eta$ at the nodes of the other component are $a$ and $m_i-2-a$, respectively, for $0\leq a \leq \lfloor \frac{m_i-2}{2}\rfloor$. Then the associated pointed stable differentials are parameterized by the locus
\[
({\rm JV}_{m_{i}}^{a})\,\coloneqq\,\PP\omoduliinc[0,3]{m_i, -a-2, a-m_i}\times\pi_{2}(\PP\omoduliinc[0,4]{m_{3-i}, a, m_i-a-2, m_{3}}).
\]

All the loci $({\rm JJJ}_{m_{i}})$ and $({\rm JV}_{m_{i}}^{a})$ for $i=1,2,3$ are one-dimensional, hence they give boundary components of $\PP\obarmoduliinc[1,3]{m_{1},m_{2},m_{3}}$.

We claim that there are no further boundary components of $\PP\obarmoduliinc[1,3]{m_{1},m_{2},m_{3}}$. By
Lemma~\ref{lm:degenloop} we only need to consider the cases when each irreducible component of $X$ is smooth. Over any boundary stratum of codimension three in $\barmoduli[1,3]$, in order to obtain a boundary component of $\PP\obarmoduliinc[1,3]{m_{1},m_{2},m_{3}}$, the $\pi_2$-fibers need to be positive dimensional, hence any associated \twd~has at least two top level components in a compatible level graph. Each such top level component needs to carry a marked pole, which is impossible because there is only one marked pole. The only remaining codimension two boundary stratum of $\barmoduli[1,3]$ parameterizes stable curves consisting of three components $X_0$, $X_1$, and $X_2$, where $X_0\cong X_2 \cong \PP^1$, $X_1$ is of genus one, $X_0$ intersects $X_1$ and $X_2$ each at a node, $X_0$ contains a marked point, $X_2$ contains two marked points, and $X_1$ is unmarked. In this case $X_1\succ X_0$ for any compatible $\eta$ in a level graph. The global residue condition implies that $\eta$ on $X_0$ has zero residues at the poles, hence it is a meromorphic differential on $\PP^1$ that has a unique zero and two poles with zero residues. But this is impossible by Lemma~\ref{lm:gzeroontresidu}.

In summary, we have proved the following result.
\begin{prop}
The boundary of $\PP\obarmoduliinc[1,\{3\}]{m_{1},m_{2},m_{3}}$ with $m_{1}\geq m_{2}>0$ and $m_{3}<0$ is the union of the closures of the loci
$({\rm J})$, $({\rm JJ})$, $({\rm JJJ}_{m_{i}})$, and $({\rm JV}_{m_{i}}^{a})$ for $i=1,2,3$ defined above.
\end{prop}

\subsection{Weierstra\ss{} divisors on a binary curve}
\label{subsec:weierstrass}

In this section, we apply our main result to study the limits of  Weierstra\ss{} points on so-called {\em binary curves}. Recall that a binary curve is a curve that consists of two smooth irreducible components intersecting in a number of nodes.
Suppose $X$ is a smooth and connected curve of genus~$g$. Recall that a point $z_0 \in X$ is a Weierstra\ss{} point if $h^0(X, gz_0) \geq 2$. Equivalently, this is to say that there exists an effective canonical divisor
$m_0 z_0 + m_1 z_1 + \cdots + m_n z_n$ on~$X$ such that $m_0 \geq g$. We call such a divisor a
\emph{Weierstra\ss{} divisor of type $\mu$}, where $\mu = (m_0, \ldots, m_n)$ with $m_0 \geq g$. Note that Weierstra\ss{} divisors of type $\mu$ correspond to pointed differentials in the stratum $\PP\omoduli[g, n+1](\mu)$.

Degenerations of Weierstra\ss{} points and divisors have been studied intensively, see e.g.~\cite{eihaweier} for limits of Weierstra\ss{} points on nodal curves of compact type. Since the theory of limit linear series used in \cite{eihaweier} does not work in full generality for curves of non-compact type, it is more complicated to study limits of Weierstra\ss{} points on a curve of non-compact type. As a consequence of Theorem~\ref{thm:main}, in what follows we give an efficient description for the degeneration of Weierstra\ss{} divisors on certain binary curves, which recovers a main case treated by \cite{esme}.

Let $(X, \omega, z_0, \ldots, z_n)$ be a stable pointed differential of genus $g$, where $X$ consists of two genus one curves $E_1$ and $E_2$ attached at $g-1$ nodes $q_1, \ldots, q_{g-1}$. We want to study when $(X, \omega, z_0, \ldots, z_n)$ is contained in the incidence variety compactification of $\PP\omoduli[g, n+1](\mu)$. Without loss of generality, suppose $E_1$ contains $z_0$. Suppose $\overline{\Gamma}$ is a level graph on $X$ and $\eta$ is a compatible \twd~such that it is associated with $(X, \omega, z_0, \ldots, z_n)$. Let $\eta_i$ be the restriction of $\eta$ on $E_i$ for $i= 1, 2$.

If $E_1 \succcurlyeq E_2$, then $\ord_{q_j} (\eta_{1}) \geq -1$ for all $j = 1,\ldots, g-1$. Note that
$${\rm div}(\eta_1) \=  \sum_{z_i \in E_1} m_i z_i  + \sum_{j=1}^{g-1} \ord_{q_j} (\eta_{1}) \cdot q_j. $$
Since $z_0 \in E_1$, it follows that
$$\deg ({\rm div}(\eta_1)) \geq m_0 - (g-1) > 0 \= \deg K_{E_1},$$
which contradicts that $\eta_1$ is a (possibly meromorphic) differential on the genus one curve $E_1$.

Consider the remaining case $E_1 \prec E_2$ in $\overline{\Gamma}$. Then $\ord_{q_j} (\eta_{2}) \geq 0$ for all $j$. Since
$${\rm div}(\eta_2) \=  \sum_{z_i \in E_2} m_i z_i  + \sum_{j=1}^{g-1} \ord_{q_j} (\eta_{2}) \cdot q_j, $$
in order to ensure $\deg ({\rm div}(\eta_2)) = 0$, the only possibility is that $\ord_{q_j} (\eta_{2}) = 0$ for all $j$ and that
$E_2$ does not contain any marked point $z_i$. Therefore, we conclude that
$$ {\rm div}(\eta_1) \= \sum_{i=0}^n m_i z_i - 2 \sum_{j=1}^{g-1} q_j, $$
hence necessarily $\sum_{i=0}^n m_i z_i \sim 2 \sum_{j=1}^{g-1} q_j$ on $E_1$, and $\eta_2$ is holomorphic and nowhere vanishing on $E_2$.

Conversely if the above linear equivalence relation holds, the global residue condition imposed on $\eta$ follows from the residue theorem on $E_1$. In that case, take $\omega|_{E_1} = 0$ and $\omega|_{E_2} = \eta_2$. Then $(X, \omega, z_0, \ldots, z_n)$ lies in the incidence variety compactification of $\PP\omoduli[g, n+1](\mu)$ by Theorem~\ref{thm:main}.

Applying the above method, it is clear that our compactification can be used to determine the degeneration of Weierstra\ss{} divisors on stable pointed curves of any given type.

\section{Plumbing and gluing differentials}\label{sec:proof}

\subsection{Proof of the main theorem: conditions are necessary}
First we show how to obtain a collection of meromorphic differentials associated
to any boundary point of the incidence variety compactification. Suppose $f:\famcurv \to\Delta$ is a family of pointed curves over a sufficiently small disk~$\Delta$ whose central fiber is nodal, along with (possibly meromorphic) differentials of type $\mu = (m_1, \ldots, m_n)$ on the fibers $\famcurv^\ast\to\Delta^\ast$ over the punctured disk $\Delta^\ast\coloneqq\Delta\setminus\lbrace 0\rbrace$. We denote by $X_t$ the fiber of~$f$ over~$t\in \Delta^\ast$, and by $X=X_0$ the central fiber. The marked points in the family are recorded as sections $\seczero_1,\ldots,\seczero_n:\Delta\to\famcurv$ that correspond to the zeros and poles of the differentials. The family of stable differentials on $\calX$ can then be viewed as a section
$\famomega:\Delta^\ast\to f_*\dualsheave[\famcurv^\ast/\Delta^\ast](-\seczero_{\rm pol})$ where
$\seczero_{\rm pol} = \sum_{m_i < 0} m_{i}\seczero_{i}$ is the polar part
of $\mu$.
\par
The result below shows that for any irreducible component $X_v$ of the central fiber, the family of differentials $\famomega$ over $\Delta^\ast$ can be rescaled by a power of $t$ so that the limit will exist and will not be identically zero on $X_v$.
\begin{lm} \label{le:scalinglimit}
Let $\left(f:\famcurv \to\Delta,\,\famomega,
\seczero_{1},\ldots,\seczero_{n}\right)$
be a family of pointed differentials over a disk $\Delta$ described above. Then for every irreducible component $X_v$ of the central fiber $X$, there exists a
unique integer $\ell_v$ such that for a generic section
$s:\Delta^{\ast}\to\famcurv$ with $\bar{s}(0)\in X_v$ there exists a nonzero limit
\begin{equation} \label{eq:scalinglimit}
\eta_v \,\coloneqq \,\lim_{t\to 0} \, t^{\ell_v}\, \famomega\left(t,s(t)\right) \neq 0\,.
\end{equation}
\end{lm}
We refer to the exponent $\ell_v$ in the lemma as the {\em scaling parameter}
of the component~$X_v$ and to the $\eta_v$ as the {\em scaling limits}. If one re-parameterizes the base disk, the scaling limits on the irreducible components with the same value of the scaling parameter will be multiplied by the same nonzero scalar.
\begin{proof}
Since $\mathcal{X}$ is normal it is regular in codimension-one, so the local ring of $\mathcal{O}_{\mathcal X}$ at $X_v$ is
a discrete valuation ring (DVR) (see e.g.,~\cite[Theorem 9.3]{atmcd}).
As the  fiber $f^{-1}(0)$ is reduced, $t$ is a uniformizing parameter
for this DVR. It follows that any section $s$ of a line bundle over $\mathcal{X}$ is of the form
$t^{a_v} g$, where $g$ is another local holomorphic section of the same line bundle over $\mathcal{X}$
which does not vanish on $X_v$. In these terms
$a_v$ is just the order of vanishing of $s$ along $X_v$, and $\ell_v$ in the claim is equal to $- a_v$.
\end{proof}
We now show that the collection of these scaling limits gives a \twd\ satisfying all the conditions.
\begin{proof}[Proof of Theorem~\ref{thm:main}: conditions are necessary]
Suppose a pointed stable differential $(X,\omega,z_1,\ldots,z_n)$ lies
in the incidence variety compactification of $\PP\omoduli(m_1,\ldots,m_n)$.
Take a family of pointed stable differentials over a disk such that
$(X,\omega,z_1,\ldots,z_n)$ is the central fiber and choose an arbitrary
lifting of the family of differentials over $\Delta^\ast$, defined up to scale,
to an actual family of differentials. Define the collection of meromorphic
differentials $\lbrace\eta_v\rbrace$ on the irreducible components~$X_v$
of~$X$ as the scaling limits given by
Lemma~\ref{le:scalinglimit}. Then the scaling parameters $\ell_v$ can be used as a level function on the vertices of the dual graph $\Gamma$ of $X$, and thus define
a level graph $\overline\Gamma$. Note that multiplying the section $\famomega$ globally by a power of
$t$ results in adding a constant to all the scaling parameters $\ell_v$. To ensure that
the last statement of condition~(i) in Theorem~\ref{thm:main} holds, we perform such
a global rescaling of $\famomega$ by a power of $t$ so that the maximum scaling parameter (corresponding to the top
components) is equal to $\ell_v=0$, and all other scaling parameters are thus negative. From now on we fix such a choice, and show
that the collection $\eta \coloneqq \{\eta_v\}$ gives a twisted differential compatible with~$\overline{\Gamma}$.
\par
First, note that in the incidence variety compactification
the prescribed zeros and poles are marked points, which are not allowed to collide
with a node of the central fiber. Thus the vanishing orders of $\eta$ at all the marked points~$z_i$ are the same as in the family, so $\eta$ satisfies condition~(0).
\par
Next we prove conditions~(1) and (2). Suppose $q$ is a node joining two irreducible
components~$X_{v_1}$ and~$X_{v_2}$ of the central fiber~$X$. Take a neighborhood~$U$ of the
node~$q$ and choose local coordinates~$x, y$ such that $U \cong V(xy-t^a)$ for some $a \in \mathbb{N}$.
We pick a differential~$\xi$ that generates $\omega_{\famcurv/\Delta}$ and that vanishes
nowhere on $U$, for example
$$\xi =\frac{dx}{x} = -\frac{dy}{y}.$$
To establish conditions (1) and (2), we use that the orders of vanishing and residues of $\famomega$
in each branch can be detected by the turning numbers and periods of loops in these branches.  We
define two families of loops $\gamma_t^1, \gamma_t^2 \colon [0, 2 \pi] \to X_t$ by
\begin{align*}
  \gamma_t^1(\tau) &= (r e^{i\tau}, t^a / re^{i\tau}, t^a) \\
  \gamma_t^2(\tau) &= (t^a/r e^{i\tau}, re^{i\tau}, t^a) \\
\end{align*}
for some sufficiently small $r>0$.  The orders of vanishing are the turning numbers
\begin{equation}
  \label{eq:turning_number}
 \ord_q \eta_{v_j} = \Turn_{\eta_j} (\gamma_0^j) - 1,
\end{equation}
where the turning number $\Turn_{\eta} (\gamma)$ of a curve $\gamma$ with respect to a one-form
$\eta$ is the degree of the Gauss map
\begin{equation*}
  G(\tau)= \frac{ \eta( \gamma'(\tau))}{|\eta( \gamma'(\tau))|}.
\end{equation*}
  As the
$\eta_i$ are rescalings of $\famomega$, we have
\begin{equation*}
  \Turn_{\eta_j} (\gamma_0^j) = \lim_{t\to 0} \Turn_{\famomega(t)}( \gamma_t^j).
\end{equation*}
 For $t$ nonzero, the curve $\gamma_t^1$ is homotopic to $-\gamma_t^2$ in an
annular domain of $X_t$ in which the form $\famomega$ is nonzero, so $\gamma_t^1$ and $\gamma_t^2$
have opposite turning numbers.  It then follows from \eqref{eq:turning_number} that
\begin{equation*}
  \ord_q \eta_{v_1} + \ord_q \eta_{v_2} = -2,
\end{equation*}
which is condition (2).
\par
When these are both simple poles, the same argument gives that the
$\gamma_t^j$ have opposite periods for $t\neq 0$, and by taking a limit, we see that the
$\gamma_0^j$ have opposite periods.  The $\eta_{v_j}$ then have opposite residues at $q$, which is
condition (1).
\par
Alternatively, conditions (1) and (2) can be explained from the viewpoint of limit linear series (see \cite{chen, fapa}). The equation
$xy = t^a$ gives an $A_{a-1}$-singularity of $U$ at the node $q$, which can be resolved by blowing up and inserting a chain of $a-1$
rational curves $R_1, \ldots, R_{a-1}$ between the two branches $X_{v_1}$ and $X_{v_2}$. Since the $R_i$ are Cartier divisors after blowing up, one can use them to twist the relative dualizing line bundle such that the limit twisted differential corresponds to
a section of some twisted relative dualizing line bundle. Locally twisting by a branch at a node makes the zero or pole order increase by one in one branch and
decrease by one in the other branch, hence their sum remains to be $-2$ as in the case of the relative dualizing line bundle before twisting. Then condition (2) follows from applying this observation to the nodes between $X_{v_1}, R_1, \ldots, R_{a-1}, X_{v_2}$ and the fact that the relative dualizing line bundle restricted to $R_i$ has degree zero. Similarly one can see condition (1) this way, as it holds for the relative dualizing line bundle without twisting.
\par
Finally we focus on conditions~(3) and~(4) in Definition \ref{def:twistedAbType}, which
depend on the level graph.  By definition of the scaling limits, we have
$$\eta_{v_1} \= \lim_{t\to 0} t^{\ell_1} \famomega (t) |_{X_{v_1}}, \quad \eta_{v_2} \= \lim_{t\to 0} t^{\ell_2}
\famomega (t) |_{X_{v_2}}. $$
By definition of the full order, $\ell_1 \ge \ell_2$ if and only if $v_1\succcurlyeq v_2$. If
$\eta_{v_1} = x^{-1+k} dx$ for some non-negative integer $k$ (and hence $\eta_{v_2} = y^{-1-k} dy$), then
combining these conditions implies $\ell_2 - \ell_1 = k\,a \geq 0$.
Moreover $k=0$ if and only if $\ell_1 = \ell_2$, namely $v_1\asymp v_2$, as claimed in condition~(3).
\par
For the global residue condition~(4), as in its statement, let $Y$ be a connected
component of $X_{>L}$ that contains no marked poles. Let $q_1,\ldots, q_b$ denote the nodes where~$Y$ meets $X_{=L}$ and $q_{b+1},\ldots,q_{b+c}$ be the nodes where~$Y$ meets the other irreducible components of $X$.
Since $Y$ is a connected
component of~$X_{> L}$ and since all components of~$X_{=L}$
have already been accounted for, it follows that $q_{b+1}^-,\ldots, q_{b+c}^-\in X_{<L}$.
\par
For each $j=1,\ldots,b+c$, let $V_{q_j}(t) \in H_1(X_t, \ZZ)$ be the
family of vanishing cycles, i.e., $V_{q_j}(t)$ is
a simple loop on $X_t$~for $t\ne 0$ that shrinks to the point~$q_j^+$ as $t\to 0$.
We cut the surfaces $X_t$ along all the vanishing cycles $V_{q_j}(t)$
and obtain two subsurfaces $Y_t$ and $Z_t$, which we label so that
$Y_t$ converges to $Y$ and the (possibly disconnected) subsurface $Z_t$
converges to the complement $X\setminus Y$. Hence $Y_t$ is a Riemann surface with boundary
$\cup_{j=1}^{b+c} V_{q_j}(t)$. We provide $V_{q_j}(t)$ with the
counterclockwise orientation induced by the flat structure on $Y_t$.
Since $Y_t$ contains no prescribed pole, by Stokes' theorem we conclude that
\be \label{eq:diskIntegral}
\sum_{j=1}^{b+c} \int_{V_{q_j(t)}} \famomega(t)|_{Y_t} \= 0 \quad \text{for any}
\quad t \neq 0\,.
\ee
Since $\famomega(t)$ is a family of differentials on the whole surface $X_t$ for $t\neq 0$,
we can as well consider it on $Z_t$ and multiply by $t^L$ to obtain
\be \label{eq:diskIntegral2}
\sum_{j=1}^{b+c} \int_{V_{q_j(t)}} t^L \famomega(t)|_{Z_t} \= 0 \quad \text{for any} \quad t \neq 0\,.
\ee
Now we observe that $t^L \famomega(t)|_{Z_t}$ converges, as $t\to 0$, to $\eta_{v}$ for
every component~$X_v\subset X_{=L}$, and converges to zero on $X_{<L}$.
Since the limit of the loops $V_{q_j}(t)$ does not lie on any
component of level~$>L$ by construction, this means taking the limit as
$t \to 0$ of~\eqref{eq:diskIntegral2} gives precisely the global residue
condition~(4).
\end{proof}

\subsection{Deformations of standard coordinates}
\label{sec:standardform}

In this section and next, we will establish a number of technical results that are needed for proving the sufficiency of the conditions in the main theorem.
\par
The two obvious local conformal invariants of a holomorphic one-form $\omega$ defined in a small disk containing $0$ are its order of
vanishing  $k=\ord_0\omega$ and its residue $r=\Res_0\omega$.  It is somewhat less obvious that these are
the only local invariants of $\omega$.  More precisely, there is a nearly canonical change of
coordinates which puts $\omega$ in a standard normal form depending only on $k$ and $r$.
If $\omega$ is a holomorphic one-form on a small disk, and has a zero of order $k$ at the origin, one can
take a coordinate $z$ such that
$$z(p)^{k+1} = (k+1)\int_0^p\omega, $$
which makes $\omega = z^kdz$.
For meromorphic one-forms, the residue needs to be accounted for, and the general statement
is the following.
\par
\begin{prop} \label{prop:standard_coordinates}
Given any meromorphic one-form $\omega$ on a sufficiently small disk $V$ around $0$, denote
$k\coloneqq\ord_0\omega$ and $r\coloneqq\Res_0\omega$. Then there exists a conformal map
$\phi\colon(\Delta_R, 0) \to (V, 0)$ defined on a disk of sufficiently small radius $R$ such that
\begin{equation}\label{eq:standard_coordinates}
    \phi^*\omega \=
    \begin{cases}
      z^k\, dz &\text{if $k\geq 0$,}\\
      \frac{r}{z}dz &\text{if $k = -1$,}\\
      \left(z^k + \frac{r}{z}\right)dz &\text{if $k < -1$.}
    \end{cases}
 \end{equation}
The germ of $\phi$ is unique up to multiplication by a $(k+1)$st root of unity when $k > -1$,
and up to multiplication by a nonzero constant if $k=-1$.

Moreover for a neighborhood $U$ of the origin $\bfzero\in \CC^n$, let $\omega_\bfu$ be a family of meromorphic differentials on $V$ that vary holomorphically with the base parameter $\bfu \in U$ and that satisfy $\ord_0\omega_\bfu \equiv k$. Then in some
neighborhood of any point in $U$, one can find for some $R>0$ a family of conformal maps $\phi_\bfu\colon(\Delta_R, 0) \to (V, 0)$, varying holomorphically with
$\bfu$, such that $\phi_\bfu^* \omega_\bfu$ is in the standard form of \eqref{eq:standard_coordinates}.
\end{prop}

\begin{proof}
The desired coordinates are constructed in \cite{Strebel} as a solution to a suitable ordinary
differential equation.  The last statement follows immediately from holomorphic dependence on
parameters.
\end{proof}

We call the above coordinate $z$ a {\em standard coordinate} for $\omega$.

For families of differentials such that $\ord_0 \omega_\bfu$ is not constant, the situation is more
complicated.  Given a form $\omega = z^k dz$ in standard coordinates, for $k\geq 0$, a domain
$U\subset\CC^n$, and a holomorphic map $\alpha\colon\Delta_R \times U \to \CC$, written as
$\alpha_\bfu(z)$ and such that $\alpha_\bfzero$ is identically zero, consider the family of one-forms
\begin{equation}  \label{eq:alpha_def}
  \omega_\bfu \,\coloneqq\, \left(z^k + \frac{\alpha_\bfu(z)}{z}\right)dz,
\end{equation}
which we call an \emph{$\alpha$-deformation} of $\omega$.  In this deformation, the zero of order
$k$ breaks up into a simple pole at $0$ and $k+1$ nearby zeros (or $k$ zeros if
$\alpha_\bfu(0)\equiv 0$).

While the above standard coordinates can no longer be defined on a neighborhood of $0$, we now
establish that these coordinates persist on an annulus $\{R_1 < |z| < R_2\}$, where all zeros of
$\omega_\bfu$ are contained in $\{|z| < R_1\}$.

\begin{thm} \label{thm:deformed_standard_coordinates} Consider the
  annulus $A\coloneqq \{z:R_1 < |z| < R_2\}\subset\Delta_R$ for some
  fixed $R_1<R_2<R$. Let $\omega\coloneqq z^kdz$ be the holomorphic
  differential in the standard form on $\Delta_R$ for some $k\ge 0$,
  $\omega_{\bfu}$ an $\alpha$-deformation of $\omega$, and let
  $\theta_j\coloneqq e^{2\pi i j/(k+1)}$ be a $(k+1)$st root of
  unity. Choose a basepoint $p\in A$ and a holomorphic map
  $\sigma\colon U\to\Delta_R$ such that
  $\sigma({\bfzero}) = \theta_j p$.
\par
Then there exists a neighborhood $U_{\bfzero}\subset U$ of ${\bfzero}$
and a holomorphic map $\phi:A\times U_{\bfzero}\to\Delta_R$, such that
$\phi_{\bfu}^* (\omega_\bfu) = \eta_{\alpha_\bfu(0)}$ where
\begin{equation*}
  \eta_r \=  \left(z^k + \frac{r}{z}\, \right)dz,
\end{equation*}
with $\phi_{\bfzero}(z) = \theta_j z$ for all $z\in A$ and with $\phi_\bfu(p) = \sigma(\bfu)$ for all $\bfu\in U_{\bfzero}$.
\end{thm}
We remark that using $\theta_j=1$ would be enough for our purpose here, but the freedom of choosing
a different $(k+1)$st root $\theta_j$ will be useful in our forthcoming paper
\cite{BCGGM2}. Moreover, the theorem allows $\alpha_\bfu(0)=0$ for all $\bfu$, namely, it allows us
to work with deformations that do not have residues. This will be important for the procedure of
merging nearby zeroes to a higher order zero, which we will develop later.
\begin{proof}
Given a holomorphic function $h = h_{\bfu}(z)$ on $A$ depending holomorphically on the parameter
$\bfu \in U$, let $\phi_h$ denote the holomorphic function
$$
  \phi_h(z) \coloneqq \theta_j \cdot z\cdot e^{h(z)}.
$$
We interpret the condition $\phi_h^*(\omega_\bfu) = \eta_{\alpha_\bfu(0)}$  as a differential equation involving $h$ and $\alpha$.  The Implicit Function Theorem for Banach  spaces (see \cite[Theorem~XIV.2.1]{LangRealAndFunctional} or \cite{whittlesey}) will then allow us to perturb the solution $\phi_{\bfzero}(z) = \theta_j z$
(with $h_{\bfzero}(z) \equiv 0$) to obtain a holomorphically varying family of solutions.
\par
Let $\banach{A}{m}$ denote the Banach space of holomorphic functions
on $A$ whose first $m$ derivatives are uniformly bounded, equipped
with the norm
  $$\|f\|_m \,\coloneqq\, \sum_{j=0}^m \sup_{z\in A} |f^{(j)}(z)|.$$
We calculate $\phi_h^*(\omega_\bfu) - \eta_{\alpha_\bfu(0)} = F(h, \bfu)dz,$ where
 \begin{equation*}
    F(h, \bfu) \= \left( h' + \frac{1}{z} \right) \left( z^{k+1} e^{(k+1)h} + \alpha_\bfu(\theta_j z e^h)\right) - z^k - \frac{\alpha_\bfu(0)}{z}.
 \end{equation*}
 To make sure that $F(h,\bfu)$ is defined, we choose a neighborhood of
 the identically zero function $W \subset \banach{A}{1}$ such that for
 any $h$ in $W$, the function $\phi_h$ is univalent on $A$, takes
 values in $\Delta_R$, and $0$ lies in the bounded component of the
 complement of $\phi_h(A)$.  Note that $\phi_h$ is univalent on $A$ as
 long as $\psi_h(z) = \phi_h(z) -z$ satisfies $|\psi_h'(z)|< 1/2$ on
 $A$, since $\phi_h(z_1) = \phi_h(z_2)$ implies
 $|z_1 - z_2| = |\psi_h(z_1) - \psi_h(z_2)| \leq \frac{1}{2}|z_1 -
 z_2|$, so $z_1 = z_2$.

 We then regard $F$ as a holomorphic map
 $F\colon W \times U \to H$, where $H\subset\calO(A)_0$ is the
 subspace of functions with zero coefficient of $z^{-1}$ in the
 Laurent expansion.  Note that $H$ is a closed subspace, and thus
 itself a Banach space, as the zero coefficient condition is
 equivalent to the residue being zero, and computing the integral over
 a path is a continuous linear map.  That the image of $F$ lies in $H$ follows immediately from the definition of $F$, since for any curve $\gamma$ generating the fundamental group of $A$, the image $\phi_h(\gamma)$ winds once around zero, so
 $$\int_\gamma \phi_h^*\omega_\bfu = \int_{\phi_h(\gamma)}\omega_\bfu = 2\pi i \alpha_\bfu(0) = \int_\gamma \eta_{\alpha_\bfu(0)}.$$
 It follows theat $F$ has no $z^{-1}$ term.

To handle the initial condition, we define $G\colon W\times U \to H\times \CC$ by
$$G(h,\bfu)  \,\coloneqq\, (F(h, \bfu), \, \theta_j p\, e^{h(p)} - \sigma(\bfu)).$$
We then have $G(0, {\bfzero}) = (0,0)$.  We wish to apply the Implicit Function Theorem to produce a neighborhood $U_{\bfzero}\subset U$ of ${\bfzero}$ and a holomorphic function $\Phi\colon U_{\bfzero}\to W$ such that $G(\Phi(\bfu), \bfu) \equiv 0$.  The desired family of holomorphic maps is then given by
$\phi_{\Phi(\bfu)}$. To apply the Implicit Function Theorem, it suffices to show that $D_h G_{(0, {\bfzero})}\colon
\banach{A}{1}\to H\times\CC$ (the derivative of $G$ at $(0,{\bfzero})$ with respect to $h$) is a linear Banach space homeomorphism.  We first compute
$D_h F_{(0,{\bfzero})}\colon \banach{A}{1}\to H$ as follows:
  \begin{align*}
    D_h F_{(0,{\bfzero})}(h) &= \frac{d}{dt}F(th, {\bfzero})|_{t=0} \\
                     &= \frac{d}{dt} \left[ \left(th' + \frac{1}{z}\right) z^{k+1} e^{(k+1)th}
                    \right]_{t=0}\\
                     &= z^{k+1} h' + (k+1)z^k h \\
                     &= (z^{k+1} h)'.
  \end{align*}
  The derivative $D_h F_{(0,{\bfzero})}$ is thus a bounded linear
  operator with kernel spanned by $z^{-k-1}$. It also follows that
  \begin{equation*}
    D_hG_{(0,{\bfzero})}(h) \= ((z^{k+1} h)', \, \theta_j p\, h(p)).
  \end{equation*}
  Define an operator $T\colon H\times \CC\to \banach{A}{1}$ by
  $$T(h, c)  \,\coloneqq\, z^{-k-1} \tilde{h},$$
  where $\tilde{h}$ is
  the antiderivative of $h$ such that the initial condition $\tilde{h}(p) = \theta_j^{-1} p^k c$
  holds. One checks that $T$ and $D_hG_{(0,{\bfzero})}$
  are inverse operators, and hence  $D_hG_{(0,{\bfzero})}$ is a linear Banach space homeomorphism as desired.
\end{proof}
\par
\subsection{Plumbing construction}
\par
Plumbing gives a way to deform a nodal Riemann surface to a smooth one, by plumbing the two components of the punctured neighborhood of a node (see e.g. \cite{wolpertplumb} for a modern treatment). Starting with a twisted differential, we would like to construct suitable meromorphic differentials on the plumbed Riemann surface.
\par
Let us first outline the key steps involved in the plumbing construction. We want to open up a node with a holomorphic differential on the higher level component and a meromorphic differential with a pole at the node on the
lower level component.  When this pole has no residue, from the flat geometric point of view the
plumbing is a local cut-and-paste construction (see~\cite{gendron} and~\cite{chen}).
\par
When the polar node on the lower level component has a nonzero residue,
it is not possible to plumb by a purely local construction, because such nonzero residues correspond to periods of curves far from the nodes under the flat metric.
To deal with this issue, we first add to the differential on the higher level component a small meromorphic differential with simple poles at
the nodes such that their residues are opposite to that of the poles on the lower level component.  This allows us to
plumb the nodes by gluing along annuli via the conformal maps constructed in
Theorem~\ref{thm:deformed_standard_coordinates}. This is done in Theorem~\ref{thm:plumbgeneral} below.
\par
Now, when adding this meromorphic differential, a zero of the twisted differential on the higher level component
may break into a number of nearby zeroes of lower order with the same total multiplicity.  In
Lemma~\ref{lm:merge} below, we describe a local operation merging these zeros back together by cutting out
a neighborhood containing all these zeros and gluing in along an annulus the form $z^kdz$ supported on a small
disk, again using the conformal maps constructed in Theorem~\ref{thm:deformed_standard_coordinates}.  This ensures
that the resulting flat surfaces are contained in the prescribed stratum.
\par
We now set up the plumbing construction in more detail.  Suppose $\calX$ is a family of nodal
Riemann surfaces over the disc~$\Delta$ with a persistent node along the section $n(t)$ together
with a relative meromorphic one-form $\mathcal{W}$, that is, a holomorphic family of meromorphic
one-forms modulo differentials pulled back from~$\Delta$. We denote by $\omega_t$ the
restriction of $\mathcal{W}$ to the fiber $X_t$.  We want to modify $(\calX, \mathcal{W})$ to obtain a
family of smooth Riemann surfaces equipped with one-forms degenerating to the singular fiber  $(X_0,
\omega_0)$.  The construction is {\em local} near the section of
nodes~$n(t)$ and we will often tacitly shrink the neighborhood of~$n(t)$ that we work in.
\par
To smooth such a family, we use the standard \emph{plumbing fixture}.
Let $a$ be a positive integer. Consider the family $\pi_a\colon \VV_a \to \Delta_1$
of cylinders degenerating to a node
\begin{equation*}
  \VV_a \= \{(u,v,t)\in \Delta_1^3 \colon uv=t^a\}
\end{equation*}
where the projection is given by $\pi_a(u,v,t) = t$.  Let $\Omega$ be a relative meromorphic differential
on $\VV_a$, which we think of as a model for the plumbed family defined later. Denote by $\Omega_t$ the restriction of $\Omega$ to the fiber over $t$.
\par
We first introduce some notation for the precise {\em plumbing setup}:
\begin{itemize}
\item[$\bullet$] Let $\pi^+\colon\calX^+\to\Delta_\epsilon$ and
  $\pi^-\colon\calX^-\to\Delta_\epsilon$ be two holomorphic families
  of (possibly nodal) Riemann surfaces over
  a disk, equipped with two relative meromorphic one-forms $\mathcal{W}^+$ and $\mathcal{W}^-$. We remark that the two families are allowed to be the same for classical plumbing at a non-separating simple polar node (see Proposition~\ref{prop:classical_plumbing} below). We denote by $\omega^+_t$ and $\omega^-_t$ the restriction of these forms to the fibers over $t$.

  \item[$\bullet$]
  Let $n^\pm:\Delta_\epsilon\to\calX^\pm$ be two holomorphic sections of these families away from the nodes of $\calX^{\pm}$, and let
  $\pi:\calX\to\Delta_\epsilon$ be the family of nodal Riemann surfaces obtained by identifying
  $\calX^+$ and $\calX^-$ along $n^+$ and $n^-$ to form the section of nodes $n$ (and possibly along other sections $n_i^\pm$ to form additional nodes that we will ignore as the plumbing construction is local at $n$). For $t\in\Delta_\epsilon$,
  we denote by
    $$
      X_t\=X_t^+\cup X_t^-\=X_t^+\sqcup X_t^-/(n^+(t)\sim n^-(t))
    $$
    the fiber of $\calX$ over $t$.

\item[$\bullet$] For some $0<\delta\ll 1$, let $U_\delta = \Delta_\delta^2 \times \{0\}\cap \VV_a$ be a
  neighborhood of the node in the central fiber of $\VV_a$, and let $D\subset X_0$ be a neighborhood of the node in the central fiber of $\calX$.

\item[$\bullet$]  We write $\VV_{a, \delta, \epsilon}$ for the restricted family
  $\VV_{a} \cap (\Delta_\delta^2\times \Delta_\epsilon)\to\Delta_\epsilon$. For some fixed
  $0<\delta'<\delta$ and $0<\epsilon < (\delta')^{2/a}$, we let
  $\calA\coloneqq\VV_{a,\delta,\epsilon}\setminus \overline{\VV_{a,\delta',\epsilon}}$ (where the bar means simply the closure). The bound on $\epsilon$ implies that $\calA$  is the disjoint union of two families of annuli $\calA=\calA^+\cup \calA^-$.

\item[$\bullet$] Let $\calE^\pm\subsetneq\calD^\pm\subset\calX^\pm$ be two families of conformal disks centered around $n^\pm$, so that $\calB^\pm \coloneqq \mathcal{D}^\pm \setminus
  \overline{\mathcal{E}}^\pm$ are two families of annuli, disjoint from the nodes. We require that the intersection of $\mathcal{D}^\pm$ with $X_0$
  is contained in $D$.
\end{itemize}

We will often omit superscripts when we take the union of two objects in the above setup, e.g.\ $\mathcal{E} = \mathcal{E}^+\cup \mathcal{E}^-$.
\par
Given the plumbing setup, we say that the family $(\mathcal{X}, \mathcal{W})$ is \emph{plumbable with fixture $(\VV_a, \Omega)$}, if the following conditions hold:
\begin{itemize}
\item[$\bullet$] There exists a conformal isomorphism $\phi\colon U_\delta\to D$ such that $\phi^*\omega_0 = \Omega_0$.

\item [$\bullet$] There exist two families of conformal isomorphisms $\Phi^\pm\colon \mathcal{A}^\pm \to \mathcal{B}^\pm$ such that
$(\Phi^\pm)^*\mathcal{W} = \Omega$ and such that the restriction of $\Phi^\pm$ to the central fiber agrees with $\phi$.
\end{itemize}

If $(\mathcal{X}, \mathcal{W})$ is plumbable with fixture $(\VV_a, \Omega)$, let $(\mathcal{X}', \mathcal{W}')$  be
the family obtained by identifying $\calX\setminus\overline{\mathcal{E}}$ and
$\VV_{a,\delta,\epsilon}$ along $\mathcal{A}^\pm$ and $\mathcal{B}^\pm$ via the maps $\Phi^\pm$. Namely,
$$
\calX'\,\coloneqq\,\left(\left(\calX\setminus\overline\calE\right)\sqcup \VV_{a,\delta,\epsilon}\right)\,/\,
(\calA\,\mathop{\sim}\limits^\Phi\, \mathcal{B}),
$$
and $\mathcal{W}'$ is the relative meromorphic differential induced by $\mathcal{W}$ and
$\Omega$. In this case we say that $(\mathcal{X}', \mathcal{W}')$ is the \emph{plumbed family with fixture $(\VV_a, \Omega)$}. Note that
$(\mathcal{X}', \mathcal{W}')$ is smooth except for the central fiber, which is isomorphic to the
nodal surface $(X_0, \omega_0)$.
\par
Our plumbing construction differs from other constructions in that our gluing maps are defined on a
family of annuli of fixed moduli (with inner radius $\delta'$ and outer radius~$\delta$), instead of a family of growing annuli whose moduli tend to
infinity. This is necessary, because Theorem~\ref{thm:deformed_standard_coordinates}, which will be
used in several plumbing contexts, only constructs conformal maps on a fixed annulus.
\par
The case of two differentials with simple poles and opposite
residues is known in the literature as classical plumbing. We recall this construction here using the above notation.
\par
\begin{prop}[{\bf Classical plumbing}] \label{prop:classical_plumbing}
Suppose in the plumbing setup each of the two differentials $\omega^\pm_t$ has a simple pole at $n^\pm(t)$, with residue
$r^\pm(t)$.
If $r^+(t)+r^-(t) = 0$ for every $t$, then the family is plumbable with fixture $(\VV_1, \Omega)$, where
\begin{equation*}
   \Omega_t \, \coloneqq\, r^+(t)\, \frac{du}{u} = r^-(t)\, \frac{dv}{v}.
\end{equation*}
\end{prop}
\par
\begin{proof}
  By working locally in the neighborhood of the pole, we transfer the problem
to the setup of Proposition~\ref{prop:standard_coordinates}. We choose two
holomorphically varying coordinates given by
  Proposition~\ref{prop:standard_coordinates} $\psi^\pm\colon \Delta_\delta \to \mathcal{X}^\pm$
  such that $\psi^\pm(0) = n^\pm(0)$ and such that the pullback is in the standard form:
  \begin{equation*}
    (\psi_t^+)^{*} \omega^+_t = r^+(t) \frac{dw}{w} \quad\text{and}\quad (\psi_t^-)^{*} \omega^-_t = r^-(t)\frac{dw}{w}.
  \end{equation*}

We define the gluing map $\phi$ on the central fiber of $\VV_{1, \delta, \epsilon}$ by
 $\phi( u,0,0) = \psi_0^+(u)$ and $\phi(0,v,0) = \psi^-(v)$. Choose any $\delta'$ with
  $0<\delta'<\delta$ and define
  $\mathcal{A} = \VV_{1,\delta,\epsilon} \setminus \overline\VV_{1,\delta',\epsilon}$ as in the plumbing
  setup.  On the upper component $\mathcal{A}^+$ of $\mathcal{A}$, define the gluing map
  $\Phi^+(u,v,t) = \psi_t^+(u)$, and on the lower component $\mathcal{A}^-$, define the gluing map
  $\Phi^-(u,v,t) = \psi_t^-(v)$. Then these gluing maps produce the desired plumbing family.
\end{proof}
\par
We now turn to the crucial case of plumbing a zero of order $k\geq 0$ on $\calX^+$ to a pole of
order $k+2$ on $\calX^-$.   As the pole may have a nonzero residue, we need
a family of meromorphic differentials $\xi = \{\xi_t\}$ on $\calX^+$, such that $\xi_t$ has a simple
pole at the node of $X^+_t$ with residue opposite to that of $\omega^-_t$.
Scale $\xi_t$ by a suitable power of $t$, add it to $\omega^+_t$, and scale $\omega^{-}_t$ by the same power of $t$.
We will then plumb the modified families of differentials on $\calX^+$ and $\calX^-$
by using the standard coordinates for the deformation, constructed in Theorem~\ref{thm:deformed_standard_coordinates}.
We call $\xi$ the \emph{modification differential} for this plumbing procedure.
\par
\begin{thm}[{\bf Higher order plumbing}]\label{thm:plumbgeneral}
In the plumbing setup, suppose that for any $t$ the form $\omega^+_t$
has a zero of order $k\geq 0$ at $n^+(t)$, while $\omega^-_t$ has a pole
of order $k+2$ at~$n^-(t)$ with residue $r(t)$.  Let $\{\xi_t\}$ be a family of meromorphic
differentials on~$\calX^+$ which has simple poles at $n^+(t)$ with
residues $-r(t)$ and is otherwise holomorphic in a neighborhood of $n^+(t)$. Consider the family of differentials $(\mathcal{X},
\mathcal{N})$ given by
\be \label{eq:plumbdiff}
\eta^+_t = t^c (\omega^+_t + t^{b}\xi_t) \,\,\text{on}\,\, \calX^+ \quad \text{and} \quad
\eta^-_t = t^{b+c}\omega^-_t \,\,\text{on}\,\, \calX^-,
\ee
where $c$ is a nonnegative integer and $b\coloneqq a(k+1)$ for some positive integer $a$, and let~$\Omega$ be the family
\begin{equation*}
   \Omega_t\,\coloneqq\,t^c\left(u^k- t^b \,\frac{r(t)}{u}\right)du \= t^{b+c}\left(-v^{-k-2}+\frac{r(t)}{v}\right)dv
\end{equation*}
of meromorphic differentials on $\VV_a$.
\par
Then the family $(\mathcal{X}, \mathcal{N})$ is plumbable with fixture
$(\VV_a, \Omega)$. Moreover, the scaling limit in the sense of
Lemma~\ref{le:scalinglimit} of the plumbed family of
differentials $(\mathcal{X}', \mathcal{N}')$
on $X^{\pm}_0$ is equal to $\omega^{\pm}_0$.
\end{thm}
We remark that it suffices to prove the case $c = 0$, and then the general case follows from multiplying by $t^c$ at relevant places. The reason we include the exponent $c$ is for later use when we apply
the theorem in Section~\ref{subsec:proof_sufficiency}, where $c$ and $b+c$ will be related to the scaling parameters of the two branches at the plumbed node.
\begin{proof}
  As said above, we only need to consider the case $c=0$. We want to construct the gluing maps $\phi$ and $\Phi$ in the definition of plumbing with prescribed fixture. By Proposition~\ref{prop:standard_coordinates}  (and multiplying the
  forms by $-1$ in the second case), we can choose two coordinates $\psi_t^\pm\colon \Delta_\delta \to X_t^\pm$, varying holomorphically with $t$ in $\Delta_\epsilon$, such that $\psi_t^\pm(0) = n^\pm(t)$,
 \begin{equation*}
    (\psi_t^+)^*\omega_t^+ = w^k dw, \quad\text{and}\quad (\psi_t^-)^* \omega_t^-=\left(-w^{-k-2}+\frac{r(t)}{w}\right)dw.
\end{equation*}
We define the gluing map $\phi$ on the central
fiber of $\VV_{a,\delta,\epsilon}$ by $\phi(u,0,0) = \psi_0^+(u)$ and $\phi(0,v,0) = \psi_0^-(v)$.
In these coordinates on the upper component, we have
\begin{equation*}
    (\psi_t^+)^*(\omega^+_t  + t^b \xi_t) = \left(w^k + \frac{\alpha_t(w)}{w}\right)dw,
\end{equation*}
where $\alpha_t(w)$ is holomorphic in $w$ and $t$, and $\alpha_t(0)= -t^b r(t)$.  Let $A = \{ u :
\delta_1 < |u| <  \delta_2 \}$ for some $0 < \delta_1 < \delta_2 < \delta$. Theorem~\ref{thm:deformed_standard_coordinates} then gives (after possibly shrinking the base) a family of conformal maps
 $\gamma_t\colon A \to \Delta_\delta$ such that
  \begin{equation*}
       \gamma_t^*\left(w^k + \frac{\alpha_t(w)}{w}\right)dw = \left(u^k- t^b \,\frac{r(t)}{u}\right)\,du
  \end{equation*}
  The desired gluing map on the upper component $\mathcal{A}^+$ of
  $\mathcal{A} = \VV_{a,\delta_2, \epsilon} \setminus \VV_{a, \delta_1, \epsilon}$ is then defined
  by $\Phi^+ (u,v,t) = \psi_t^+\circ \gamma_t(u)$. The gluing map on the lower component $\mathcal{A}^-$ is simply $\Phi^-(u,v,t) = \psi_t^-(v)$. By construction
  $\Phi$ identifies $\Omega$ and $\mathcal{N}$ as desired.
\end{proof}
\par
To make this theorem useful in smoothing, we need a criterion describing when the modification
differential~$\xi$ can be constructed globally on a Riemann surface.
\par
\begin{lm} \label{le:resmodification} Let $f: \calX \to \Delta$ be a family of stable curves with
  sections $q_j$ for $j=1,\ldots,n$. Suppose that $r_j(t)$ for $j =1, \ldots, n$ is a collection of
  holomorphic functions on $\Delta$
  with $\sum_{j =1}^n r_j(t) = 0$. Then there exists a family of meromorphic
  differentials $\xi_t$ on $\calX$, varying holomorphically with $t$,
  with at worst simple poles at the $q_j$ and at the nodes, satisfying
  the opposite-residue condition at the nodes, holomorphic outside these sections and prescribed nodes, and such that
  $\Res_{q_j(t)} (\xi_t) \= r_j(t)$ for all $j$ and all $t \in \Delta$.
\end{lm}
\par
\begin{proof} For a single smooth Riemann surface $X$, the claim follows from the classical
Mittag-Leffler problem (see e.g.~\cite[Theorem~18.11]{forster}).
\par
To justify the holomorphic dependence we want to find a
section of $f_* \omega_{\calX/\Delta}(\sum q_j)$ that has a prescribed image under the fiberwise linear map $\Res$
to the trivial bundle $\CC^n \to \Delta$. Note that the set of sections with the prescribed image under $\Res$ is
non-empty by the preceding paragraph, and is an affine bundle modelled on a trivial
vector bundle over a disc. Such a bundle over a disc admits a holomorphic section as desired.
\par
The proof applies to the case of
stable curves as well, by treating $\omega_{\calX/\Delta}$ as the
relative dualizing sheaf. Alternatively, the stable version
can be deduced from the smooth version by degenerating to
the desired nodes via classical plumbing.
\end{proof}
\par
Now, when adding the modification differential $\xi_t$ to $\omega_t^+$ as in \eqref{eq:plumbdiff}, a zero of $\omega_t^+$
may in general break into many nearby zeros of lower order. In order to construct a family
of differentials that remains in the initial stratum, we describe one more plumbing construction
which merges these nearby zeros back into a single zero.
\par
Our setup here is similar to the plumbing setup described above, except we will only modify a
single family of Riemann surfaces. It is explicitly described as follows:
\par
\begin{itemize}
\item Let $\calX\to\Delta_\epsilon$ be a family of (possibly nodal) Riemann surfaces, and let $\famomega$ be a
  family of holomorphic one-forms on $\calX$ which vanishes to order $k>1$ at a smooth point $p$ in the
  central fiber. We denote the restriction of $\famomega$ to the fiber $X_t$ by $\omega_t$.

\item Let $\phi_0\colon \Delta_\delta\to X_0$ be a conformal map onto a neighborhood $D$ of $p$ in $X_0$ such
  that $\phi_0^*\omega_0 = w^k dw$ in the standard coordinates.

\item Let $\mathcal{E}\subset\mathcal{D}\subset\mathcal{X}$ be two families of conformal disks such
  that $\mathcal{D} \cap X_0 \subset D$, and let $\mathcal{B} = \mathcal{D} \setminus
  \overline{\mathcal{E}}$, a family of conformal annuli.   Let $\calA$ be the constant family
  of annuli $(\Delta_{\delta}\setminus\overline{\Delta}_{\delta'})
  \times\Delta_\epsilon\to\Delta_\epsilon$.
\end{itemize}

Given this setup, if there exists a family of conformal maps $\Phi\colon\mathcal{A}\to\mathcal{B}$ whose restriction to the central fiber agrees with $\phi_0$,
and such that $\Phi^*\famomega = w^k dw$ for all $t$, then we can construct a new family $(\mathcal{X}', \famomega')$ by gluing
$\mathcal{X}\setminus\overline{\mathcal{E}}$ to $\Delta_\delta\times \Delta_\epsilon$ along the
annuli via the map
$\Phi$.  The central fiber $(X_0', \omega_0')$ is isomorphic to $(X_0, \omega_0)$, and each $\omega'_t$ has a
zero of order $k$. In this case we say that the resulting family $(\mathcal{X}',\mathcal{W}')$ is a
\emph{merging} of $(\mathcal{X}, \famomega)$ at $p$.

\begin{lm}\label{lm:merge}
  Given a family $(\mathcal{X}, \famomega)$ of (possibly nodal) Riemann surfaces with a family of one-forms
  vanishing at a smooth point $p\in X_0$ to order $k>1$, there exists a merging of $(\mathcal{X}, \mathcal{W})$ at
  $p$.
\end{lm}
\begin{proof}
We provide two proofs here. The first is an application of Theorem~\ref{thm:deformed_standard_coordinates}. Extend $\phi_0$
  to a holomorphic family of conformal maps $\phi_t\colon \Delta_\delta \to X_t$. Pulling back
  $\omega_t$ yields
  \begin{equation*}
    \phi_t^* \omega_t = (w^k + \beta_t(w)) dw,
  \end{equation*}
  where $\beta$ is holomorphic in $t$ and $w$ with $\beta_0 \equiv 0$.  This is an
  $\alpha$-deformation of $w^k dw$, where $\alpha_t (w) = w \beta_t(w)$, so $\alpha_t(0) = 0$ for all
  $t$.  Fix constants $0 < \delta_1 < \delta_2< \delta$ and let $A = \{w \colon \delta_1 < |w| < \delta_2\}$.
  Theorem~\ref{thm:deformed_standard_coordinates} then yields a family of conformal maps
  $\psi_t\colon A\to \Delta_\delta$ such that $\psi_t^* \phi_t^*\omega_t = w^k dw$.  Defining
  $\Phi(w,t) = \phi_t\circ \psi_t$ then provides the desired gluing map, which completes the construction.
\par
Alternatively, we can prove the lemma without using Theorem~\ref{thm:deformed_standard_coordinates},
because the differentials involved here have no residues. Let $\famomega$ and $\widetilde{\famomega}$
 be respectively the relative differentials $\bigl(u^k + a_2(t)u^{k-2} + \cdots +a_k(t)\bigr)\,du $  (see e.g.\ \cite[Proposition~3]{kozo1})
and  $u^{k}\,du$ on $\mathbb{A}^{1}\times \Delta$. Let $x_{0}(t) = (0, t)$ be the common center of masses of the zeros of
these differentials in $X_t$. Define the maps
$$ f(u,t)\=\int_{x_{0}(t)}^{x(t)}\famomega|_{\mathbb{A}^{1}_{t}} \quad \text{and} \quad \widetilde{f}(u,t)\=\int_{x_{0}(t)}^{x(t)} \widetilde{\famomega}\,|_{\mathbb{A}^{1}_{t}}\,. $$
Let $0<r<R$ such that $|f(u,t)|<r$ for all zeros $u$ of $\famomega$. By construction, for each $t$ the annuli
$$ \quad A \,\coloneqq\,\bigl\{r<\left|f(u,t)\right|<R \bigr\} \quad \text{and}  \quad \widetilde{A}\,\coloneqq\,\bigl\{r<|\widetilde{f}(u,t)|<R \bigr\} $$
are $(k+1)$-covers (as flat surfaces defined by the corresponding differentials) of the round annulus
$A_0(r,R)=\left\{z\in\CC:\,r<|z|<R\right\}$, equipped with $dz$. Consequently, there is a biholomorphism
$\Phi: \widetilde{A} \to A$ with $\Phi^* \famomega = \widetilde{\famomega}$, which provides the desired local modification.
\end{proof}

\subsection{Proof of the main theorem: conditions are sufficient}
\label{subsec:proof_sufficiency}
For future use in \cite{BCGGM2} we prove in this section a statement that is slightly
stronger than the sufficiency in Theorem~\ref{thm:main}, since we construct
a degenerating family that has not only a given \twd~$\eta$ as its scaling
limit, but moreover the scaling parameters define a prescribed level graph (rather than just some level
graph with which~$\eta$ is compatible).
\par
\begin{prop} \label{add}
Let $(X,\omega,z_1,\ldots,z_n)$ be a pointed stable differential
and let $\eta = \{\eta_v\}$ be a \twd\ on $X$ compatible
with a level graph $\overline\Gamma$ on $X$,
satisfying the conditions (i), (ii), and (iii) of Theorem~\ref{thm:main}.
\par
Then there exists a family $(f:\famcurv \to\Delta, \famomega, \calZ_1, \ldots, \calZ_n)$ of pointed stable
differentials over a disk~$\Delta$ with parameter~$t$, smooth outside $t=0$,
such that the set of differentials obtained from~$f$ as scaling limits as
in~\eqref{eq:scalinglimit} coincides with~$\eta$ and such that its scaling parameters define precisely
the level graph~$\overline{\Gamma}$.
\end{prop}
\par
The sufficiency part of the main theorem follows immediately from this proposition:
\par
\begin{proof}[Proof of Theorem~\ref{thm:main}, conditions are sufficient]
Let~$f$ be the family of pointed stable differentials given by Proposition~\ref{add}.
Note that this proposition constructs everything required in Theorem~\ref{thm:main}.
In particular, $\eta$ determines the limit location of the marked zeros and poles, and the
components of~$\eta$ are just multiples of the differentials in the
degenerating family. It follows that the central fiber of~$f$ is
indeed $(X,\omega,z_1,\ldots,z_n)$.
\end{proof}
\par
In order to prove Proposition~\ref{add}, we apply induction on the number of levels of the graph. We first
use classical plumbing as a base case, then apply higher order plumbing to smooth the nodes connecting the bottom to upper levels,
and finally merge dispersed zeros to make sure that we end up in the desired stratum.
\par
\begin{proof}[Proof of Proposition~\ref{add}] For the
induction it will be useful that we allow the curve $X$ to be possibly disconnected.
Hence we extend the notion of the number of levels of
a graph to the disconnected case as the maximum of numbers of levels
over all connected components. The proof is then by induction on the
number~$N$ of levels of a possibly disconnected curve~$X$.
\par
The {\em base case of induction} is when the number of levels of $X$
is equal to one, namely, we have $v_1\asymp v_2$ for any two irreducible
components of~$X$. In this case by
conditions~(1) and~(2) all poles at the nodes are
simple poles, of opposite residues on the two sides. Moreover, every
component is maximal for the order, and thus the pointed stable differential
is simply equal to $\eta$.
This case follows directly by applying classical plumbing given in Proposition~\ref{prop:classical_plumbing} to every node one by one,
which shows that any such differential is plumbable. This procedure
obviously preserves the type of the differentials outside the nodes, and
we thus obtain a family in the desired stratum $\proj\omoduli(\mu)$.
\par
For the {\em inductive step}, suppose that for any \twd~$\eta'$ on a (possibly disconnected) nodal curve $X'$ with a full order $\overline{\Gamma'}$ such that the number of levels in~$\overline{\Gamma'}$ is at most $N-1$, satisfying all the
conditions of the theorem, there
exists a degenerating family $(f':\famcurv' \to\Delta, \famomega', \calZ_1, \ldots, \calZ_{n'})$
of pointed stable differentials such that $\eta'$ is the collection of scaling limits of this family
in the sense of Lemma~\ref{le:scalinglimit}. Moreover, we may suppose that~$\overline{\Gamma'}$
is the level graph defined by the function~$\ell(\cdot)$ given by the scaling
parameters of the family.
\par
To prove Proposition~\ref{add} by induction, starting with a \twd~$\eta$ on $X$ compatible
with a full order $\overline\Gamma$ with~$N$
levels given by a level function~$\ell$, we need to construct a family degenerating
to it, satisfying the conditions claimed. Let $L$ be the minimum value of $\ell$ on~$\Gamma$ corresponding to the bottom level of $\overline{\Gamma}$.
Let~$Y_{>L}$ and~$Z$ be the (possibly disconnected) stable subcurves of $X$ corresponding to the
graphs $\overline{\Gamma}_{>L}$ and $\overline{\Gamma}_{=L}$ respectively.
\par
The restriction of $\eta$ to~$Y_{>L}$ is again a \twd, in the sense that it is a \twd\
in our original definition for each connected component of~$Y_{>L}$. Since
the conditions (0)--(3) are local and condition~(4) is imposed level-by-level,
we conclude that the restriction of $\eta$ to~$Y_{>L}$ satisfies all the conditions (0)--(4).
By the inductive assumption, there thus exists a family of differentials
$$
  \left(f_{>L}: \calY \to \Delta,\, \famomega_{>L} :\Delta^{\ast}\to
  (f_{>L})_* \dualsheave[\calY/\Delta](-\seczero_{\rm pol})\right)
$$
satisfying the conditions in Proposition~\ref{add}.
\par
Next, we want to plumb the family~$(\calY, \famomega_{>L})$ to the constant family~$\calZ \coloneqq Z\times\Delta\to\Delta$ with differential
$\eta|_{Z}$ along the
nodes where $\overline{\Gamma}_{>L}$ and $\overline{\Gamma}_{=L}$ intersect, by using higher order
plumbing provided by Theorem~\ref{thm:plumbgeneral}. In this theorem the exponent $b$ in~\eqref{eq:plumbdiff} needs to be divisible by $k+1$. Hence our
first task is to adjust the initially chosen level function $\ell$ in order
to satisfy the divisibility constraint.
\par
Let $\lbrace q_j\rbrace_{j \in J}$ be the set of nodes where $\overline{\Gamma}_{>L}$~and~$\overline{\Gamma}_{=L}$
intersect. Denote by $k_j \geq 0$ the order of $\famomega_{>L}$ at the branch~$q_j^+$
of such a node. For convenience we introduce $\kappa_j\coloneqq k_j + 1$ and let
$K \coloneqq {\rm lcm}(\kappa_j: j \in J)$. Pulling back the family~$f_{>L}: \calY \to \Delta$
via the base change $t\to t^K$ on~$\Delta$ we can from now on suppose that all the
scaling parameters appearing in the family~$f_{>L}$ are multiples of~$K$, while the maximum
of the scaling parameters is still zero. For ease of notation, we continue to use $\ell(\cdot)$ for the new level
function on $\overline{\Gamma}_{>L}$ given by the scaling parameters of~$f_{>L}$ after the base change.
Because of the choice of~$K$ we can find a negative
integer as the new value of the lowest level $\ell(Z)$, still denoted by $L$ (again for ease of notation),
such that
\begin{equation} \label{eq:Lcondition}
\ell(v^+(q_j))- a_j\kappa_j  \= L \quad \text{for all} \quad j \in J\,
\end{equation}
for some positive integers~$a_j$. In this way, the corresponding level graph remains the same as $\overline{\Gamma}$.
\par
The second task is to construct the modification differential~$\xi$ needed in
Theorem~\ref{thm:plumbgeneral}. We consider each connected component $\calY^0$
of~$\calY$ separately. We want to define a family of meromorphic
differentials~$\xi|_{\calY^0}$ such that
$$
  \Res_{q_j^+}(\xi|_{Y^0_t}) \= -\Res_{q_j^-}(\eta|_{Z}) \quad
  \text{for all nodes~$q_j$ in each fiber $Y^0_t$ of $\calY^0$}.
$$
Such a family of differentials exists by Lemma~\ref{le:resmodification} under the assumption that the sum of
residues at those nodes is zero in each $Y^0_t$. If~$Y^0_t$ does not contain a marked point corresponding to
a prescribed pole, the global residue condition~(4) that we imposed gives precisely the desired assumption.
If~$Y^0_t$ contains a marked pole $z_p$, then we allow $\xi|_{Y^0_t}$ to
have a simple pole at $z_p$ as well, with $\Res_{z_p}(\xi|_{Y^0_t})$ equal to
minus the sum of all $\Res_{q_j^+}(\xi|_{Y^0_t})$ appearing in the above.
In this way, we can still apply Lemma~\ref{le:resmodification} to obtain the desired~$\xi$.
The modification differential~$\xi$ will disappear in the scaling limit, because it
is further scaled by a positive power $t^b$ in the setting of Theorem~\ref{thm:plumbgeneral}. Moreover for $t$ small enough, adding $t^b\xi$ does not change the type of the polar part of the differential on the smooth locus of~$Y^0_t$.
\par
We now plumb the family $\calY \cup \calZ$ joined along the nodes $q_j$, for
$j \in J$, where $\calY$ and $\calZ$ play the role of $\calX^+$ and  $\calX^-$ in the notation of Theorem~\ref{thm:plumbgeneral}.
For each $q_j$, set $c =  - \ell(v^+(q_j))$ and $b = a_j \kappa_j$ in the notation of Theorem~\ref{thm:plumbgeneral}. Then
$b+c = - L$ by~\eqref{eq:Lcondition}, hence the scaling factor $t^{b+c} = t^{-L}$ for the modification differential $\xi$ and for the differential
$\eta|_{Z}$ on $\calZ$ is independent of $q_j$ for $j\in J$. Therefore,
we can apply Theorem~\ref{thm:plumbgeneral} simultaneously to plumb all the nodes $q_j$, for $j\in J$, with fixtures $\VV_{a_j}$ and desired model differentials $\Omega_j$ locally around all $q_j$.
As a result, we thus obtain the plumbed family
$$
  \left(f_{\rm nod}: \calX_{\rm nod} \to \Delta, \,\,
  \famomega_{\rm nod} :\Delta^{\ast}\to (f_{\rm nod})_*\, \dualsheave[\calX_{\rm nod}/\Delta](-\seczero_{\rm pol})\right)
$$
of differentials whose scaling limit is the \twd~$\eta$
and where the condition of non-vanishing on top level still holds.
\par
There are two things left to modify in order to obtain the desired family.
First, we apply classical plumbing in Proposition~\ref{prop:classical_plumbing} to smooth the nodes with simple poles in~$f_{\rm nod}$, originally from the constant family $\calZ$,
to obtain a family of differentials
$$
  \left(f_{\rm zeros}: \calX_{\rm zeros} \to \Delta, \,
  \famomega_{\rm zeros} :\Delta^{\ast}\to (f_{\rm zeros})_*\, \dualsheave[\calX_{\rm zeros}/\Delta](-\seczero_{\rm pol})\right)
$$
with smooth fibers for~$t\neq 0$, and with the same scaling limit~$\eta$ as that of $\famomega_{\rm nod}$.
Next, this family $(f_{\rm zeros}, \famomega_{\rm zeros})$ may not belong to
the desired stratum, because the
orders of marked zeros may have been altered when adding the modification differential~$\xi$
to the family of differentials $\famomega_{>L}$ on $\calY$. But by Lemma~\ref{lm:merge} there exists a merging in
the neighborhood of each marked zero, such that the nearby dispersed zeroes of $\famomega_{\rm zeros}$ with total multiplicity $m_i$
are merged back together to form a single zero of multiplicity $m_i$.
We thus finally obtain a family of smooth Riemann surfaces along with differentials in the desired stratum, converging to the given stable curve, and the scaling limit of the differentials is the \twd\ $\eta$ that we started with.
\end{proof}

\section{Flat geometric smoothing}
\label{sec:flatproof}

In this section we give an alternative proof of the sufficiency of the conditions in
Theorem~\ref{thm:main} using techniques from flat geometry. The core ideas
of the proof and the induction procedure are parallel to the proof by plumbing
in the preceding section. However, the flat geometric pictures
presented in this section might look more familiar to some readers. In particular,
they make the necessity of the global residue condition quite transparent. One
difference is that we construct a degenerating family only over a real segment
$[0,\varepsilon)$
rather  than a disc~$\Delta_\varepsilon$ as in the proof via plumbing.
\par
For any \twd~$\eta$ compatible with a full order, our goal is to construct
a family of flat surfaces in the appropriate stratum
that makes the degeneration to~$\eta$ visible. We will first describe the construction,
and then in the last part of this
section justify that the family we construct indeed converges
to the desired limit. Below we start with two examples that illustrate the main
features of the flat geometric proof.

\subsection{Two illustrative examples}

In the first example, suppose $X$ is a nodal curve
with two irreducible components~$Y$ and~$Z$ of genus~$1$ and~$2$
respectively, joined at two nodes. Let $\overline{\Gamma}$
be a level graph of $X$ such that~$Y$ and~$Z$ are of the same level.
Let $\eta = (\eta_Y,\eta_Z)$ be a \twd~in terms
of the flat geometric pictures on the left and in the middle of
Figure~\ref{fig:simplepole}. In particular, $\eta_Y$ has two simple poles at $q_1^{+}$ and $q_{2}^{+}$,
$\eta_Z$ has two simple poles at $q_1^{-}$ and $q_{2}^{-}$, and
$\Res_{q_i^{+}}(\eta_Y) + \Res_{q_i^{-}}(\eta_Z) = 0$ for $i= 1, 2$,
since the widths of the strips agree.

\begin{figure}
\begin{tikzpicture}[scale=0.8]
\tikzstyle{every node}=[font=\scriptsize]
\draw[dashed] (1,2) -- (1,3);
\draw (1,3) node(1){} -- (1,4) node(2){} -- (.5,4.5) node(3){} -- (1,5) node(4){} -- (1,6.5);
\draw[dashed] (1,6.5) -- (1,7.5);
\draw[dashed] (2.5,2) -- (2.5,3);
\draw (2.5,3) node(5){} -- (2.5,4) node(6){} -- (3,4.5) node(7){} -- (2.5,5) node(8){} -- (2.5,6.5);
\draw[dashed] (2.5,6.5) -- (2.5,7.5);

\node () at (1.75,.8) {$\Omega M_1 (2,-1,-1)$};

\node () at (1.8,7.7) {$q_1^+$};
\node () at (1.8,1.7) {$q_2^+$};
\node () at (6.8,7.7) {$q_2^-$};
\node () at (6.8,1.7) {$q_1^-$};
\node () at (.8,3) {$b_{1}$};
\node () at (.6,4.1) {$v_{2}$};
\node () at (.6,4.9) {$v_{1}$};
\node () at (.8,6.5) {$a_{1}$};
\node () at (2.9,4.1) {$v_{1}$};
\node () at (2.9,4.9) {$v_{2}$};
\node () at (2.7,3) {$b_{1}$};
\node () at (2.7,6.5) {$a_{1}$};

\begin{scope}[xshift=-.2cm]

\draw[dashed] (6,2) -- (6,3);
\draw (6,3) node(9){} -- (6,4) node(10){} -- (5,4) node(11){} -- (4.5,4.5) node(12){} -- (5,5)
             node(13){} -- (6,5) node(14){} -- (6,6.5);
\draw[dashed] (6,6.5) -- (6,7.5);
\draw[dashed] (7.5,2) -- (7.5,3);
\draw (7.5,3) node(15){} -- (7.5,4) node(16){} -- (8.5,4.5) node(17){} -- (9.5,4.5) node(18){} -- (10,5)
             node(19){} -- (9.5,5.5) node(20){} -- (8.5,5) node(21){} -- (7.5,5) node(22){} -- (7.5,6.5);
\draw[dashed] (7.5,6.5) -- (7.5,7.5);

\node () at (6.75,.8) {$\Omega M_2 (3,1,-1,-1)$};

\node () at (5.8,3) {$a_{2}$};
\node () at (7.7,3) {$a_{2}$};
\node () at (5.5,3.8) {$w_{5}$};
\node () at (4.6,4.1) {$w_{2}$};
\node () at (4.6,4.9) {$w_{1}$};
\node () at (5.5,5.2) {$w_{3}$};
\node () at (5.8,6.5) {$b_{2}$};
\node () at (7.7,6.5) {$b_{2}$};
\node () at (8,5.2) {$w_{5}$};
\node () at (8.9,5.4) {$w_{4}$};
\node () at (9.9,5.35) {$w_{2}$};
\node () at (9.95,4.65) {$w_{1}$};
\node () at (8.9,4.3) {$w_{3}$};
\node () at (8,4.1) {$w_{4}$};

\end{scope}

\begin{scope}[xshift=-.7cm]

\draw (12.5,1.5) node(23){} -- (14,1.5) node(24){} -- (14,3) node(25){} -- (14.5,3.5) node(26){} -- (14,4)
             node(27){} -- (14,7) node(28){} -- (15,7.5) node(29){} -- (16,7.5) node(30){} -- (16.5,8)
             node(31){} -- (16,8.5) node(32){} -- (15,8) node(33){} -- (14,8) node(34){} -- (14,9.5)
             node(35){} -- (12.5,9.5) node(36){} -- (12.5,8) node(37){} -- (11.5,8) node(38){} -- (11,7.5)
             node(39){} -- (11.5,7) node(40){} -- (12.5,7) node(41){} -- (12.5,4) node(42){} -- (12,3.5)
             node(43){} -- (12.5,3) node(44){} -- (12.5,1.5) node(45){} -- (14,1.5);
\draw [<->] (13.8,1.5) -- (13.8,3);
\draw [<->] (13.8,4) -- (13.8,5.5);
\draw [<->] (13.8,5.5) -- (13.8,7);
\draw [<->] (13.8,8) -- (13.8,9.5);

\node () at (13.25,.8) {$\Omega M_4 (3,2,1)$};

\node () at (12.3,2.25) {$b^{1}$};
\node () at (14.2,2.25) {$b^{1}$};
\node () at (12.1,3.1) {$v_{2}$};
\node () at (12.1,3.9) {$v_{2}$};
\node () at (14.3,3.1) {$v_{1}$};
\node () at (14.3,3.9) {$v_{1}$};
\node () at (12.3,5.5) {$a$};
\node () at (14.2,5.5) {$a$};
\node () at (12,6.8) {$w_{5}$};
\node () at (11.15,7) {$w_{2}$};
\node () at (11.15,7.9) {$w_{1}$};
\node () at (12,8.2) {$w_{3}$};
\node () at (12.3,8.75) {$b^{2}$};
\node () at (14.3,8.75) {$b^{2}$};
\node () at (14.5,8.2) {$w_{5}$};
\node () at (16.4,8.4) {$w_{2}$};
\node () at (16.4,7.6) {$w_{1}$};
\node () at (15.5,7.3) {$w_{3}$};
\node () at (13.25,1.3) {$c$};
\node () at (13.25,9.7) {$c$};
\node () at (13.6,2.25) {$\frac{1}{\vert t\vert}$};
\node () at (13.6,4.75) {$\frac{1}{\vert t\vert}$};
\node () at (13.6,6.25) {$\frac{1}{\vert t\vert}$};
\node () at (13.6,8.75) {$\frac{1}{\vert t\vert}$};
\end{scope}

\end{tikzpicture}
\caption{Smoothing two pairs of simple poles} \label{fig:simplepole}
\end{figure}
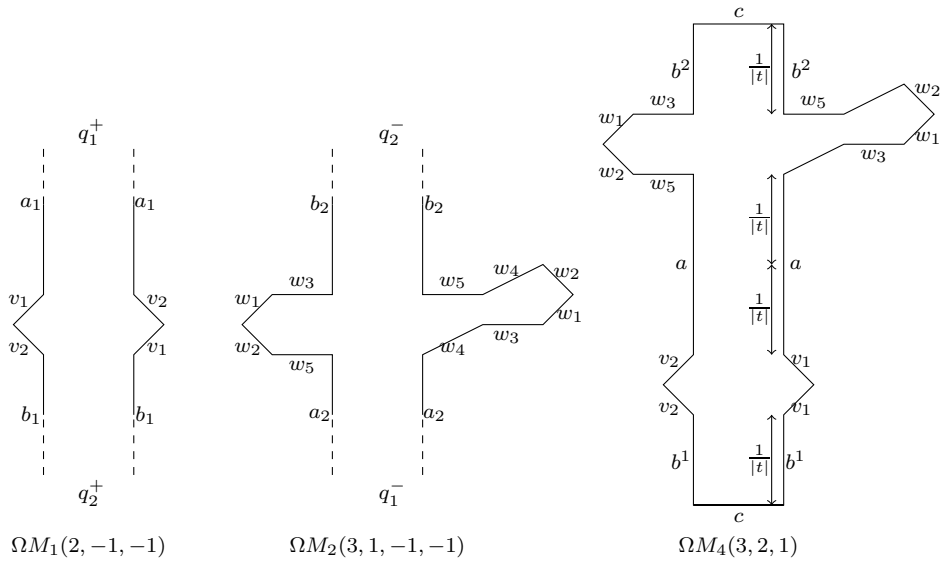

The pairs of points $(q_i^+,q_i^-)$, $i=1,2$ at the ends of the infinitely long
strips are then glued together to form the two nodes of $X$. Obviously~$\eta$
is compatible with~$\overline{\Gamma}$.
\par
We want to show that
the pointed stable differential defined by~$\omega = \eta$ with~marked points $z_i$ being
the zeros of $\omega$ is a limit as flat surfaces in the stratum~$\omoduli[4](3,2,1)$ degenerate.
For this purpose, it suffices to truncate each of the half-infinite cylinders at a
finite height $R=1/|t|$ and glue them along 'seam' strips of
some fixed height~$\varepsilon$, where the \emph{matching residues} at each pair of simple poles ensure that
the gluing procedure can be carried out by translation. For $t \to 0$ the flat surfaces obtained in this construction
visibly degenerate to $(X,\omega, z_1, z_2, z_3)$, and we will justify the convergence in general
in the proof at the end of the section. This example is an instance of components joined by {\em flat cylinders} in the
language of~\cite{RafiThTh} or, equivalently, the flat geometric viewpoint of classical plumbing (Proposition~\ref{prop:classical_plumbing}).
\par
\smallskip
In the second example, suppose $X$ is a nodal curve
with two irreducible components~$Y$ and~$Z$ of genus~$2$ and~$1$
respectively, joined at three nodes. Let $\overline{\Gamma}$
be a level graph of $X$ such that $Y \succ Z$.
Let $\eta = (\eta_Y,\eta_Z)$ be a \twd~
in terms of the flat geometric pictures in Figure~\ref{cap:ExTop} (without the slit in the interior)
and Figure~\ref{cap:ExBottom}. The flat surface~$Z$ consists of~$5$ copies of the Euclidean plane (with modifications in the center),
each of which is drawn in Figure~\ref{cap:ExBottom} as a disk of
large radius. Parallel edges with the same labels are identified via translation.
We define~$q_1^-$ to be the point at infinity of the first two top discs,
$q_3^-$ at infinity of the third, and $q_2^-$ at infinite of the bottom two
discs. The pairs of points $(q_i^+,q_i^-)$ for $i=1,2,3$ are
identified together to form the three nodes of $X$.
\par
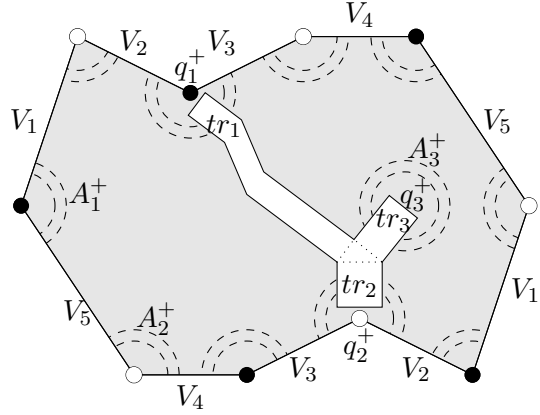
\begin{figure}
\begin{tikzpicture}[scale=1.5]
\fill (0,0) coordinate (y1) circle (2pt);
\fill (1,-0.5) coordinate (y2) circle (2pt);
\fill (2,0) coordinate (y3) circle (2pt);
\fill (3,0) coordinate (y4) circle (2pt);
\fill (4,-1.5) coordinate (y5) circle (2pt);
\fill (3.5,-3) coordinate (y6) circle (2pt);
\fill (2.5,-2.5) coordinate (y7) circle (2pt);
\fill (1.5,-3) coordinate (y8) circle (2pt);
\fill (0.5,-3) coordinate (y9) circle (2pt);
\fill (-0.5,-1.5) coordinate (y10) circle (2pt);

\filldraw[color=black!10!] (y1) -- (y2)--    (y2) -- (y3)--   (y3) -- (y4)--   (y4) -- (y5)-- (y5) -- (y6)-- (y6) -- (y7)--   (y7) -- (y8)--   (y8) -- (y9)--   (y9) -- (y10)--  (y10) -- (y1);
 \draw[postaction={decorate}] (y1) -- (y2);  \draw[postaction={decorate}] (y2) -- (y3); \draw[postaction={decorate}] (y3) -- (y4); \draw[postaction={decorate}] (y4) -- (y5); \draw (y5) -- (y6);\draw (y6) -- (y7); \draw[postaction={decorate}] (y7) -- (y8); \draw[postaction={decorate}] (y8) -- (y9); \draw[postaction={decorate}] (y9) -- (y10);\draw[postaction={decorate}] (y10) -- (y1);
\draw (y1) circle (2pt);
\draw (y3) circle (2pt);
\draw (y5) circle (2pt);
\draw (y7) circle (2pt);
\draw (y9) circle (2pt);
\draw (y1) --(y2) node [above, midway] {$V_2$};
\draw (y2) --(y3) node [above left, midway] {$V_3$};
\draw (y3) --(y4) node [above, midway] {$V_4$};
\draw (y4) --(y5) node [right, midway] {$V_5$};
\draw (y5) --(y6) node [right, midway] {$V_1$};
\draw (y6) --(y7) node [below, midway] {$V_2$};
\draw (y7) --(y8) node [below, midway] {$V_3$};
\draw (y8) --(y9) node [below, midway] {$V_4$};
\draw (y9) --(y10) node [below, midway] {$V_5$};
\draw (y10) --(y1) node [left, midway] {$V_1$};

\fill[color=white] (y1) circle (2pt);
\fill  (y2) circle (2pt);
\fill[color=white]  (y3) circle (2pt);
\fill  (y4) circle (2pt);
\fill[color=white]  (y5) circle (2pt);
\fill  (y6) circle (2pt);
\fill[color=white]  (y7) circle (2pt);
\fill  (y8) circle (2pt);
\fill[color=white]  (y9) circle (2pt);
\fill  (y10) circle (2pt);
\begin{scope}
\clip (y1)--(y2)--(y3)--(y4)--(y5)--(y6)--(y7)--(y8)--(y9)--(y10);
\draw[dashed] (y1) circle (0.3cm);
\draw[dashed] (y2) circle (0.3cm);
\draw[dashed] (y3) circle (0.3cm);
\draw[dashed] (y4) circle (0.3cm);
\draw[dashed] (y5) circle (0.3cm);
\draw[dashed] (y6) circle (0.3cm);
\draw[dashed] (y7) circle (0.3cm);
\draw[dashed] (y8) circle (0.3cm);
\draw[dashed] (y9) circle (0.3cm);
\draw[dashed] (y10) circle (0.3cm);
\end{scope}
\begin{scope}
\clip (y1)--(y2)--(y3)--(y4)--(y5)--(y6)--(y7)--(y8)--(y9)--(y10);
\draw[dashed] (y1) circle (0.4cm);
\draw[dashed] (y2) circle (0.4cm);
\draw[dashed] (y3) circle (0.4cm);
\draw[dashed] (y4) circle (0.4cm);
\draw[dashed] (y5) circle (0.4cm);
\draw[dashed] (y6) circle (0.4cm);
\draw[dashed] (y7) circle (0.4cm);
\draw[dashed] (y8) circle (0.4cm);
\draw[dashed] (y9) circle (0.4cm);
\draw[dashed] (y10) circle (0.4cm);
\end{scope}
\draw[dashed] (2.9,-1.5) circle (0.3cm);
\draw[dashed] (2.9,-1.5) circle (0.4cm);
\node[above] at (y2) {$q_{1}^{+}$};
\node[below] at (y7) {$q_{2}^{+}$};
\node[] at (3,-1.4) {$q_{3}^{+}$};
\node[] at (3.1,-1) {$A_{3}^{+}$};
\node[] at (0.1,-1.4) {$A_{1}^{+}$};
\node[] at (.7,-2.5) {$A_{2}^{+}$};

\begin{scope}[xshift=2.7cm,yshift=-2cm]
\coordinate (a) at (0, 0);
\coordinate (b) at (-0.4, 0);
\coordinate (c) at (-0.25, 0.2);
\coordinate (d) at ($ (a)!.4cm!90:(b) $);
\coordinate (e) at ($ (b)!.4cm!-90:(a) $);
\coordinate (f) at ($ (a)!.5cm!-90:(c) $);
\coordinate (g) at ($ (c)!.5cm!90:(a) $);
\coordinate (h) at ($ (b)!1cm!90:(c) $);
\coordinate (i) at ($ (c)!1cm!-90:(b) $);
\coordinate (j) at ($ (h)!.5cm!60:(i) $);
\coordinate (k) at ($ (i)!.5cm!-120:(h) $);
\coordinate (l) at ($ (j)!.4cm!90:(k) $);
\coordinate (m) at ($ (k)!.4cm!-90:(j) $);

\filldraw[color=white!50!] (a)--(d)--(e)--(b)--(h)--(j)--(l)--(m)--(k)--(i)--(c)--(g)--(f)--(a);
\draw[] (a)--(d)--(e)--(b)--(h)--(j)--(l)--(m)--(k)--(i)--(c)--(g)--(f)--(a);
\draw[dotted] (a)-- (b) --(c)-- (a);
\node[above]  at ($ (d)!.5!(e) $) {$t r_{2}$};
\node[below ]  at ($ (f)!.84!(g) $) {$t r_{3}$};
\node[  below right]  at ($ (l)!.45!(m) $) {$t r_{1}$};
\end{scope}

\end{tikzpicture}
\caption{Top level flat surface~$Y$ inside $\omoduli[2](1,1,0)$}  \label{cap:ExTop}
\end{figure}

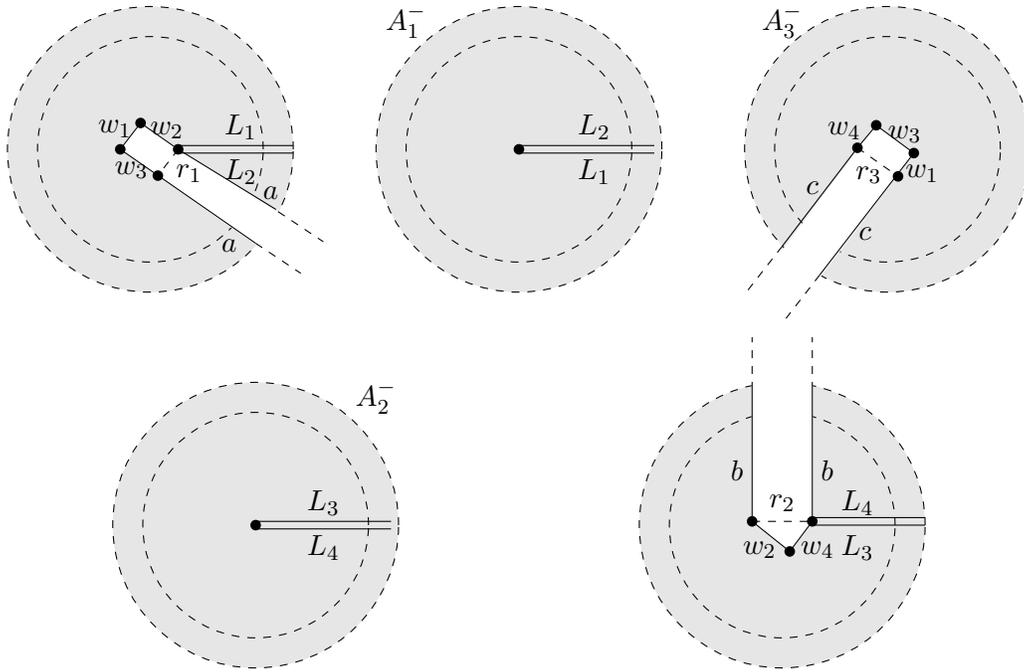
\begin{figure}
\begin{tikzpicture}
\begin{scope}[xshift=-1.4cm]
\draw[dashed, fill=black!10] (0,0) coordinate (Q) circle (1.9cm);
     \node(A1) at (1.9,0.0){};
     \node(A2) at (-0.1,0){};
     \node(A3) at (0,0){};
     \node(A4) at (0,0){};

\draw[dashed, fill=black!10] (0,0) coordinate (Q) circle (1.5cm);

\draw[] (-0.4, 0) coordinate (q1) -- (0.37, 0) coordinate (q2);

\filldraw[color=white!50!](-0.4, 0) -- (2,-1.68)--(2.3,-1.23) --(0.37, 0)-- (-0.13, 0.35)-- (-.4,0) ;

\draw[] (0.33,0.05) -- (1.9,0.05);
\draw[] (0.4,-0.05) -- (1.9,-0.05);

\draw[dashed](0.37, 0)--(0.1, -0.35) coordinate (q4);
\draw[](-0.13, 0.35)  coordinate (q3)--(-0.4, 0);
\draw[](-0.13, 0.35)--(0.37, 0);

\draw[] (-0.4, 0) -- (1.4,-1.25);
\draw[] (1.67,-0.8) -- (0.37, 0);

\node[above](R) at (1.2,0){$L_1$};
\node[below](R) at (1.2,0){$L_2$};

\draw[dashed] (1.4,-1.25) -- (2,-1.65);
\draw[dashed] (1.7,-0.83) -- (2.3,-1.23);

\node[above](R) at (-.47,.03){$w_{1}$};
\node[above](R) at (.22,0.01){$w_{2}$};
\node[above](R) at (-.25,-.52){$w_{3}$};

\node[above](R) at (1.6,-0.8){$a$};
\node[above](R) at (1.05,-1.5){$a$};
\node[below right] at ($(.3, 0.05) !.3! (0, -0.35) $) {$r_{1}$};

\fill (q1)  circle (2pt);\fill (q2) circle (2pt);
\fill (q3)  circle (2pt);
\fill (q4)  circle (2pt);

\end{scope}

\begin{scope}[xshift=-3.5cm]

\draw[dashed, fill=black!10] (7,0) coordinate (Q) circle (1.9cm);
     \node(A5) at (8.9,0.0){};
     \node(A6) at (6.9,0){};
     \node(A7) at (7,0){};
     \node(A8) at (7,0){};

\draw[dashed, fill=black!10] (7,0) coordinate (Q) circle (1.5cm);

\fill (Q)  circle (2pt);
\draw[] (7,0.05) -- (8.8,0.05);
\draw[] (7,-0.05) -- (8.8,-0.05);

\node[above](R) at (8,0){$L_2$};
\node[below](R) at (8,0){$L_1$};

\node at (5.5,1.7) {$A_{1}^{-}$};

\end{scope}

\begin{scope}[xshift=3cm]

\draw[dashed, fill=black!10] (5.4,0) coordinate (Q) circle (1.9cm);
     \node(A1) at (1.9,0.0){};
     \node(A2) at (-0.1,0){};
     \node(A3) at (0,0){};
     \node(A4) at (0,0){};

\draw[dashed, fill=black!10] (5.4,0) coordinate (Q) circle (1.5cm);
\filldraw[color=white!50!](5.75, -0.05) --(4.5,-1.7)--(4,-1.35)--  (5.25, 0.32)    --cycle;

\draw[]  (5.75, -0.05) --(4.5,-1.67);
\draw[]  (5, 0.02) --(4,-1.29);

\draw[dashed] (4.5,-1.67) -- (4.05,-2.25);
\draw[dashed] (4,-1.29) -- (3.5,-1.93);
\draw[dashed] (5,0.02) coordinate (q1) -- (5.54,-0.36) coordinate (q2);
\draw[] (5.25,0.32) coordinate (R) -- (5.75,-0.05) coordinate (q3);
\draw[] (R) -- (5,0.02);

\fill (q1)  circle (2pt);\fill (q2) circle (2pt);
\fill (q3)  circle (2pt);
\fill (R)  circle (2pt);
\node[below](R) at (5.63,0.45){$w_{3}$};
\node[below](R) at (5.86,-0.07){$w_{1}$};
\node[left](R) at (5.18,0.25){$w_{4}$};

\node[below](R) at (4.4,-0.3){$c$};
\node[below](R) at (5.1,-0.9){$c$};

\node[below] at ($(5,0.02) !.3! (5.5,-0.35) $) {$r_{3}$};
\node at (4,1.7) {$A_{3}^{-}$};
\end{scope}

\begin{scope}[yshift=-5cm]

\draw[dashed, fill=black!10] (0,0) coordinate (Q) circle (1.9cm);
     \node(A1) at (1.9,0.0){};
     \node(A2) at (-0.1,0){};
     \node(A3) at (0,0){};
     \node(A4) at (0,0){};

\draw[dashed, fill=black!10] (0,0) coordinate (Q) circle (1.5cm);
\fill (Q)  circle (2pt);
\node[above](R) at (0.9,0){$L_3$};
\node[below](R) at (0.9,0){$L_4$};

\node at (1.6,1.7) {$A_{2}^{-}$};

\draw[dashed, fill=black!10] (7,0) coordinate (Q) circle (1.9cm);
     \node(A5) at (8.9,0.0){};
     \node(A6) at (6.9,0){};
     \node(A7) at (7,0){};
     \node(A8) at (7,0){};

\draw[dashed, fill=black!10] (7,0) coordinate (Q) circle (1.5cm);

\filldraw[color=white!50!] (7.4,2) -- (7.4,0.05) -- (7.1, -0.35) --(6.6,0.05) --(6.6,2)    --cycle;

\draw (7.4,1.8) -- (7.4,0.05);
\draw (6.6,0.05) --(6.6,1.8) ;

\draw[dashed](7.4,1.8)-- (7.4,2.5);
\draw[dashed](6.6,1.8)-- (6.6,2.5);

\node[above](R) at (8,0){$L_4$};
\node[below](R) at (8,0){$L_3$};
\node[below](R) at (6.4,1){$b$};
\node[below](R) at (7.6,1){$b$};

\node[below](R) at (6.7,-0.08){$w_{2}$};
\node[below](R) at (7.47,-0.08){$w_{4}$};

\draw[] (7.4,0.1) -- (8.9,0.1);
\draw[] (7.4,0) -- (8.9,0);

\node[above] at ($(6.6, 0.05) !.5! (7.4, 0.05) $) {$r_{2}$};

\draw[] (0,0.05) -- (1.8,0.05);
\draw[] (0,-0.05) -- (1.8,-0.05);
\draw[] (6.6, 0.05)  -- (7.1, -0.35) --(7.4, 0.05);
\draw[dashed] (6.6, 0.05) -- (7.4, 0.05);

\fill[] (6.6,0.05)  circle (2pt);
\fill[] (7.4,0.05)  circle (2pt);
\fill[] (7.1,-0.35)  circle (2pt);
\end{scope}
\end{tikzpicture}
\caption{Bottom level flat surface~$Z$ inside $\omoduli[1](8,-2,-3,-3)$} \label{cap:ExBottom}
\end{figure}
\par
Since the orders of $\eta$ at $q_i^+$ and $q_i^-$ add up
to~$-2$ and the global residue condition automatically holds by the residue theorem on~$Z$,
we conclude that~$\eta$ is a \twd~compatible with
$\overline{\Gamma}$. The stable differential $\omega$ associated to $\eta$ is equal to $\eta_Y$ on $Y$, and is identically zero on $Z$.
Denoting by $z$ the unique zero of $\eta_Z$, we want to show that
$(X, \omega, z)$ is a limit of flat surfaces in~$\omoduli[5](8)$, by constructing a family
of flat surfaces in this stratum which visibly degenerates to $(X, \omega, z)$.
\par
Our strategy is to remove from~$Y$ a small disk (as a union of metric half-disks) around each point~$q_i^+$,
scale $\eta_Z$ by a smaller factor, take a disk of the same size under the flat geometric presentation for the pole $q_i^-$
in the rescaled surface $Z$, and glue it into~$Y$ along the annuli of the same size around $q_i^{\pm}$.
More precisely, we want to glue along the annuli~$A_i^+$ to~$A_i^-$ for $i=1,2,3$, as presented in Figures~\ref{cap:ExTop} and~\ref{cap:ExBottom}.
To first approximation the total angles around $q_i^+$ and $q_i^-$
match, because the orders of $\eta_Y$ and $\eta_Z$ add up to~$-2$ at the two branches of a node.
The issue is that we cannot directly glue their boundaries, because in each annulus
around $q_i^-$ a slit of 'size' of the residue at that pole is missing.
In order to remedy this problem, we slit the flat surface~$Y$ appropriately,
as drawn in Figure~\ref{cap:ExTop}. The global residue condition ensures that we can remove a central polygon (given by the dashed triangle in this example)
from~$Y$, and slit from the points $q_i^+$ to get a surface in which we
can glue to~$Z$ along the modified annuli~$A_i^+$ with $A_i^-$.
\par
We next make the degeneration process more precise by specifying the
sizes of the annuli and the slits. Fix $0< \delta<1$ close to one. For $t\in (0, \varepsilon)$ sufficiently small,
we may assume that the circles of radius $t^{1/2}$ under the flat metric of~$(Y, \eta_Y)$ around the
points~$q_i^+$ are disjoint and contain no other special points inside (except for $q_i^+$). Moreover, we may assume that
the circles of radius $t^{1/2}\delta$ under the flat metric of~$(Z, t \eta_Z)$ (i.e., rescale $Z$ by $t$) contain all the interior edges of~$(Z, \eta_Z)$ (those with
labels~$w_i$), because these circles have large radius $t^{-1/2}\delta $ under the original flat metric of $\eta_Z$.
\par
After these preparations, we can glue the annuli (modified by the neighborhoods of size $t r_i$ around the slits in $Y$)
between the circles of radius~$t^{1/2}$
and~$t^{1/2}\delta$ with labels~$A_i^+$ to those with labels~$A_i^-$. In particular as $t\to 0$, the annuli are shrinking to $q_i^+$ on $(Y, \eta_Y)$ and expanding to $q_i^-$ on $(Z, \eta_Z)$. This gives the desired family converging to~$(X,\omega, z)$. Again, we will justify the convergence later
in the proof. This example is an instance of components joined by
{\em expanding annuli} in the language of \cite{RafiThTh}, or, equivalently
the flat geometric viewpoint of higher order plumbing (Theorem~\ref{thm:plumbgeneral}).

\subsection{Construction of degenerating families in the general case}
\label{sec:ConDegFam}

The first step is to prove that a slit as in the second example above exists
in general and carry out the construction inductively with
respect to a given level graph.
\par
For simplicity of exposition, we restrict to the case of strata of holomorphic type. The same method works
for the case of strata of meromorphic type. There is only one place where the global residue condition plays a role to distinguish the two cases, and we will remark on it in the construction below. In addition, we may assume that the level graph does not have an edge joining two
vertices on the same level, or equivalently, that a compatible twisted differential has no simple pole at any node. This is because smoothing such a simple polar node with matching residues is a local procedure ensured by classical plumbing, which can be performed in the
flat geometric picture according to the first example above.
\par
Let $(X,\omega,q_1,\ldots,q_N)$ be a Riemann surface with a
meromorphic differential~$\omega$ and marked points~$q_1,\ldots,q_N$
that form a subset of zeros and ordinary points of $\omega$. We let $X^0$ be the open Riemann
surface with disks of radius~$\delta R$ removed around the points~$q_i$, where $0 < \delta<1$ is chosen sufficiently close to one. We refer to the annuli
between the circles of radius~$\delta R$ and~$R$ as
\emph{boundary annuli}.
\par
We next work towards the definition of a residue slit, i.e,
a collection of broken lines in the surface
around which we will modify neighborhoods of size given by residues. We remark that a residue slit is in general more complicated than just a slit. Given a tuple of sufficiently small complex numbers~$T = (r_1,\ldots, r_N)$ with $\sum_{i=1}^N
r_i = 0$, we fix a permutation $\pi \in S_N$ such that the slopes of $r_{\pi(1)},
\ldots,r_{\pi(N)}$ are monotone on~$S^1$. Let $P = P(T,\pi)$ be the polygon whose
edges are given by the vectors $r_{\pi(1)}, \ldots,r_{\pi(N)}$ consecutively. It follows
that $P(T, \pi)$ is convex. Let~$B(P)$ be the barycenter of~$P$. Let~$p$ be a point in~$X^{0}$,
disjoint from the zeros and poles of $\omega$. We place $P$ inside $X^0$ with $B(P) = p$.
\par
A {\em residue slit} for~$(T,\pi)$ is then defined to be a collection of broken lines $(b_1,\ldots,b_n)$ with
the following properties:
\begin{itemize}
\item Each broken line $b_i$ starts with the segment
from $p$ to the midpoint of the edge $e_{\pi^{-1}(i)}$ of the polygon~$P(T,\pi)$, and then
connects to $q_i$. We denote by $b_{ij}$ the line segments constituting $b_i$ and by $\theta(b_{ij})$ the slopes of $b_{ij}$.
\item The broken lines $b_i$ do not intersect (except for the starting point~$p$)
and they are disjoint from the zeros and poles of $\omega$ (except for the endpoints~$q_i$).
 \item The slopes $\theta(b_{ij})$ are different from that of $\pm r_i$ for each $i$.
\end{itemize}
\par
For $t\in (0, \varepsilon)$ sufficiently small, let $t T = (tr_1,\ldots, tr_N)$. We define the {\em surface $X^0_t = X^0_t(T,\pi)$
obtained by modifying neighborhoods of size $tT$ around the residue slit} as follows. Remove $P(T,\pi)$
from $X^0$. In the case that
 $\langle \theta(b_{ij}), r_i \rangle > 0$ (as in Figure~\ref{cap:ExTop} for
all three broken lines), we \emph{remove} neighborhood parallelograms swept out by planar segments with holonomy
vector~$tr_i$ centered at points of $b_{ij}$ in general, and glue the parallel sides of these removed parallelograms in pairs.
On the other hand for each of the segments where $\langle \theta(b_{ij}), r_i \rangle < 0$,  we
\emph{add} a parallelogram swept out by segments of holonomy~$tr_i$ to the existing flat
surface and glue the pieces as indicated in Figure~\ref{cap:ResSlitBack}. A small modification one needs to carry out is that near a turning point,
e.g., the right endpoint of $a_3$ or the left endpoint of $d_1$ in Figure~\ref{cap:ResSlitBack}, we move the removed parallelogram region
up or down by $tr_i/2$, which ensures that adding a parallelogram works near the turning point as the residue slit changes its direction.
The size of~$t$ is constrained by the requirement that the neighborhoods of the residue
slit used in this construction are disjoint and do not contain the zeros and poles of~$\omega$ on $X^0$. In this way we obtain
the desired surface $X^0_t$.

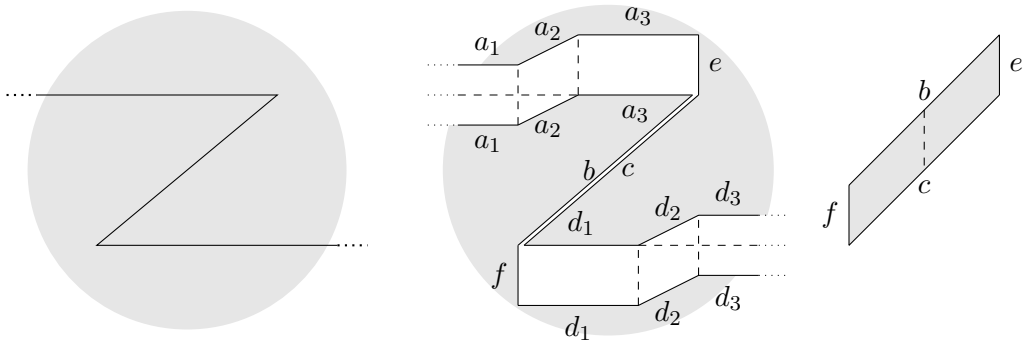
\begin{figure}[htb]
  \centering
\begin{tikzpicture}[scale=.8]

\begin{scope}[xshift=-7cm]
\fill[ color=black!10!] (3.5,2.25) circle (2.65);
     \draw (1,3.5) -- (5,3.5) -- (2,1) -- (6,1);
     \draw[dotted,thick] (.5,3.5) -- (1,3.5)
                         (6,1) -- (6.5,1);
                         \end{scope}
\begin{scope}
\fill[ color=black!10!] (3.5,2.25) circle (2.75);

     \fill[color=white]
      (0,3) -- (2,3) -- (3,3.5) -- (5,3.5) -- (5,4.5) --
      (3,4.5) -- (2,4) -- (0,4) -- cycle;
     \draw
      (1,3) -- (2,3) node[midway,below] {$a_1$}
      (2,3) -- (3,3.5) node[midway,below] {$a_2$}
      (3,3.5) -- (4.9,3.5) node[midway,below] {$a_3$}
       (5,3.5) -- (5,4.5) node[midway,right] {$e$}
      (5,4.5) -- (3,4.5) node[midway,above] {$a_3$}
      (3,4.5) -- (2,4) node[midway,above] {$a_2$}
      (2,4) -- (1,4) node[midway,above] {$a_1$};

      \fill[color=white]
      (2,0) -- (4,0) -- (5,.5) -- (7,.5) -- (7,1.5) --
      (5,1.5) -- (4,1) -- (2,1) -- cycle;
     \draw
      (2,0) -- (4,0) node[midway,below] {$d_1$}
      (4,0) -- (5,.5) node[midway,below] {$d_2$}
      (5,.5) -- (6,.5) node[midway,below] {$d_3$}
      (6,1.5) -- (5,1.5) node[midway,above] {$d_3$}
      (5,1.5) -- (4,1) node[midway,above] {$d_2$}
      (4,1) -- (2.1,1) node[midway,above] {$d_1$}
      (2,1) -- (2,0) node[midway,left] {$f$};

 \fill[color=white]
      (2,1) -- (4.9,3.5) -- (5,3.5) -- (2.1,1) -- cycle;

     \draw (2,1) -- (4.9,3.5) node[midway,left] {$b$} ;
       \draw (2.1,1) -- (5,3.5)  node[midway,right] {$c$};
\draw[dashed,very thin] (1,3.5) -- (3,3.5);
\draw[dashed,very thin] (4,1) -- (6,1);

     \draw[dotted] (.5,3.5) -- (1,3.5)
                         (6,1) -- (6.5,1)
                         (.5,4) -- (1,4)
                         (6,1.5) -- (6.5,1.5)
                         (.5,3) -- (1,3)
                         (6,.5) -- (6.5,.5);
\draw[dashed] (2,3) -- (2,4);
\draw[dashed] (3,3.5) -- (3,4.5);
\draw[dashed] (4,0) -- (4,1);
\draw[dashed] (5,0.5) -- (5,1.5);

     \begin{scope}[xshift=.5cm,yshift=1cm]
     \fill[color=black!10!]
      (7,0) -- (7,1) -- (9.5,3.5) -- (9.5,2.5) -- cycle;
     \draw
      (7,0) -- (7,1) node[midway,left] {$f$}
      (7,1) -- (9.5,3.5) node[midway,above] {$b$}
      (9.5,3.5) -- (9.5,2.5) node[midway,right] {$e$}
      (9.5,2.5) -- (7,0) node[midway,below] {$c$};
     \draw[dashed] (8.25,2.25) -- (8.25,1.25);
      \end{scope}
      \end{scope}
    \end{tikzpicture}
\caption{A piece of a residue slit and the gluing procedure in the case of a backwards slope} \label{cap:ResSlitBack}
\end{figure}

\par
\begin{lm} \label{le:residueslits}
Given a pointed flat surface $(X,\omega,q_1,\ldots,q_N;p)$, for any tuple~$T$ of sufficiently small complex numbers with sum equal to zero,
there exists a residue slit.
\end{lm}
\par
\begin{proof} We produce the broken lines inductively. At each step we need to find
a path from $q_i$ to $p$, avoiding finitely many contractible subsets, namely
the zeros and poles of $\omega$ on $X^0$ and the union of broken lines constructed in the preceding steps.
Such a path clearly exists and can be straightened into broken lines, avoiding slopes of $\pm r_i$.
\end{proof}
\par
Next, we show that every meromorphic differential (without simple poles)
can be presented similarly to Figure~\ref{cap:ExBottom}. The building blocks
of this construction are basic domains, as introduced by
Boissy (see \cite[Section~3.3]{boissymero}). Compared to his definition
we ignore infinite cylinders (i.e., simple poles) because such nodes can be smoothed out locally by the classical plumbing, but we need to allow more flexibility in the direction that we slit.
\begin{figure}[htb]
	\centering
	\begin{tikzpicture}[scale=.8]
	\coordinate (A) at (0,0);
	\coordinate (G) at (2.5,0);
	\path (A) --++(45:1.5cm) coordinate (B) --++(90:.5cm) coordinate (C) --++(0:1cm) coordinate (D)
	--++(90:.75cm) coordinate (E) --++(45:1.5cm) coordinate (F);
	\path[draw] (G) --++(45:1.5cm) coordinate (H) --++(0:.5cm) coordinate (I) --++(45:3.25cm) coordinate (J);
	\fill[color=black!10!]  (A) -- (B) -- (C) -- (D) -- (E) -- (F) -- (0,3.375) -- (A);
	\fill[color=black!10!]  (G) -- (H) -- (I) -- (J) -- (0:6.355cm) -- (G);
	\draw[black] (A) -- (B) -- (C) -- (D) -- (E) -- (F)
	      (G) -- (H) -- (I) -- (J);	

	\foreach \x in {B,C,D,E,H,I}
	\fill[black] (\x) circle (2pt);
	\end{tikzpicture}
	\caption{Two basic domains in the direction $\theta =\pi/4$}
	\label{fig:basicdomains}
\end{figure}
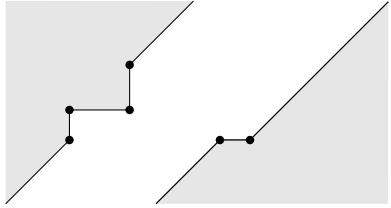
Consequently we define
for a fixed angle $\theta$ a {\em basic domain in the direction $\theta$} to
be a half-plane with broken polygonal boundary, such that the infinite
boundary rays have direction~$\pm e^{i\theta}$ and the finite edges~$w_{j}$
of the boundary satisfy $\left<w_{j},e^{i\theta}\right>>0$ as pictured in Figure~\ref{fig:basicdomains}.
\par
\begin{lm} \label{le:infzipp}
Let $(X,\omega)$ be a meromorphic differential without simple poles, but with~$l$ poles of higher order $|m_{n-l+1}|, \ldots, |m_n|$, respectively. Then for almost every
direction~$\theta$ there exist $2\sum_{i=1}^\ell (|m_{n-l+i}|-1)$
basic domains in the direction $\theta$  such that $(X,\omega)$ is
obtained by gluing the boundary segments of the basic domains by translation.
Moreover, the infinite boundary rays are glued in a way such that the set of basic domains
is partitioned in $l$ cycles of lengths $2|m_{n-l+1}|-2,\ldots, 2|m_{n}|-2$, respectively.
\par
If~$(X,\omega)$ can be represented by basic domains in the direction~$\theta$, then all flat surfaces in a neighborhood
of~$(X,\omega)$ also have this property.
\end{lm}
\par
\begin{proof}
This lemma is a restatement of the infinite zippered rectangle construction of Boissy (\cite{boissymero}).
\end{proof}
\par
\medskip
We are ready to construct the desired family of flat surfaces by induction
on the number of levels. Let~$(X,\omega,z_1,\ldots,z_n)$ be a pointed stable differential
associated with a \twd~$\eta$ of type~$\mu$, where $\eta$ is compatible with a full order~$\overline{\Gamma}$ on
$X$. In the sequel we fix a number $0< \delta<1$ sufficiently close to $1$.
\par
First consider the case when $\overline{\Gamma}$ has exactly two levels. We
want to glue the two levels as in the second example above. For each connected
component~$Y_j$ of top level, let $q_1,\ldots,q_N$ be the nodes where $Y_j$ is
joined to the lower level. For $t \in (0,\varepsilon)$
sufficiently small and a positive integer $k$, we construct residue slits using
Lemma~\ref{le:residueslits} for $(t^{k}T_j,\pi)$, where the residues $r_i$ in $T_j$
are given by $r_i = \Res_{q^{-}_i} (\eta_{v^{-}(q_i)})$, i.e., they arise from the lower level branches at the nodes~$q_i$,
and the permutation~$\pi$ rearranges the residues in~$T_j$ such that their slopes are in some monotone order. Then we add or remove parallelogram neighborhoods of
size $t^{k}T_j$ around the residue slits on the open flat surfaces~$(Y_j^0, \eta_{Y_j^0})$, where $Y_j^0$ is obtained by removing discs of radius $t^{k-1/2}\delta $ around all~$q_i^+$ from $(Y_j, \eta_{Y_j})$ with the induced flat metric. As a result, we obtain a continuous family of open flat surfaces
$(X^{0}_t,\omega^{0}_t)$ over $(0,\varepsilon)$, where the removed discs are arbitrarily small as $t\to 0$.
\par
Next we scale the subsurfaces $(Z_j, \eta_{Z_j})$ of lower level
as $(Z_j, t^{k}\eta_{Z_j})$ and present them using Lemma~\ref{le:infzipp}
for any $\theta$ satisfying the conclusion of this lemma.
For $t$ small enough, we can assume that the union of
{\em inscribed half-discs} of radius~$t^{k-1/2}$
centered at the common limit point as $t \to 0$ of all broken
line segments
contains all finite boundary segments in the basic
domain presentation for each pole $q_i^-$ in $(Z_j, t^{k}\eta_{Z_j})$.
We then glue these inscribed half-discs into $(X^{0}_t,\omega^{0}_t)$, by identifying the boundary annuli $A_{i}^{-}$ of inner radius $t^{k-1/2} \delta$
and outer radius $t^{k-1/2}$ with the boundary annuli $A_{i}^{+}$ of the same size (modified by the residue slit construction) for all $q_i$. We have thus finished the construction for the case of two levels.
\par
Now for general $\overline{\Gamma}$, the induction step follows from the same idea. Let $\ell(\cdot)$ be any integral valued level function on $\Gamma$
that gives the full order $\overline{\Gamma}$, mapping the top level to zero.
Suppose that we have inductively constructed families of flat
surfaces~$(Y_j(t),\omega_j(t))$ for each of the connected components
of~$\overline{\Gamma}_{>L}$ for some level~$L$. For given index~$j$, let
$q_1,\ldots,q_N$ be the nodes joining $Y_j$ to level~$L$. If $Y_j$ does not contain
any marked point as a prescribed pole, then the corresponding tuple~$T_j$ of residues at all $q_i^-$ has sum equal to zero
by the global residue condition. Here is the only place where the global residue condition makes a difference
if $Y_j$ contains a marked pole. In that case, there is
no global residue condition imposed on it, but we can leverage on the flat geometric presentation of the marked pole to create a
residue $r_0$ such that $r_0 = -\sum_{i=1}^N r_N$. Then the construction
as in Figure~\ref{cap:ExTop} still goes through by taking the residue tuple $T_j = (r_0, r_1, \ldots, r_N)$.
\par
Next we construct a residue slit in the family $(Y_j(t),\omega_j(t))$ as in
Lemma~\ref{le:residueslits}, which has been stated for a single flat surface.
A closer look at the proof of that lemma reveals that the changes of
complex structures happen in some neighborhoods of the residue slits
constructed in the previous inductive steps. Since these finitely many
contractible neighborhoods are disjoint from the points $q_1^+,\ldots,q_N^+$ on $Y_j(t)$, the method of the proof still applies in families.
Finally we scale the subsurfaces $(Z_v, \eta|_{Z_v})$ with level $\ell(v)=L$ as $(Z_v, t^{-L}\eta_{Z_v})$, present
them as in Lemma~\ref{le:infzipp}, remove disks of radius $t^{-L-1/2}\delta $ around $q_i^{+}$ from $Y_j(t)$,
and glue in the inscribed half-disks of radius $t^{-L-1/2}\delta$ in $(Z_v, t^{-L}\eta_{Z_v})$ along the boundary annuli of inner radius~$t^{-L-1/2}\delta$
and outer radius~$t^{-L-1/2}$ (modified by the residue slits of size $t^{-L}T_j$ in $Y_j(t)$). The term $1/2$ in the exponents of $t$ ensures that the annuli
are expanding on the original surface $(Z_v, \eta_{Z_v})$ and shrinking on $(Y_j(t), \omega_j(t))$ as $t\to 0$.
This thus completes the construction for the general case. As a result, we have constructed a family of flat surfaces
$(X_t, \omega_t)$, which is continuous over $(0,\varepsilon)$ and lies in the given stratum of type $\mu$.
\par
\begin{rem}
The flat geometric construction is over a real ray, i.e., the base parameter~$t$ is real. It is natural to ask whether one can similarly construct such a family that varies holomorphically over a punctured disk. We point out in general there is a \emph{monodromy} obstruction. For instance, recall the construction by gluing Figure~\ref{cap:ExBottom} into Figure~\ref{cap:ExTop} to form a nearby flat surface in the resulting family. If $t$ varies as a complex parameter, it means that the residue cuts $r_i$ have to vary their arguments besides shrinking the lengths. If the corresponding cycle of $r_i$ is not homologous to zero, when its argument turns back, the Picard-Lefschetz formula implies that it affects those periods that have nonzero intersection with $r_i$, e.g., the periods arising from the boundary of the residue slit in Figure~\ref{cap:ExTop} cannot remain unchanged during the entire residue turning process. Consequently one has to vary holomorphically those affected periods to cancel out the excess periods caused by monodromy, which is less visible in terms of flat geometric coordinates. This phenomenon is closely related to extendability of period coordinates to the boundary of the strata compactification. In~\cite{BCGGM2} we will study this question systematically.
\end{rem}
\par
\begin{rem}
The flat geometric construction of gluing higher
order poles along boundary annuli with residue slits is analogous to higher order plumbing in Theorem~\ref{thm:plumbgeneral}.
In particular, the residue slit construction plays a similar role as the modification differential in the plumbing construction.
However, even over a real ray the two constructions in general do not give the same families if the residues are nonzero.
To see this, note that the residue slit construction involves various
choices, and so does the plumbing construction, especially when invoking Lemma~\ref{lm:merge} to merge the zeros that were dispersed. Hence they can agree only if we make very specific choices in the respective constructions.
\end{rem}
\par

\subsection{Flat geometric proof of Theorem~\ref{thm:main}:
conditions are sufficient}

So far we have constructed a family of flat surfaces $(X_t, \omega_t)$ that vary continuously
with $t\in (0,\varepsilon)$ sufficiently small. The remaining step
is to show that it converges as $t \to 0$ to the pointed stable differential
we started with. For this purpose we need a certain topology on $\barmoduli[g,n]$
in which we can verify convergence of the family.
\par
Let $(X, Z)$ be a pointed stable curve, where $Z = \{z_1, \ldots, z_n\}$ is the set of marked points on $X$,
and let $X'$ denote $X$ minus its nodes, so $X'$ is a disjoint union of irreducible components of $X$, punctured at the nodes.
An \emph{exhaustion} of $X$ is a sequence of subsets $\cdots K_{m-1}\subset K_m \subset K_{m+1}\cdots $ of $X'$, whose union is $X'$.  A sequence of pointed stable curves $(X_m, Z_m)$ in $\barmoduli[g,n]$ converges to $(X, Z)$ in the
\emph{quasiconformal topology}, if for some exhaustion $\{ K_m\}$ of $X$ there exists a
sequence of maps $f_m\colon K_m \to X_m$ that are quasiconformal onto the images,
such that $f_m$ respects the marked points and such that the \emph{dilatation} of $f_m$ tends to $1$.
The quasiconformal topology was introduced by Abikoff (\cite{Abikoff}) in
the setting of the augmented Teichm\"uller space. As $\barmoduli[g,n]$ is
a complex projective variety, it also inherits a standard topology. Hubbard and Koch (\cite{HubKoch}) established that these two topologies
are equivalent.
\par
For each node $q$ of the stable curve~$X$, let $D^+_{q}(t)$ be the closed disc of radius $t^{\ell(v^{+}(q))-\ell(v^{-}(q))-1/2}$ around $q^+$ on the higher level branch $(X_{v^+(q)}, \eta_{v^{+}(q)})$ of $q$ (modified in the neighborhood of the residue slit). In the gluing construction of $X_t$, the boundary of $D^+_{q}(t)$ becomes a (modified) circle of radius $t^{-1/2}$ on the lower level branch $(X_{v^-(q)}, \eta_{v^{-}(q)})$, which separates the basic domain presentation of the pole $q^-$ into two regions.
Let $D^{-}_{q}(t)$ be the closure of the region that contains $q^-$.
Denote by $\gamma_q(t)$ the common boundary curve of $D^{\pm}_{q}(t)$ in $X_t$ after the gluing construction.
Denote by $D^{\pm}(t)$ the union of $D^{\pm}_{q}(t)$ and by $\gamma(t)$ the union of $\gamma_q(t)$ over all nodes $q$.
\par
Let $K_t\coloneqq X \setminus D^{\pm}(t)$.
Then $\{ K_t \}$ is an exhaustion of $X$ as $t\to 0$. Moreover, $K_t$ looks just alike $X_t\setminus \gamma(t)$, and the only difference is that $X_t$ is further modified around the neighborhoods of the residue slits outside $D^+(t)$. This difference can be measured by the following result, which is the main step towards showing that the family of flat surfaces we have constructed actually converges to the prescribed limit object.
\par
\begin{lm}\label{lm:quasiconformal}
As $t \to 0$, there exists a quasiconformal map
$f_t \colon K_t \to X_t \setminus \gamma(t) $, respecting the marked points $z_i(t)$,
such that the dilatation of $f_t$ tends to~$1$.
\end{lm}
\par
For a simple example, consider a two-level surface with a torus on the top level joined to the lower level at
two nodes~$q_1$ and~$q_2$. Suppose the value of the lower level is $-1$ and the direction of the residues $\pm r$ at the two nodes
is vertical. Then the map~$f_t$ is given in Figure~\ref{cap:QCmapsft}.
Only the top level surface is depicted, since there is no residue slit on the
lower level. The residue slit in this case is just a straight
line joining~$q^+_1$ to~$q^+_2$, presented as the dotted line in the picture.

\begin{figure}[htb]
\begin{tikzpicture}
  \begin{scope}[scale=1.0]

    \fill[color=black!10!] (0,0) rectangle (5.5,3.6);
\draw(1.3,1.8) circle (0.8);
\fill[ color=white] (1.3,1.8) circle (0.8);
\draw(4.2,1.8) circle (0.8);
\fill[ color=white] (4.2,1.8) circle (0.8);
  \end{scope}

    \draw (0,0) -- (5.5,0) -- (5.5,3.6) -- (0,3.6) -- cycle;
    \draw (1.25,1.2) -- (4.25,1.2) -- (4.25,2.4) -- (1.25,2.4) -- cycle;
    \draw[thick,dotted] (1.25,1.8) -- (4.25,1.8);

        \draw[dashed] (2.75,1.8) -- (2.75,1.2);
        \draw[dashed] (2.75,1.8) -- (2.75,2.4);

  \begin{scope}
    \node at (1.3,2.76) {$D^{+}_{q_1}$};
     \end{scope}
  \begin{scope}[xshift=3cm]
    \node at (1.3,2.76) {$D^{+}_{q_2}$};
  \end{scope}
  \begin{scope}[xshift=1.6cm]
  \draw [decorate,decoration={brace}]
	(2.75,2.39) -- (2.75,1.81) node [midway, right] {$t^{1/2} \rho $};
\end{scope}
	
  \draw [->, bend angle=25, bend left]
        (5.75,1.8) to (7.25,1.8) node [midway, below] {};
  \node at (6.55,2.3) {$f_t$};
  \begin{scope}[scale=1.0,xshift=7.5cm]

\fill[color=black!10!] (0,0) rectangle (5.5,3.6);

\draw[ color=black] (1.3,1.8) circle (0.8);
\fill[ color=white] (1.3,1.8) circle (0.8);
\draw[ color=black] (4.2,1.8) circle (0.8);
\fill[ color=white] (4.2,1.8) circle (0.8);
\fill[color=white] (1.25,1.7) rectangle (4.25,1.9);
   \draw (0,0) -- (5.5,0) -- (5.5,3.6) -- (0,3.6) -- cycle;
    \draw (1.25,1.2) -- (4.25,1.2) -- (4.25,2.4) -- (1.25,2.4) -- cycle;
    \draw[thick,dotted] (1.25,1.8) -- (4.25,1.8);

    \draw (1.25,1.7) -- (4.25,1.7);
    \draw (1.25,1.9) -- (4.25,1.9);
    \draw [thick,dotted] (1.25,1.8) -- (4.25,1.8);

        \draw[dashed] (2.75,1.7) -- (2.75,1.2);
        \draw[dashed] (2.75,1.9) -- (2.75,2.4);

      \begin{scope}[xshift=2.1cm]
    \draw [decorate,decoration={brace}]
	  (2.2,2.39) -- (2.2,1.91) node [midway, right] {$t^{1/2}\rho -tr /2$};
	  \end{scope}	
  \end{scope}
\end{tikzpicture}
\caption{The quasiconformal map $f_t$ on the top level} \label{cap:QCmapsft}
\end{figure}
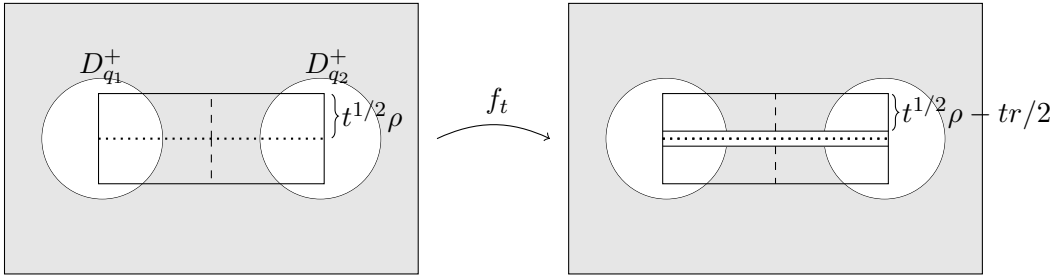
\par

The quasiconformal map $f_t$ in this case is the identity outside the union of the
two disks~$D^{+}_{q_1}(t)$, $D^{+}_{q_2}(t)$, and the rectangular region swept out by segments of holonomy
$t^{1/2} (2\rho )$, for some fixed $\rho < 1$, centered around the residue slit. We choose the scaling factor
$t^{1/2}$ because it is the radius of the disks $D^{+}_{q_i}(t)$, hence the rectangular region goes into the disks under the assumption on $\rho$.
Inside the rectangular region (and still outside of the disks), $f_t$ takes the upper (resp.\ lower) half
in the left picture to the upper (resp.\ lower) half in the right picture, by a family of
linear maps preserving vertical lines and shortening the height of the vertical segments by that of the removed residue slit neighborhood, with
a small perturbation near the boundary of $D^{+}_{q_{i}}(t)$. Note that $f_t$ is continuous along the dotted line on the left, because the upper and lower boundary segments of the residue slit neighborhood
are identified in the image surface~$X_t$ on the right. Moreover, since the vertical segments are shortened by at most $tr/2$, which becomes arbitrarily small compared to $t^{1/2}\rho$ as $t\to 0$, it implies that the dilatation of $f_t$ tends to $1$.
\par
\smallskip
\noindent\textit{Proof of Lemma~\ref{lm:quasiconformal}}.
The idea of the proof is similar to the example presented above. Let $b_{ij}$ be a broken line segment of the residue slit on the subsurface~$(X_v, \eta_v)$
that was constructed when gluing in a subsurface at level $L<\ell(v)$.
If $\langle \theta(b_{ij}), r_i \rangle > 0$, consider the neighborhood of~$b_{ij}$ swept out by
segments of holonomy $ t^{\ell(v)-L-1/2} (2\rho )$ centered around the residue slit. On the upper and lower parts of this neighborhood we
define the quasiconformal map~$f_t$ to be an
affine transformation onto the respective $( t^{\ell(v)-L-1/2} (\rho - t^{1/2}r_i/2))$-segments as illustrated in Figure~\ref{cap:QCmapsft}, and then
composed with scaling by~$t^{-\ell(v)}$, because $(X_v, \eta_v)$ is scaled by $t^{-\ell(v)}$ before gluing into $X_t$. If $\langle \theta(b_{ij}), r_i \rangle < 0$ we use the same composition, except that the image
of the affine transformation in the upper and lower parts will be swept out by the
$(t^{\ell(v)-L-1/2} (\rho + t^{1/2}r_i / 2) )$-segments. Since the ratio of $\rho$ and $\rho \pm t^{1/2} r_i / 2$ tends to $1$ as $t\to 0$, it follows that the dilatation of $f_t$
on these neighborhoods of the residue slit tends to $1$.
\par
Moreover, in the neighborhood of a corner where the residue
slit changes its direction from $\langle \theta(b_{ij}), r_i \rangle > 0$ to
$\langle \theta(b_{ij}), r_i\rangle< 0$, the quasiconformal map $f_t$ is
the identity composed with scaling by~$t^{-\ell(v)}$. To see this, e.g., take the angular sector between~$a_3$ and~$b$
in Figure~\ref{cap:ResSlitBack}, slide it up so that the sides~$a_3$ touch,
and glue in the additional parallelogram. One can easily find linear maps along segments
in the direction $r_i$ that interpolate continuously from the situation generically along
the broken line to the situation near such a corner.
\par

In the central polygon $P$, we extend the residue slit by connecting the middle points of the edges to the barycenter $B(P)$. We define the quasiconformal map $f_t$ in a neighborhood of this polygon by affine transformations that stretch along the lines joining the extended slit segments to the respective vertices of the polygon neighborhood
as in the figure to the right, and then composed with scaling by~$t^{-\ell(v)}$. Again, one can find a continuous family of linear maps along segments
whose direction varies from $r_i$ to $r_{i+1}$ near the central polygon. Such a map is sketched in Figure~\ref{fig:quasicentral}, where the dashed lines represent curves along which $f_{t}$ stretches.
\par

\begin{figure}[htb]
 \begin{tikzpicture}[description/.style={fill=white,inner sep=0pt},scale=0.75]
  \fill[ color=black!10!] (0,0) circle (2.5);

\begin{scope}[scale=.3]
  \coordinate[] (P1) at (-3,-1);
  \coordinate[] (P2) at (2,-1);

  \coordinate[] (P3) at (0.5,1.4);

  \coordinate (M1) at ($(P1)!.5!(P2)$);
    \coordinate (M2) at ($(P2)!.5!(P3)$);
  \coordinate (M3) at ($(P1)!.5!(P3)$);

  \coordinate[] (P4) at ($(0,0)!.5!(M1)$);
  \coordinate[] (P5) at ($(0,0)!.5!(M2)$);
  \coordinate[] (P6) at ($(0,0)!.5!(M3)$);

\begin{scope}[xshift=.5cm,yshift=1.4cm]
 \draw[rotate=125] (0,0) -- (0,-7.5) coordinate (F4);
 \draw[rotate=215] (0,0) -- (0,-7.5) coordinate (F5);
\end{scope}
\begin{scope}[xshift=2cm,yshift=-1cm]
 \draw (0,0) -- (0,-7.5) coordinate (F6);
 \draw[rotate=125] (0,0) -- (0,-7.5) coordinate (F7);
\end{scope}
\begin{scope}[xshift=-3cm,yshift=-1cm]
 \draw (0,0) -- (0,-7.5) coordinate (F8);
 \draw[rotate=215] (0,0) -- (0,-7.5) coordinate (F9);
\end{scope}

    \fill[color=white]
    (F5) -- (.5,1.4) --(F4)  -- (F7) -- (2,-1) -- (F6)-- (F8) -- (-3,-1) -- (F9) -- cycle;

 \draw (P1) -- (P2) -- (P3) -- (P1);
  \draw (0,0) -- (M1);
    \draw (0,0) -- (M2);
  \draw (0,0) -- (M3);

    \begin{scope}[shift=(M2)]
 \draw[rotate=125] (0,0) -- (0,-7.5) coordinate [pos=.35] (F1);
\end{scope}
\begin{scope}[shift=(M3)]
 \draw[rotate=215] (0,0) -- (0,-7.5) coordinate [pos=.3] (F2);
\end{scope}
\begin{scope}[shift=(M1)]
 \draw[] (0,0) -- (0,-7.5) coordinate [pos=.3] (F3);
\end{scope}
\end{scope}

\begin{scope}[xshift=.45cm,yshift=1.25cm]
 \draw[rotate=125] (0,0) -- (0,-1.5) coordinate [pos=.2](F4)coordinate [pos=.35](G4)coordinate [pos=.62](H4);
 \draw[rotate=215] (0,0) -- (0,-1.5) coordinate [pos=.2](F5)coordinate [pos=.35](G5)coordinate [pos=.62](H5);
\end{scope}
\begin{scope}[xshift=1.45cm,yshift=-.45cm]
 \draw (0,0) -- (0,-1.5) coordinate [pos=.2](F6)coordinate [pos=.4](G6)coordinate [pos=.7](H6);
 \draw[rotate=125] (0,0) -- (0,-1.5) coordinate [pos=.2](F7)coordinate [pos=.4](G7)coordinate [pos=.7](H7);
\end{scope}
\begin{scope}[xshift=-1.7cm,yshift=-.5cm]
 \draw (0,0) -- (0,-1.5) coordinate [pos=.2](F8)coordinate [pos=.4](G8)coordinate [pos=.7](H8);
 \draw[rotate=215] (0,0) -- (0,-1.5) coordinate [pos=.2](F9)coordinate [pos=.4](G9)coordinate [pos=.7](H9);
\end{scope}

\draw[dashed] (P3) -- (.45,1.25);
\draw[dashed] (P2) -- (1.45,-.45);
\draw[dashed] (P1) -- (-1.7,-.5);
\draw[dashed] (P1) -- (0,0);
\draw[dashed] (P2) -- (0,0);
\draw[dashed] (P3) -- (0,0);

\draw[dashed] (M1) -- (F8);
\draw[dashed] (M1) -- (F6);
\draw[dashed] (M2) -- (F4);
\draw[dashed] (M2) -- (F7);
\draw[dashed] (M3) -- (F5);
\draw[dashed] (M3) -- (F9);

\draw[dashed] (F1) -- (G4);
\draw[dashed] (F1) -- (G7);
\draw[dashed] (F2) -- (G5);
\draw[dashed] (F2) -- (G9);
\draw[dashed] (F3) -- (G6);
\draw[dashed] (F3) -- (G8);

\draw[dashed] (H4) -- (H7);
\draw[dashed] (H5) -- (H9);
\draw[dashed] (H6) -- (H8);
\end{tikzpicture}
\caption{The quasiconformal map $f_{t}$ nearby the central polygon}\label{fig:quasicentral}
\end{figure}
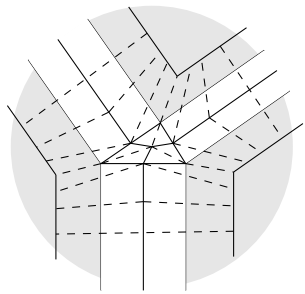

On the complement of these neighborhoods of the residue slits in each component $(X_v, \eta_v)$,
we take~$f_t$ to be the scaling map by $t^{-\ell(v)}$, which glues continuously to the quasiconformal map on the neighborhoods constructed above.
\par
It is clear from the construction, as explained above, that the dilatation
of~$f_t$ depends continuously on~$t$ and tends to $1$ as $t\to 0$.
\qed
\par
\smallskip
Lemma~\ref{lm:quasiconformal} proves the desired convergence in $\barmoduli[g,n]$. Nevertheless, it does not quite imply convergence of
pointed stable differentials in the incidence variety compactification, as explained in Remark~\ref{rem:noninj}.
In order to show that $(X_t, \omega_t)$ converges to $\eta_v$ projectively
on each component $X_v$ of $X$, it suffices to show that the location
and multiplicity of all zeros and poles agree. This is obvious for
all zeros and poles of~$\eta$ in the smooth locus of~$X$. It remains to check the vanishing orders of $\eta$ at the
nodes of~$X$ (compared to Example~\ref{exa:etanotunique}). At each node $q$, the vanishing order can be detected
by the index of a path centered around~$q^+$ under the flat metric in the higher level branch of $q$.
Since this quantity is defined
in a neighborhood of $q^+$, untouched in the construction for
$t$ small enough, the vanishing orders of the limit of
$\omega_t$ as $t \to 0$ are equal to those of $\eta$ at every node.
\par
To show that $(X_t, \omega_t)$ converges not only projectively, but also with the given scales at top level, it suffices to compare
ratios of the length of a nonzero (relative or absolute) period in each top
level component. This is obvious for a top level component of positive genus
or with more than one marked zero. In the remaining cases, such a component has at least two nodes $q_1$ and $q_2$ joining
it to lower level. A path on $X_t$ joining a boundary point
on each annulus $A_i^+$ around $q_i$ (that we used for gluing $X_t$) converges
to a path joining $q_1$ to $q_2$, and hence can be used for comparison
of relative size.
\qed

\end{document}